\documentclass[10pt,a4paper,english]{article}

\usepackage{setspace}
\usepackage{lipsum}

\usepackage{authblk}
\usepackage[style=alphabetic,sorting=nyt,backend=bibtex,firstinits=true]{biblatex}
\AtBeginBibliography{\setstretch{1}\footnotesize}

\usepackage{mathtools}
\usepackage{amsfonts,amssymb,stmaryrd,amsmath,amsthm,amssymb}
\usepackage{graphicx,color}
\usepackage{multimedia}
\usepackage{float}
\usepackage{placeins}

\usepackage{dsfont}

\usepackage[a4paper]{geometry}
\newtheorem{theorem}{Theorem}[section]
\newtheorem{lemma}{Lemma}[section]
\newtheorem*{mylemma}{Lemma}
\newtheorem{remark}{Remark}[section]
\newtheorem{proposition}{Proposition}[section]
\newtheorem{corollary}{Corollary}[section]
\newtheorem{definition}{Definition}[section]

\newcommand \Diff {\mathcal{C}_{\Gamma_s}}

\newcommand \p {\partial}

\newcommand \R {\mathbb{R}}

\renewcommand \L {\mathrm{L}}
\newcommand \W {\mathrm{W}}

\newcommand \HH {\mathbf{H}}
\newcommand \LL {\mathbf{L}}
\newcommand \LLL {\mathbb{L}}

\renewcommand \H {\mathrm{H}}

\newcommand \I {\mathrm{I}}
\newcommand \Id {\mathrm{Id}}

\renewcommand \d {\mathrm{d}}

\renewcommand \det {\mathrm{det}}

\newcommand \cof {\mathrm{cof}}

\DeclareMathOperator{\divg}{div}

\begingroup\makeatletter\ifx\SetFigFont\undefined%
\gdef\SetFigFont#1#2#3#4#5{%
  \reset@font\fontsize{#1}{#2pt}%
  \fontfamily{#3}\fontseries{#4}\fontshape{#5}%
  \selectfont}%
\fi\endgroup%

\title{Feedback stabilization of a two-fluid surface tension system modeling the motion of a soap bubble at low Reynolds~number:\\The two-dimensional case\thanks{This work is supported by...}}

\author[1, 2]{S\'ebastien Court}
\affil[1]{\begin{small}Department of Mathematics, University of Innsbruck, Technikerstrasse 13, 6020 Innsbruck, Austria.\end{small}}
\affil[2]{\begin{small}Digital Science Center, University of Innsbruck, Innrain 15, 6020 Innsbruck, Austria. Email: {\tt sebastien.court@uibk.ac.at}\end{small}}

\begin{document}
\maketitle

\begin{abstract}
The aim of this paper is to design a feedback operator for stabilizing in infinite time horizon a system modeling the interactions between a viscous incompressible fluid and the deformation of a soap bubble. The latter is represented by an interface separating a bounded domain of $\R^2$ into two connected parts filled with viscous incompressible fluids. The interface is a smooth perturbation of the 1-sphere, and the surrounding fluids satisfy the incompressible Stokes equations in time-dependent domains. The mean curvature of the surface defines a surface tension force which induces a jump of the normal trace of the Cauchy stress tensor. The response of the fluids is a velocity trace on the interface, governing the time evolution of the latter, via the equality of velocities. The data are assumed to be sufficiently small, in particular the initial perturbation, that is the initial shape of the soap bubble is close enough to a circle. The control function is a surface tension type force on the interface. We design it as the sum of two feedback operators: one is explicit, the second one is finite-dimensional. They enable us to define a control operator that stabilizes locally the soap bubble to a circle with an arbitrary exponential decay rate, up to translations, and up to non-contact with the outer boundary.
\end{abstract}

\hfill 

\tableofcontents

\section{Introduction}

In this paper we consider a model that describes the time evolution of an interface~$\Gamma(t)$ separating a bounded fluid domain~$\Omega \subset \R^2$ into two connected components. The fluid is assumed to be viscous and incompressible. This interface represents a soap bubble, and is subject to surface tension forces. Surface tension is the result of intermolecular forces in fluids (air or liquid, see~\cite{Probstein}). We consider the low Reynolds number case, that is the inertia forces are assumed to be negligible compared to the other involved forces (viscosity effects, surface tension, electric field, etc\dots), which is typically the case of blood flows~\cite{Blood2009} containing globules~\cite{Darling}. We will also neglect the temperature effects. Our aim is to stabilize the deformations of this soap bubble such that this latter converges to a circle, with a prescribed exponential decay rate, via the design of a feedback operator acting on the interface. Our contribution falls within the modeling and mathematical aspects of the motion of bubbles out of equilibrium.

\subsection{The model}
Due to incompressibility, the volume contained inside the soap bubble remains constant, and any sphere of this volume, strictly contained in~$\Omega$, can be considered as the sphere of reference~$\Gamma_s$. Denoting by $X(\cdot,t)$ the Lagrangian deformation of this sphere, the deformed soap bubble is given by~$\Gamma(t) = X(\Gamma_s,t)$, with the initial condition $X_0 = X(\cdot,0)$, and separates the whole domain~$\Omega$ into two subdomains~$\Omega^+(t)$ and~$\Omega^-(t)$, corresponding to the exterior and the interior of the soap bubble, respectively. The geometric description and notation are given in Figure~\ref{fig-deformation}.

\begin{minipage}{\linewidth}
\vspace*{1cm}
\begin{center}
\scalebox{0.45}{
\begin{picture}(0,0)
\includegraphics{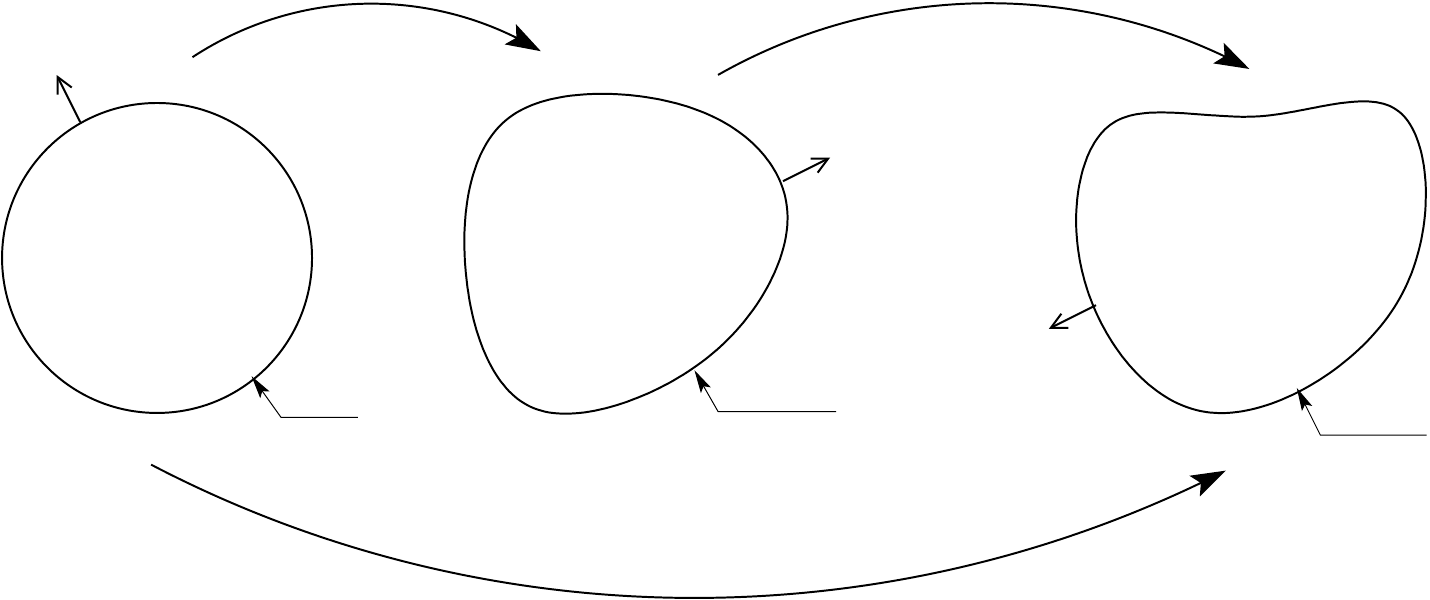}
\end{picture}
\setlength{\unitlength}{4144sp}
\begin{picture}(10887,4564)(1775,-5187)
\put(7600,-1700){\makebox(0,0)[lb]{\smash{{\SetFigFont{18}{20.4}{\rmdefault}{\mddefault}{\updefault}{\color[rgb]{0,0,0}$n_0$}%
}}}}
\put(2380,-1321){\makebox(0,0)[lb]{\smash{{\SetFigFont{18}{20.4}{\rmdefault}{\mddefault}{\updefault}{\color[rgb]{0,0,0}$n_s$}%
}}}}
\put(2700,-2621){\makebox(0,0)[lb]{\smash{{\SetFigFont{28}{20.4}{\rmdefault}{\mddefault}{\updefault}{\color[rgb]{0,0,0}$\Omega_s^-$}%
}}}}
\put(900,-3021){\makebox(0,0)[lb]{\smash{{\SetFigFont{28}{20.4}{\rmdefault}{\mddefault}{\updefault}{\color[rgb]{0,0,0}$\Omega_s^+$}%
}}}}
\put(9670,-2870){\makebox(0,0)[lb]{\smash{{\SetFigFont{18}{20.4}{\rmdefault}{\mddefault}{\updefault}{\color[rgb]{0,0,0}$n$}%
}}}}
\put(7336,-3616){\makebox(0,0)[lb]{\smash{{\SetFigFont{18}{20.4}{\rmdefault}{\mddefault}{\updefault}{\color[rgb]{0,0,0}$X_0(\Gamma_s)$}%
}}}}
\put(6500,-5000){\makebox(0,0)[lb]{\smash{{\SetFigFont{22}{24.4}{\rmdefault}{\mddefault}{\updefault}{\color[rgb]{0,0,0}$X(\cdot,t)$}%
}}}}
\put(4200,-500){\makebox(0,0)[lb]{\smash{{\SetFigFont{22}{24.4}{\rmdefault}{\mddefault}{\updefault}{\color[rgb]{0,0,0}$X_0$}%
}}}}
\put(8400,-500){\makebox(0,0)[lb]{\smash{{\SetFigFont{22}{24.4}{\rmdefault}{\mddefault}{\updefault}{\color[rgb]{0,0,0}$X(\cdot,t) \circ X_0^{-1}$}%
}}}}
\put(4000,-3700){\makebox(0,0)[lb]{\smash{{\SetFigFont{18}{20.4}{\rmdefault}{\mddefault}{\updefault}{\color[rgb]{0,0,0}$\Gamma_s$}%
}}}}
\put(12000,-3800){\makebox(0,0)[lb]{\smash{{\SetFigFont{18}{20.4}{\rmdefault}{\mddefault}{\updefault}{\color[rgb]{0,0,0}$\Gamma(t)$}%
}}}}
\put(10800,-2650){\makebox(0,0)[lb]{\smash{{\SetFigFont{28}{20.4}{\rmdefault}{\mddefault}{\updefault}{\color[rgb]{0,0,0}$\Omega^-(t)$}%
}}}}
\put(12800,-2650){\makebox(0,0)[lb]{\smash{{\SetFigFont{28}{20.4}{\rmdefault}{\mddefault}{\updefault}{\color[rgb]{0,0,0}$\Omega^+(t)$}%
}}}}
\end{picture}
}
\end{center}
\vspace*{-0.5cm}
\begin{figure}[H]
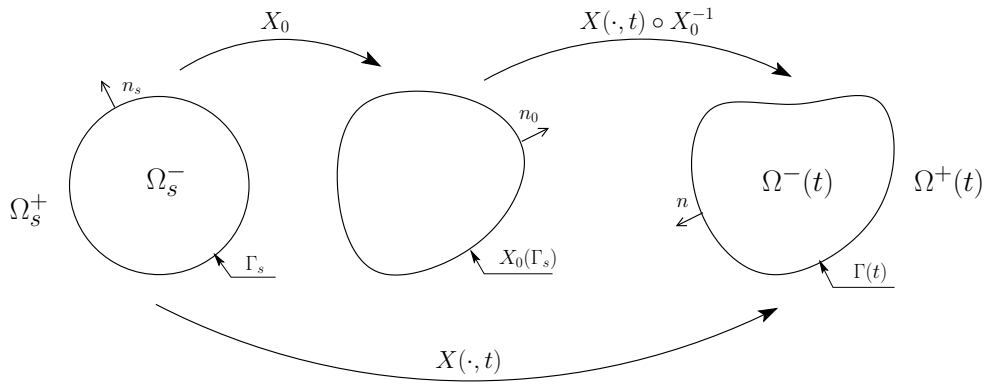

\caption{Deformation of the soap bubble, from the reference configuration $\Gamma_s$, to $\Gamma(t)$ at time $t$. \label{fig-deformation}}
\end{figure}
\end{minipage}
\FloatBarrier
\hfill \\
We will assume throughout the paper that $\Gamma(t)$ is a Jordan curve, that~$X(\cdot,t)$ is invertible and orientation preserving, and the {\it non-contact} condition, that is $\Gamma(t) \subset \mathring{\Omega}$. This can be guaranteed by assuming the data smooth and sufficiently small, in particular $X_0$ close enough to the identity, implying that $X(\cdot,t)$ stays close to the identity as well. With this condition, we then have $\Gamma(t) = \p \Omega^{-}(t)$ and $\p \Omega^+(t) = \p \Omega \cup \Gamma(t)$ (disjoint union), and $\Omega^+(t)$ is connected. The deformation $X$ and the velocity/pressure couples $(u^+,p^+)$, $(u^-,p^-)$ constitute the unknowns of the following system
\begin{eqnarray*}
- \divg \sigma(u^+,p^+) = f^+ & & \text{in } \Omega^+(t), \ t \in (0,\infty), \\
- \divg \sigma(u^-,p^-) = f^- & & \text{in } \Omega^-(t), \ t \in (0,\infty), \\
 \divg u^+ = 0 & & \text{in } \Omega^+(t), \ t \in (0,\infty), \\
 \divg u^- = 0 & & \text{in } \Omega^-(t), \ t \in (0,\infty), \\
u^+ = 0 & & \text{on } \p \Omega \times (0,\infty), \\
u^+ = u^- = \frac{\p X}{\p t} \left( X(\cdot,t)^{-1},t\right), & & 
\text{on } \Gamma(t), \ t \in (0,\infty), \\
\sigma(u^+,p^+)n^+ + \sigma(u^-,p^-)n^- =  \mu\kappa n^- + g & &
\text{on } \Gamma(t), t \in (0,\infty), \\
X(\cdot,0) = X_0 & & \text{on } \Gamma_s,
\end{eqnarray*}
where $\Gamma(t) = X(\Gamma_s,t)$ splits $\Omega$ into $\Omega^+(t)$ and $\Omega^-(t)$.  We adopt the Lagrangian formalism for describing the interface $\Gamma(t)$: Particles of coordinates $x\in \Gamma(t)$ are obtained uniquely from particles $y \in \Gamma_s$ as $x = X(y,t)$, and their velocity writes $\displaystyle \frac{\p X}{\p t}(y,t)$. On the other hand, the Eulerian velocities $u^+$ or $u^-$ describe the velocity field of particles occupying position $x \in \Omega^+(t)$ or $\Omega^-(t)$ at time $t$. Therefore from the equality of particle velocities on $\Gamma(t)$ we have the relation $\displaystyle \frac{\p X}{\p t}(y,t) = u^+(X(y,t),t) = u^-(X(y,t),t)$, leading to the equality of Eulerian velocities on $\Gamma(t)$ above, by using $y = X(x,t)^{-1}$. The pressure variables $p^+$ and $p^-$ play the role of Lagrange multipliers for the zero divergence conditions, referring to the incompressibility of the fluid. We have introduced the Cauchy stress tensor~$\sigma(u,p) := 2\nu \varepsilon(u) -p \Id$, where $\varepsilon(u) := \mathrm{Sym}(\nabla u) = \frac{1}{2}(\nabla u + \nabla u^T)$, and the viscosity~$\nu>0$ is constant and assumed to be the same in $\Omega^+(t)$ and~$\Omega^-(t)$, for the sake of simplicity, but without loss of generality. We denoted by~$n^+$ and~$n^-$ the outward unit normal of $\Omega^+(t)$ and $\Omega^-(t)$, respectively. By default we denote $n=n^- = -n^+$. The parameter $\mu >0$ is a given constant surface tension coefficient, and $\kappa$ denotes the mean curvature of~$\Gamma(t)$, with the convention~$\kappa <0$ for the 1-sphere. In dimension~2, the mean curvature is simply called the {\it curvature}. The right-hand-sides~$f^+$ and $f^-$ are given volume forces, representing for example the effect of an electric field~\cite{Zografov2014}. The function~$g$ is a surface tension type force, acting on the interface~$\Gamma(t)$. It will be considered as the control function, and will be chosen in the form of a feedback operator. For the sake of concision, we rewrite the system above as follows:
\begin{subequations} \label{mainsys}
\begin{eqnarray}
- \divg \sigma(u^\pm,p^\pm) = f^\pm \quad \text{ and } \quad \divg u^\pm = 0 & & \text{in } \Omega^\pm(t), \ t \in (0,\infty), \\
u^+ = 0 & &  \text{on } \p \Omega \times (0,\infty), \\
u^\pm = \frac{\p X}{\p t}\left( X(\cdot,t)^{-1},t\right) 
\quad \text{ and } \quad
-\left[ \sigma(u,p)\right]n = \mu \kappa n + g & & \text{on } \Gamma(t), \ t \in (0,\infty), \label{mainsys-jump}\\
X(\cdot,0) = X_0 & & \text{on } \Gamma_s. 
\end{eqnarray}
\end{subequations}
We have denoted by $[\varphi] = \varphi^+ - \varphi^-$ the jump across $\Gamma(t)$ of a vector/matrix field, and by $\varphi^\pm$ when we consider $\varphi^+$ and $\varphi^-$ separately, simultaneously, and respectively. In~\eqref{mainsys-jump}, the control function $g$ and the surface tension force $\mu\kappa n$ induces the jump of $\sigma(u,p)n$ across $\Gamma(t)$, and the response of the surrounding fluid is the trace of the velocity field on the interface $\Gamma(t)$, governing the time evolution of the latter via the time derivative of $X(\cdot,t)$. Recall that mapping~$X(\cdot,t)$ determines the interface $\Gamma(t)$, and therefore also the domains~$\Omega^\pm(t)$ in which the Stokes equations are set, as well as the mean curvature that satisfies the relation $\kappa n = \Delta_{\Gamma(t)} \Id$, involving the Laplace-Beltrami operator. Thus system~\eqref{mainsys} couples in a nonlinear manner the geometry given by $X(\cdot,t)$ and the state variables of the fluid, namely~$(u^\pm,p^\pm)$.\\
Existence of solutions for such a model in the context of two fluids separated by a closed free interface has been studied by Denisova and Solonnikov~\cite{Denisova1994, Denisova1995}, in line with prior contributions~\cite{Denisova1989, Denisova1990, Solonnikov1991, Denisova1991, Denisova1993}, based on the work of Rivkind~\cite{Rivkind1973, Rivkind1976, Rivkind1977, Rivkind1979}. More recently Pr\"{u}ss and Simonett revisited wellposedness questions in the context of the $L^p$-maximal regularity~\cite{Simonett2009, Simonett2010, Simonett2011}, addressing also the case of phase transitions~\cite{Simonett2012, Simonett2013, Simonett2016}. Modeling aspects, leading to the transmission equations on $\Gamma(t)$, were first introduced in~\cite{Rivkind1971}, as far as we know. Addressing advanced wellposedness questions for systems of type~\eqref{mainsys} can be a difficult task, as for example global existence of solutions for the Navier-Stokes model is still an open problem (see for example~\cite{Fischer2020}). We refer to~\cite{Denisova2021} for an exploration in this direction. In the present article we will only be interested in the stabilizability question, and will address wellposedness only for the linearized system as well as for the feedback-control-stabilized nonlinear system~\eqref{mainsys}.

\subsection{Main result}

Note that the volume enclosed by $\Gamma(t)$ is constant and prescribed by the one enclosed by the reference circle~$\Gamma_s$, due to incompressibility. Since any circle of the same volume, strictly included inside $\Omega$, is a stationary state (see Lemma~\ref{lemma-Temam}), we can decide to stabilize system~\eqref{mainsys} around any of these circles, and thus stabilization {\it around the sphere} is understood up to elements of the following space
\begin{equation*}
\Diff =  \left\{
X_c \in \HH^2(\Gamma_s)
\text{ such that }
X_c(\Gamma_s) \subset  \mathring{\Omega} 
\text{ is a circle of the same radius as }\Gamma_s
\right\}.
\end{equation*}
The initial configuration $\Gamma_0$ is represented via the deformation $X_0$ of $\Gamma_s$, so that $\Gamma_0 = X_0(\Gamma_s)$. Since we need smallness assumptions on the data, in particular $X_0$, in practice we can choose $X_c\in\Diff$ such that $\|X_0-X_c\|_{\HH^2(\Gamma_s)}$ is minimal, for example by restricting~$\Diff$ to translations. We refer to Remark~\ref{rk-admi} in section~\ref{sec-notation2} for further comments. This amounts to say that~$\Gamma_0$ is {\it close} to a circle of the same radius as~$\Gamma_s$. 
The main result of the present article is Theorem~\ref{th-main}:
\begin{theorem} \label{th-main}
Choose $X_c \in \Diff$. Let be $X_0 \in \HH^2(\Gamma_s)/\R^2$ and $f^\pm = f_{|\Omega^\pm(t)}$ such that $f\in \L^2(0,\infty;\LL^2(\Omega))$. For all $\lambda >0$, there exists a finite-dimensional linear operator $\mathcal{K}_{\lambda}$, depending only on~$\lambda$, $\Gamma_s$, $\nu$ and $\mu$, such that if the quantities $\|X_0-X_c\|_{\mathbf{H}^2(\Gamma_s)}$ and $\|e^{\lambda t}f\|_{\L^2(0,\infty;\LL^2(\Omega))}$ are small enough, then the solution of system~\eqref{mainsys} with
\begin{equation*}
g =  \left(r(\det \mathfrak{g})^{-1/2} \left(
\divg_{\Gamma_s}\big((\tau_s \otimes \tau_s) \nabla_{\Gamma_s}(X-X_c) \big)+
\mathcal{K}_{\lambda}\nabla_{\Gamma_s}(X-X_c)\right)\right) \circ X^{-1}
\end{equation*}
satisfies
\begin{equation*}
\left\| e^{\lambda t} (X-X_c) \right\|_{\L^2(0,\infty;\HH^{5/2}(\Gamma_s)/\R^2)\cap \H^1(0,\infty; \HH^{3/2}(\Gamma_s)/\R^2)} \leq C\|X_0-X_c\|_{\mathbf{H}^2(\Gamma_s)},
\end{equation*}
where $r$ denotes the radius of the circle~$\Gamma_s$, $\tau_s$ denotes the tangent vector of $\Gamma_s$, $C>0$ is a constant, and $\mathfrak{g}$ denotes the metric tensor of $\Gamma(t)$.
\end{theorem}
This result implies the asymptotic convergence towards a stationary state corresponding to an immobile circle, up to a translation in~$\R^2$. Note that in the choice of $g$, we need two different operators, one explicit, namely $\divg_{\Gamma_s}\big((\tau_s \otimes \tau_s) \nabla_{\Gamma_s}(X-X_c) \big)$, and one of finite dimension. Both deal with the tangential derivative of $X-X_c$. This feedback stabilization result is in line with many others obtained for other fluid-structure models, for example~\cite{JPR2010, CourtEECT1, CourtEECT2}. But as far as we know, there are very few mathematical contributions that address control-related questions for models involving surface tension forces: The models considered in~\cite{Carlos2013, Antil2015, Alazard2017, Gancedo2020, Cui2021} deal with free boundary problems, that do not involve jump conditions like in~\eqref{mainsys-jump}. However, other non-mathematical references explain the practical realization of surface tension controls on a small scale~\cite{Encyclo, Rana2012}. 

\subsection{Method}

Like in~\cite{JPR2010, CourtEECT1, CourtEECT2}, our method is based on the feedback stabilization of the linearized system. For deriving the latter, we first need to rewrite system~\eqref{mainsys} in cylindrical domains, in order to uncouple the fluid domains and the state variables, in particular the deformation~$X$. Since this surface deformation is initially defined on~$\Gamma_s$ only, we need to define a suitable extension to the whole domain, which leads us to study the non-trivial question of extension of diffeomorphisms from boundaries. This is realized in section~\ref{sec-extension1}, and the corresponding proof relies on recent results on the harmonic extensions of diffeomorphisms (see Appendix~\ref{appendix1}).

Concerning the model, for the sake of simplicity, in the fluid domain we choose to consider the linear stationary Stokes system, which corresponds to a low Reynolds number fluid. Considering fluid models at average or high Reynolds number would introduce other difficulties. For example, in~\cite{Shkoller2007, Shkoller2013} the authors consider the Euler system, and existence of solutions is obtained with high-order energy estimates, without any operator formulation that could be used for designing a possible feedback operator. On the other hand, addressing the Navier-Stokes system consists classically in deriving an operator formulation of the linearized system, involving the non-stationary Stokes system, for which a {\it lifting method} is used. For the present model, based on jump conditions, it is not clear how a lifting method would enable us to derive an operator formulation. Still in the case of the non-stationary Stokes equations with transmission conditions, the respective authors of~\cite{Simonett2010} and~\cite{Denisova1994} first reduced the interface to a straight line, and derived existence and uniqueness results via pseudo-differential calculus techniques, which lead to a space-time operator for describing the solution. In practice, such approaches can not be used in a infinite-horizon stabilization problem, moreover dealing with general surfaces. Further, the unique continuation argument used in our case for obtaining approximate controllability and stabilizability of the linearized system does not apply to the case of the non-stationary Stokes system, as the latter would necessitate taking into consideration zeros of spherical harmonics, leading to difficulties that go beyond the scope of this article.

Therefore we adopt the stationary Stokes equations, as we prefer to focus on the time-evolution of the interface displacement, more specifically an operator formulation that involves the time derivative of the interface displacement only. For the linearized system, involving the Stokes system in time-independent domains, we describe the solution via a stationary Poincar\'e-Steklov operator denoted by~$\mathcal{P}_{\Gamma_s}$, of Neumann-to-Dirichlet type, mapping the different transmission conditions. The existence of this operator is obtained via the Ladyzhenskaya-Babu\v{s}ka-Brezzi condition. Next the operator formulation for time-evolution of the interface displacement is obtained, involving $\mathcal{P}_{\Gamma_s}$, as follows
\begin{equation*}
\frac{\p Z}{\p t}  - \mu \mathcal{P}_{\Gamma_s}(\divg_{\Gamma_s}
\nabla^{n_s}_{\Gamma_s} Z) = \mathcal{P}_{\Gamma_s}G  \text{ on } \Gamma_s \times (0,\infty), \quad Z(\cdot,0) = X_0-\Id \text{ on } \Gamma_s,
\end{equation*}
where $Z = X-\Id$ represents the displacement of the interface, $\nabla^{n_s}_{\Gamma_s} := (n_s \otimes n_s)\nabla_{\Gamma_s}$, and $G$ represents a control function to be chosen in a feedback form. Since $\nabla^{n_s}_{\Gamma_s}$ is not coercive (see section~\ref{sec-kernel}), we define a first feedback operator as $\mu \divg_{\Gamma_s} \nabla^{\tau_s}_{\Gamma_s} Z$, where $\nabla^{\tau_s}_{\Gamma_s}Z :=  (\tau_s \otimes \tau_s)\nabla_{\Gamma_s} Z$, so that $Z \mapsto -\mu \divg_{\Gamma_s} \nabla^{n_s}_{\Gamma_s}Z - \mu \divg_{\Gamma_s} \nabla^{\tau_s}_{\Gamma_s} Z = -\Delta_{\Gamma_s} Z$ is coercive. This operator is explicit, and supported on the tangent to~$\Gamma_s$, which is compatible with possible practical realization. The resulting operator generates an analytic semi-group of contraction, with compact resolvent. Next we prove approximate controllability for the linear system via a unique continuation argument. Since the spectrum of the operator is discrete, the number of unstable modes is of finite number, and thus we can reduce the problem to a finite-dimensional control problem where approximate controllability implies stailizability by a finite-dimensional feedback operator satisfying a Riccati equation. This feedback operator is re-used for defining another control function that stabilizes locally the nonlinear system, via a fixed-point argument, provided that the perturbations of the steady state are sufficiently small.

One of the main reasons why we restricted our study to dimension~2 is due to theoretical difficulties that arise in dimension~3 when trying to extend diffeomorphisms. Another reason lies in the fact that the linear evolution equation above, involving the operator~$\nabla^{n_s}_{\Gamma_s}$, would be {\it a priori} more complex in dimension~3. All these reasons are explained in section~\ref{sec-dim3}. Nevertheless, the methodology adopted in the present article for designing in practice the feedback operator is still valid in dimension~3. Throughout the paper we try to keep as much as possible a formalism that is not restricted to dimension~2.

The paper is organized as follows: Notation and functional spaces are defined in~section~\ref{sec-notation}. Comments and important properties of the model are described in section~\ref{sec-notation2}. A change of variable is introduced in section~\ref{sec-extension1}, enabling us to rewrite in section~\ref{sec-extension2} the main system in time-independent domains. Section~\ref{sec-linear} is devoted to the study of the corresponding linearized system, where in particular in section~\ref{sec-Steklov} we define operator $\mathcal{P}_{\Gamma_s}$, leading in section~\ref{sec-linSH} to an operator formulation. In section~\ref{sec-feedback} we design a linear feedback operator that stabilizes the linear system. In section~\ref{sec-nonlinear} we deduce another feedback operator that stabilizes the nonlinear system, and thus prove the main result. Questions related with the extension to dimension~3 of the present work are posed in section~\ref{sec-dim3}. Finally, technical proofs of intermediate results are given in the Appendix.

\section*{Acknowledgments}
Professor Jean-Pierre Raymond is warmly thanked for having introduced the problem, when the author was a PhD-student at the Institute of Mathematics of Toulouse.

\section{Functional setting and preliminaries} \label{sec-notation}

\subsection{Function spaces and notation}
Denote by $\L^2$, $\H^s$ and $\W^{k,p}$ the standard Lebesgue/Sobolev spaces of real-valued functions, and their multi-dimensional versions for $d=2$ as follows:
\begin{equation*}
\LL^2(\Omega_s^\pm) = [\L^2(\Omega_s^\pm)]^d, \quad 
\LLL^2(\Omega_s^\pm) = [\L^2(\Omega_s^\pm)]^{d\times d}, \quad
\mathds{L}^2(\Gamma_s) = [\L^2(\Gamma_s)]^{d\times (d-1)}.
\end{equation*}
The notation $\mathds{L}^2(\Gamma_s)$ applies when considering for example tangential gradients on~$\Gamma_s$. Naturally we transpose the same type of notation for other types of spaces and domains. Recall the notation~$\HH^{-1/2}(\Gamma_s) = \HH^{1/2}(\Gamma_s)'$. For matrix fields~$A,\ B$ of $\R^{2\times 2}$ we recall the inner product $A:B = \mathrm{trace}(A^TB)$ and the corresponding Euclidean norm satisfies $|AB|_{\R^{2\times 2}} \leq |A|_{\R^{2\times 2}}|B|_{\R^{2\times 2}}$. We will denote $\mathrm{Sym}(A) = \frac{1}{2}(A+A^T)$. Denote by~$\cof(A)$ the cofactor matrix of any matrix field~$A$, and note that in dimension~2 the mapping $A\mapsto \cof(A)$ is linear. Given two vectors $a$ and $b$ of $\R^2$, the tensor product $a\otimes b$ denotes the matrix of $\R^{2\times 2}$ defined by $(a\otimes b)_{ij} = a_ib_j$.\\

For $0< T \leq \infty$, the displacements $Z=X-\Id$ of $\Gamma_s$ will be considered in the following space
\begin{equation*}
\mathcal{Z}_{T}(\Gamma_s) := \L^2(0,T; \HH^{5/2}(\Gamma_s)/\R^2) \cap
\H^1(0,T; \mathbf{H}^{3/2}(\Gamma_s)/\R^2),
\end{equation*}
where quotient spaces have been introduced for considering displacements $Z$ up to a constant of $\R^2$. This is equivalent to consider deformations $X = Z+\Id$ up to translations of~$\R^2$. Note that the translations of $\R^2$ are also elements of the space~$\Diff$. For any Banach space $B$ and subset $I \subset B$, we define $\displaystyle
\|Z\|_{B/I} = \inf_{Z_I \in I}\|Z -Z_I\|_B$. 
Subsequently, because of~\eqref{mainsys-jump}, we will consider the velocity/pressure variables in the respective spaces
\begin{equation*}
\begin{array} {l}
\mathcal{U}_{T}(\Omega_s^+) := \left\{u\in \L^2(0,T; \HH^2(\Omega_s^\pm))\mid 
u_{| \p \Omega} = 0\right\}, \\
\mathcal{U}_{T}(\Omega_s^-) := \L^2(0,T; \HH^2(\Omega_s^-)),
\quad
\mathcal{Q}_{T}(\Omega_s^\pm) := \L^2(0,T; \H^1(\Omega_s^\pm)/\R).
\end{array}
\end{equation*}
We endow the spaces~$\mathcal{U}_{T}(\Omega_s^\pm)$ with the classical norms, and $\mathcal{Q}_{T}(\Omega_s^\pm)$ with $\|p^\pm \|_{\mathcal{Q}_{T}(\Omega_s^\pm)} := \| \nabla p^\pm\|_{\L^2(0,T;\LL^2(\Omega_s^\pm))}$. Note that the pressures in $\mathcal{Q}_{T}(\Omega_s^\pm)$ are determined up to a constant. Actually these constants are the residual static pressures $p_s^\pm$ corresponding to the stationary state. We refer to Lemma~\ref{lemma-Temam} for more details. Still for $0<T\leq \infty$, the data will be considered in the following spaces
\begin{equation*}
\mathcal{F}_{T}(\Omega_s^\pm) := \L^2(0,T; \LL^2(\Omega_s^\pm)), 
\quad
\mathcal{G}_{T}(\Gamma_s) := \L^2(0,T; \HH^{1/2}(\Gamma_s)),
\quad
\mathcal{Z}_0(\Gamma_s) := \HH^2(\Gamma_s)/\R^2.
\end{equation*}
Note that $\mathcal{Z}_0(\Gamma_s)$ is the trace space of $\mathcal{Z}_{T}(\Gamma_s)$, and we recall that the following continuous embedding holds:
\begin{equation*}
\mathcal{Z}_{\infty}(\Gamma_s) \hookrightarrow \mathcal{C}_b([0,\infty); \HH^2(\Gamma_s)).
\end{equation*}
The interest of this regularity framework is that we can define extensions $\tilde{X}$ of mappings $X$ that are of continuous in time with values in $\HH^{5/2}(\Omega_s^\pm)\hookrightarrow \mathcal{C}^1(\overline{\Omega_s^\pm})$. More specifically, extensions $\tilde{X}$ of $X$ will be considered in the following space:
\begin{equation*}
\mathcal{X}_{\infty}(\Omega_s^\pm) :=
\L^{\infty}(0,\infty; \HH^{5/2}(\Omega_s^\pm)).
\end{equation*}
Besides, the extensions $\tilde{X}$ are such that $\nabla \tilde X \in \L^{\infty}(0,\infty;\mathbb{H}^{3/2}(\Omega_s^\pm))$. The same property holds for the inverse of~$\nabla \tilde{X}$ (Corollary~\ref{coro-extension}), which is convenient for deriving Lipschitz estimates when stabilizing the nonlinear system in section~\ref{sec-nonlinear}, as the space $\mathbb{H}^{3/2}(\Omega_s^\pm)$ is an algebra (see~\cite[Proposition~B.1, page~283]{Grubb}). In the same fashion, the space~$\mathbb{H}^1(\Gamma_s)$ is also an algebra.\\

Recall the Petree-tartar lemma~\cite[Lemma A.38 page 469]{Ern}, that we will use several times.

\begin{mylemma}[Petree-Tartar lemma] \label{lemma-Petree}
Let $B_1$, $B_2$ and $B_3$ be Banach spaces. Let $A \in \mathcal{L}(B_1,B_2)$ be an injective operator, and let $C \in \mathcal{L}(B_1,B_3)$ be a compact operator. Assume that there exists a positive constant\footnote{Throughout the paper the notation $C$ refers to a positive constant generically independent of the different variables.} $C>0$ such that for all $\varphi \in B_1$ we have
\begin{equation*}
\|\varphi \|_{B_1} \leq C\left(\|A\varphi\|_{B_2}+\|C\varphi\|_{B_3}\right).
\end{equation*}
Then there exists $C>0$ such that
\begin{equation*}
\|\varphi \|_{B_1} \leq C\|A\varphi\|_{B_2}
\end{equation*}
for all~$\varphi \in B_1$.
\end{mylemma}

\subsection{On the surface tension model and the stationary states} \label{sec-notation2}

The surface tension force is generated by the mean curvature vector of the surface~$\Gamma(t)$. It is related to the Laplace-Beltrami operator via the following relation (see~\cite[p.~151, Exercise~2]{Willmore}):
\begin{equation*}
\kappa n  = \Delta_{\Gamma(t)} \Id. 
\end{equation*}
Using~\cite[Theorem~2.6]{Abels}, the following energy estimate holds, showing that the surface tension force derives from a potential energy quantified by the area of~$\Gamma(t)$:
\begin{equation*}
\mu \frac{\d}{\d t}|\Gamma(t) | + 2\nu \left( 
\|\varepsilon(u^+) \|_{\LLL^2(\Omega^+(t))}^2 + 
\|\varepsilon(u^-) \|_{\LLL^2(\Omega^-(t))}^2\right) = 
\langle f^+ , u^+\rangle_{\LL^2(\Omega^+(t))}
+ \langle f^- , u^-\rangle_{\LL^2(\Omega^-(t))}
+ \langle g , u^\pm\rangle_{\LL^2(\Gamma(t))}.
\end{equation*}
This energy is dissipated with the help of the viscosity terms. From~\cite[page~18]{Temam}, the kernel of~$\varepsilon$ is reduced to the tangent space of the special Euclidean group~$SE(2)$, namely the functions of type $u^\pm(x) = h^\pm + \omega^\pm x^\perp$, where $h^\pm \in \R^2$ and $\omega^\pm \in \R$ are constant. Using $u^+ = 0$ on $\p \Omega$, we deduce $h^+ = 0$ and $\omega^+ = 0$, which also implies $h^- = 0$ and $\omega^- =0$ when we have $u^+ = u^-$ on~$\Gamma(t)$. Therefore, the first Korn's inequality combined with the Rellich-Kondrachov theorem and the Petree-Tartar lemma yields the following general result:
\setcounter{lemma}{-1}
\begin{lemma} \label{lemma-Korn}
Let $\Omega^\pm$ be any smooth subdomains of $\Omega$ split by a closed smooth curve $\Gamma = \p \Omega^-$ such that $\Gamma \subset \mathring{\Omega}$. Then, for all $u^\pm \in \HH^1(\Omega^\pm)$ such that $u^+ = 0$ on $\p \Omega$ and $u^+ = u^-$ on $\Gamma$, we have
\begin{equation*}
\| u^+\|_{\HH^1(\Omega^+)} + \|u^- \|_{\HH^1(\Omega^-)} \leq C \left(
\| \varepsilon(u^+) \|_{\LLL^2(\Omega^+)}^2 + \| \varepsilon(u^-) \|_{\LLL^2(\Omega^-)}^2
\right),
\end{equation*}
where $C>0$ is independent of~$u^\pm$.
\end{lemma}
Since our approach is based on the study of the linearized system that involves $\divg_{\Gamma_s} \nabla^{n_s}_{\Gamma_s}(X-\I)$ (see section~\ref{sec-extension2}), we will rather focus on the different differential operators rather than on the mean curvature. The curve~$\Gamma_s$ is considered as a Riemannian manifold, and due to its regularity, we claim that the trace spaces~$\HH^s(\Gamma_s)$ coincide with the definition of Sobolev spaces given on Riemannian manifolds~\cite[section~2.2]{Hebey}. For a circle~$\Gamma_s \subset \mathring{\Omega}$ of given radius~$r$ we adopt the parameterization by arc length
\begin{equation*}
X_s: [0,2\pi r) \ni \xi \mapsto \left(
r\cos (\xi/r) , r\sin (\xi/r)
\right)^T \in \Gamma_s \subset\R^2.
\end{equation*}
Denoting~$\tau_s$ the tangent vector of~$\Gamma_s$, we have
\begin{equation*}
n_s \circ X_s = \left(
\cos (\xi/r) , \sin (\xi/r)
\right)^T, \quad
\tau_s \circ X_s = \left(
\sin (\xi/r) , -\cos (\xi/r)
\right)^T
\end{equation*}
Recall the notation for the tangential gradient $\nabla_{\Gamma_s}$ and the tangential divergence~$\divg_{\Gamma_s}$. In the particular case of a circle, with the parameterization~$X_s$ chosen above, the metric tensor~$\mathfrak{g}_s$ of $\Gamma_s$ is scalar-valued, equal to $|\nabla_{\xi} X_s|_{\R^2}^2 = 1$. Therefore, for all $\varphi \in \mathbf{H}^1(\Gamma_s)$, these operators simply write
\begin{equation*}
\begin{array} {rcl}
(\nabla_{\Gamma_s} \varphi)\circ X_s & = &  \nabla_{\xi} (\varphi \circ X_s), \\
(\divg_{\Gamma_s})\circ X_s & = & \displaystyle (\det \mathfrak{g}_s)^{-1/2} \frac{\p }{\p \xi}\left( (\det \mathfrak{g}_s)^{1/2} (\varphi \circ X_s) \right)
= \nabla_{\xi}(\varphi\circ X_s) = (\nabla_{\Gamma_s} \varphi)\circ X_s.
\end{array}
\end{equation*}
We will still use the general notation~$\nabla_{\Gamma_s}$ and~$\divg_{\Gamma_s}$, for the sake of consistency with higher dimension. Recall the Frenet-Serret formulas:
\begin{equation*}
\nabla_{\Gamma_s} n_s = -\frac{1}{r}\tau_s, \quad 
\nabla_{\Gamma_s} \tau_s = \frac{1}{r} n_s.
\end{equation*}
The integrals on $\Gamma_s$ have to be understood as surface integrals. We recall the Stokes formula on smooth manifolds without boundary, that we will use in this article only for the 1-sphere $\Gamma_s$. For all $\varphi, \psi \in \mathbf{H}^{1/2}(\Gamma_s)$, we have
\begin{equation*}
\langle \Delta_{\Gamma_s} \varphi , \psi 
\rangle_{\HH^{-1/2}(\Gamma_s), \HH^{1/2}(\Gamma_s)}  =  
-  \langle \nabla_{\Gamma_s} \varphi , 
\nabla_{\Gamma_s} \psi \rangle_{\mathds{L}^2(\Gamma_s)}, \label{Stokes-mani}
\end{equation*}
where we recall the definition $\Delta_{\Gamma_s} = \divg_{\Gamma_s} \circ \nabla_{\Gamma_s}$ of the Laplace-Beltrami operator on~$\Gamma_s$. More generally, for all matrix field~$\Sigma \in \mathds{H}^1(\Gamma_s)$ and vector field~$\varphi \in \mathbf{H}^1(\Gamma_s)$ we have
\begin{equation}
\langle \divg_{\Gamma_s} \Sigma , \varphi 
\rangle_{\HH^{-1/2}(\Gamma_s), \HH^{1/2}(\Gamma_s)}  =  
-  \langle \Sigma , 
\nabla_{\Gamma_s} \varphi \rangle_{\mathds{L}^2(\Gamma_s)}. \label{Stokes-mani2}
\end{equation}

We define {\it admissible} deformations, summarizing the basic assumptions we consider for mappings~$X$, as well as the set of admissible deformations transforming the circle~$\Gamma_s$ into another circle of the same radius and orientation:
\begin{definition} \label{def-admi}
We say that $X\in \mathbf{H}^2(\Gamma_s)$ is {\it admissible} if $X$ is invertible, orientation-preserving and volume-preserving, that is that the volume contained by $\Gamma_s$ is the same as the one contained by $X(\Gamma_s)$, and if~$X(\Gamma_s) \subset \mathring{\Omega}$. Further, we define 
\begin{equation*}
\Diff = \left\{
X \in \HH^2(\Gamma_s)\mid \ X(\Gamma_s) \subset \mathring{\Omega} \text{ is a circle of the same radius as }\Gamma_s
\right\}.
\end{equation*}
The relation $X(\Gamma_s) \subset \mathring{\Omega}$ is the condition of {\it non-contact} with the outer boundary, that is~$X(\Gamma_s) \cap \p \Omega = \emptyset$.
\end{definition}

\begin{remark} \label{rk-admi}
Note that any volume-preserving deformation $X\in \HH^2(\Gamma_s)$ that is close enough to the identity and volume-preserving is {\it admissible}. Relaxing the non-contact condition, remark that the space~$\Diff$ contains elements of the special Euclidean group $SE(2)$, made of proper rigid transformations $X_R$, namely direct isometries, composed of translations and rotations, as $X_R:\R^2 \ni y \mapsto h + \mathbf{R}y  \in \R^2,$
where $h\in \R^2$ and $\mathbf{R}$ is an orthogonal matrix with $\det \mathbf{R} = 1$. The space~$\Diff$ also includes the group~$\mathrm{Diff}^+(\Gamma_s)$ of direct diffeomorphisms of the circle~$\Gamma_s$. We claim that~$\Diff$ can be generated by composing elements of~$SE(2)$ with elements of~$\mathrm{Diff}(\Gamma_s)$. Finally, we note that we can obtain the circle $X(\Gamma_s)$ globally -- as geometric object -- from $\Gamma_s$ simply by composing the latter by a translation of~$SE(2)$. But for the sake of completeness we introduce~$\Diff$ as above for describing all the possible stationary states.
\end{remark}

We derive the three following lemmas, that will be used several times throughout the paper. The first one characterizes the stationary states, given by system~\eqref{stationary}:

\begin{lemma} \label{lemma-Temam}
Let be subdomains $\Omega^\pm \subset \Omega$ like in Lemma~\ref{lemma-Korn}. If $g\in \mathbf{H}^{-1/2}(\Gamma)$ is the right-hand-side of the following system
\begin{equation}
\left\{ \begin{array} {rcl}
-\divg \sigma (u^\pm, p^\pm) = 0 \quad \text{ and } \quad 
\divg u^\pm = 0 & & \text{in } \Omega^\pm, \\
u^+ = 0 & & \text{on } \p \Omega, \\
u^\pm = 0  \quad \text{ and } \quad  -\left[ \sigma(u,p) \right]n = g
& & \text{on } \Gamma,
\end{array} \right. \label{sys-1st}
\end{equation}
then necessarily $g = c n$, where $c$ is a constant equal to the difference~$[p]$ of constant pressures. In particular, the admissible mappings $X$ such that $X(\Gamma_s)\subset \mathring{\Omega}$ splits $\Omega$ into two subdomains denoted by $\Omega_{X(\Gamma_s)}^\pm = \Omega^\pm$, like in Lemma~\ref{lemma-Korn}, and satisfying
\begin{equation}
\left\{ \begin{array} {rcl}
-\divg \sigma (u^\pm, p^\pm) = 0 \quad \text{ and } \quad 
\divg u^\pm = 0 & & \text{in } \Omega_{X(\Gamma_s)}^\pm, \\
u^+ = 0 & & \text{on } \p \Omega, \\
u^\pm = 0  \quad \text{ and } \quad  -\left[ \sigma(u,p) \right]n = \mu \Delta_{X(\Gamma_s)}\Id
& & \text{on } X(\Gamma_s),
\end{array} \right. \label{stationary}
\end{equation}
describe the set~$\Diff$. That is, $X(\Gamma_s)$ is a circle of the same radius $r>0$ as $\Gamma_s$. The velocities $u^\pm$ are equal to zero everywhere, and the pressures $p^\pm$ are constant, equal to the static pressures $p_s^\pm$ such that $[p_s]n_s = \mu \Delta_{X(\Gamma_s)} \Id = \kappa_s n_s$ (namely the so-called Young-Laplace equation), where $\kappa_s = -1/r<0$ is the curvature of $X(\Gamma_s)$, implying that $p_s^- > p_s^+$.
\end{lemma}

\begin{proof}
Taking the scalar product of the first equation of~\eqref{sys-1st}, and integrating by parts leads to~$\|\varepsilon(u^\pm)\|_{\LLL^2(\Omega_{X(\Gamma_s)}^\pm)} = 0$, and from Lemma~\ref{lemma-Korn}, to $u^\pm = 0$ in $\HH^1(\Omega^\pm)$. Then we also deduce $\nabla p^\pm = 0$ in the first equation, that yields that $p^\pm$ are both constant, equal to the static pressures $p_s^\pm$. Thus $g = \left[2\nu \varepsilon(u) - p \I \right]n = -\left[p_s\right] n$, where~$\left[p_s\right]$ is a constant, which completes the first part of the proof. Next, using this result for system~\eqref{stationary} with $\Omega^\pm = \Omega_{X(\Gamma_s)}^\pm$, $\Gamma= \Gamma_{X(\Gamma_s)}$ and $g = \mu\Delta_{X(\Gamma_s)} \Id$, the constant $\left[p\right]$ obtained previously corresponds to $\left[p\right]n = \mu\Delta_{X(\Gamma_s)} \Id$. Actually $\mu\Delta_{X(\Gamma_s)} \Id = \kappa n$, where $\kappa$, namely the (mean) curvature of~$X(\Gamma_s)$, is constant, equal to $\kappa_s = -1/r$. Therefore~$X(\Gamma_s)$ is a circle, and since $X$ is assumed to be {\it admissible}, this condition yields that $X(\Gamma_s)$ and $\Gamma_s$ have the same radius $r$, and so the same mean curvature $-1/r$. 
Therefore~$X$ lies necessarily in~$\Diff$, and conversely, which completes the proof.
\end{proof}

The stationary states are then obtained from the reference circle~$\Gamma_s$ via transformations of~$\Diff$. The deformations $X$ of $\Gamma_s$ will be then compared to an element $X_c \in \Diff$, and the corresponding displacements writes $X-X_c$. For the sake of simplicity we will rather consider $X-\Id$ in what follows, by keeping in mind that $\Id$ is to be replaced by any given $X_c \in \Diff$. From there, we will use the notation
\begin{equation*}
Z = X- \Id 
\end{equation*}
for the displacements, keeping in mind that ultimately we shall consider $Z=X-X_c$ with~$X_c \in \Diff$. Next we introduce the following differential operator 
\begin{equation*}
\nabla^{n_s}_{\Gamma_s} = (n_s \otimes n_s) \nabla_{\Gamma_s}.
\end{equation*}
Operator $\nabla^{n_s}_{\Gamma_s}$ appears in the linearized system~\eqref{mainsyslin} (see section~\ref{sec-extension2} for its derivation). Note that the matrix field $n_s \otimes n_s$ is never invertible. 

\subsection{The kernel of~$\nabla^{n_s}_{\Gamma_s}$ and the lack of coercivity}
\label{sec-kernel}
The operator~$\divg_{\Gamma_s}\nabla^{n_s}_{\Gamma_s}$ appears in the linearization of $(\kappa n) \circ X$ (see section~\ref{sec-extension2}). The description of the kernel of~$\nabla_{\Gamma_s}$ is then central for the methodology we adopted, namely the wellposedness of the corresponding linearized system (section~\ref{sec-linear}), and the unique continuation argument (section~\ref{sec-approx-cont}). Unfortunately there are non-trivial mappings~$X = Z+\Id$ that are smooth, orientation-preserving, volume-preserving, transforming the circle into a Jordan curve, that can be chosen arbitrarily close to the identity, and such that~$\nabla^{n_s}_{\Gamma_s} Z = 0$.\\
Indeed, $\nabla^{n_s}_{\Gamma_s} Z = (\nabla_{\Gamma_s}Z \cdot n_s)n_s= 0$ implies that $\nabla_{\Gamma_s} Z$ is tangent to~$\Gamma_s$. There exists a function $\xi \mapsto \alpha(\xi)$ such that $\nabla_{\Gamma_s} Z = \alpha \tau_s$. Decompose~$\alpha$ with its Fourier series:
\begin{equation*}
\alpha(\xi) = \sum_{k=2}^{\infty} \big(
a_{1,k} \cos(k\xi/r) + a_{2,k} \sin(k\xi/r)
\big).
\end{equation*}
The modes $k=0$ and $k=1$ are not considered, as they introduce constants that are necessarily equal to~$0$, because of the periodicity of~$Z$, meaning that $(\Id+Z)(\Gamma_s)$ is a closed curve ($Z(X_s(0)) = Z(X_s(2\pi r))$). Integrating the equality~$\nabla_{\Gamma_s} Z = \alpha \tau_s$ leads us to
\begin{equation*}
Z(X_s(\xi)) =  \sum_{k=2}^{\infty} \frac{1}{k^2-1}
\Big(
\big(
a_{1,k} \cos(k\xi/r) + a_{2,k} \sin(k\xi/r)
\big) n_s +
\big(
-a_{2,k}k \cos(k\xi/r) + a_{1,k}k \sin(k\xi/r)
\big)\tau_s
\Big),
\end{equation*}
up to a constant that corresponds to a translation. Several examples of such displacements are represented in Figure~\ref{fig-bubble} below, corresponding to $k \in \{2,3,4,5,6,7\}$ and coefficients $a_k, b_k \in \{0,1\}$.\\
\begin{minipage}{0.96\linewidth}
\centering
\begin{tabular} {c|c|c|c|c|c}
\hspace*{0.0cm}
\begin{minipage}{0.16\linewidth}
\begin{figure}[H]
\includegraphics[trim = 5cm 9cm 5cm 9cm, clip, scale=0.14]{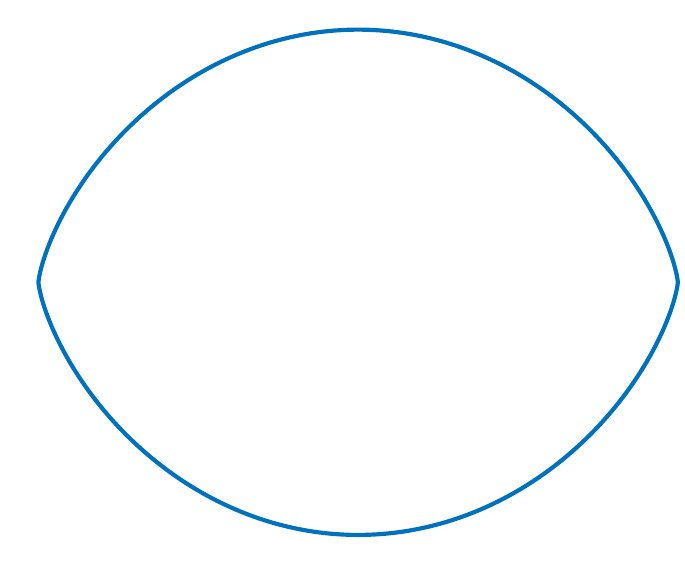}
\end{figure}
\end{minipage}\hspace*{-15pt}
&
\hspace*{5pt}
\begin{minipage}{0.16\linewidth}
\begin{figure}[H]
\includegraphics[trim = 5cm 9cm 5cm 9cm, clip, scale=0.16]{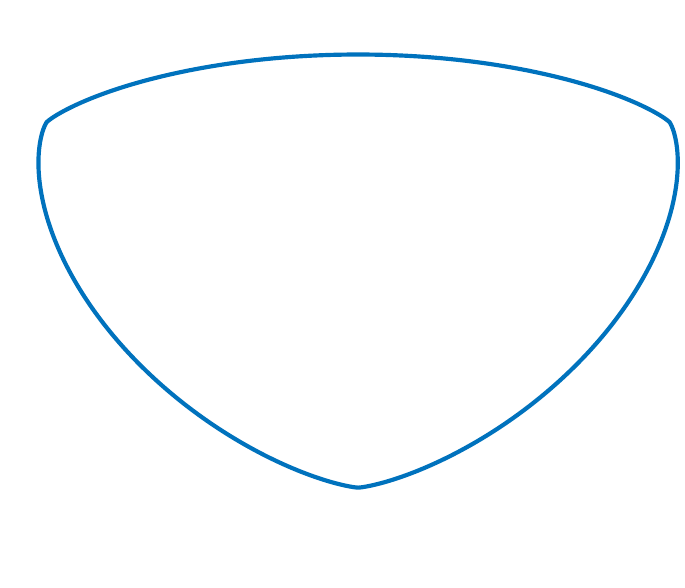}
\end{figure}
\end{minipage}\hspace*{-10pt}
&
\begin{minipage}{0.16\linewidth}
\begin{figure}[H]
\includegraphics[trim = 5cm 9cm 5cm 9cm, clip, scale=0.20]{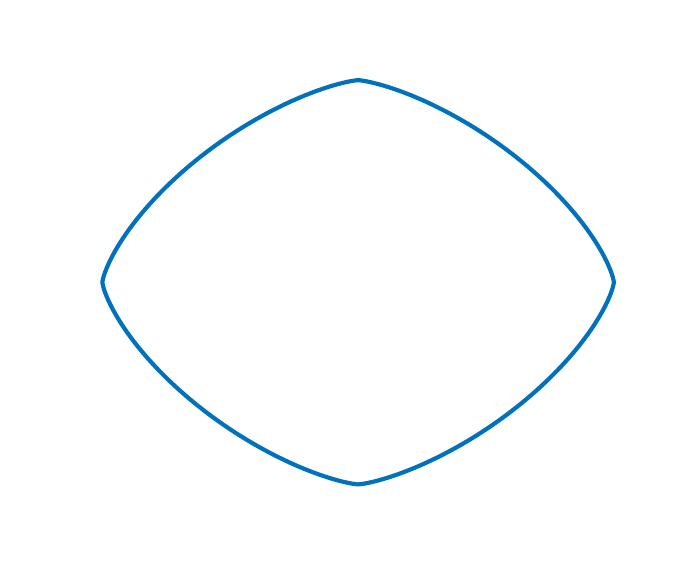}
\end{figure}
\end{minipage}
&
\hspace*{5pt}
\begin{minipage}{0.16\linewidth}
\begin{figure}[H]
\includegraphics[trim = 5cm 9cm 5cm 9cm, clip, scale=0.14]{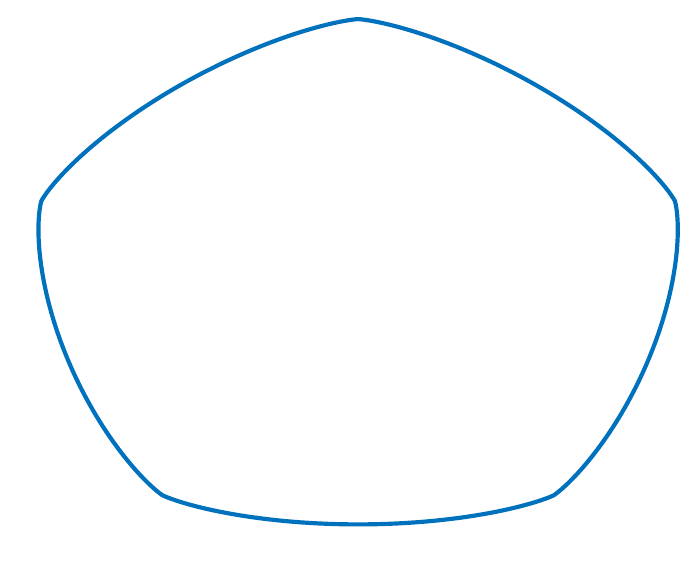}
\end{figure}
\end{minipage}\hspace*{-15pt}
&
\hspace*{5pt}
\begin{minipage}{0.16\linewidth}
\begin{figure}[H]
\includegraphics[trim = 5cm 9cm 4.5cm 9cm, clip, scale=0.14]{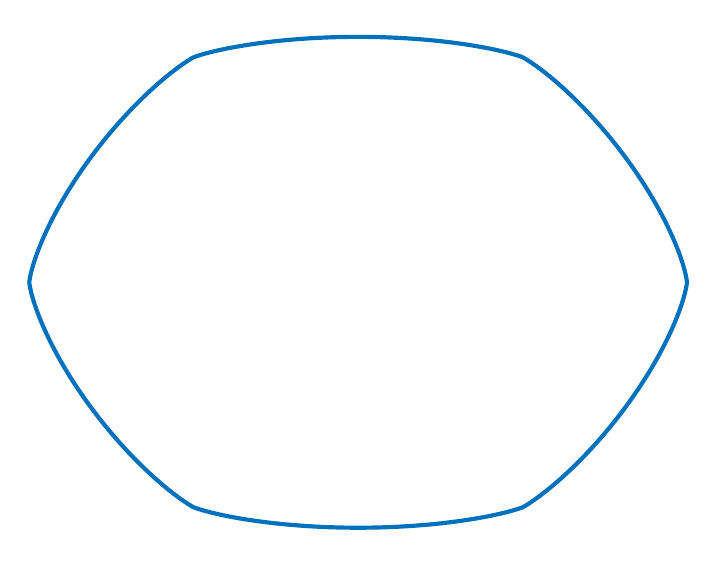}
\end{figure}
\end{minipage}\hspace*{-15pt}
&
\hspace*{5pt}
\begin{minipage}{0.16\linewidth}
\begin{figure}[H]
\includegraphics[trim = 5cm 9cm 5cm 9cm, clip, scale=0.14]{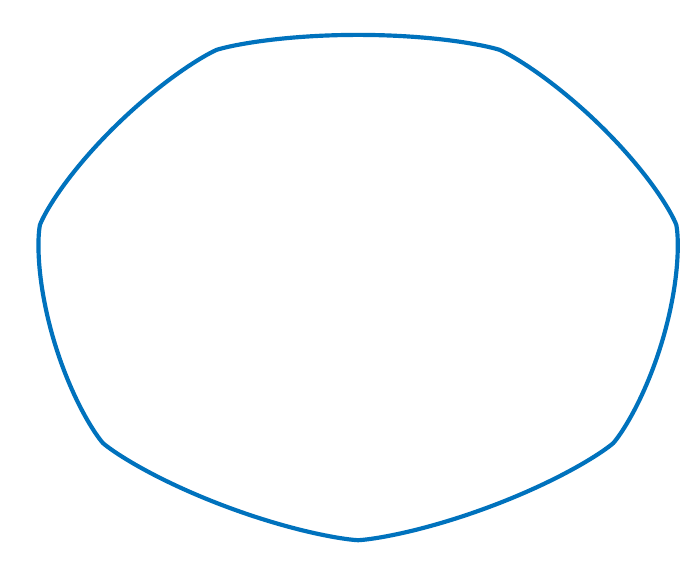}
\end{figure}
\end{minipage}
\end{tabular}
\vspace*{-10pt}
\begin{figure}[H]
\caption{Different deformations $X(\Gamma_s)$ of the circle such that~$\nabla^{n_s}_{\Gamma_s} (X-\Id) =0$.\label{fig-bubble}}
\end{figure}
\end{minipage}\vspace*{5pt}\\ 
In order to define a linear operator whose the kernel is reduced to a trivial set, we need to regularize the operator~$\divg_{\Gamma_s}\nabla^{n_s}_{\Gamma_s}$.

\subsection{Regularization}
Instead of considering~$\mu\divg_{\Gamma_s}\nabla^{n_s}_{\Gamma_s}$ alone, we add a regularizing term, namely
\begin{equation*}
\mu\divg_{\Gamma_s}\nabla^{\tau_s}_{\Gamma_s} = \mu \divg_{\Gamma_s} \circ \big( (\tau_s \otimes \tau_s) \nabla_{\Gamma_s} \big),
\end{equation*}
so that the linear system studied in section~\ref{sec-linear} involves the following operator
\begin{equation*}
Z\mapsto -\mu\divg_{\Gamma_s} \nabla^{n_s}_{\Gamma_s} Z 
-\mu\divg_{\Gamma_s} \nabla^{\tau_s}_{\Gamma_s} Z = -\mu\divg_{\Gamma_s} \nabla_{\Gamma_s} Z = -\mu\Delta_{\Gamma_s} Z.
\end{equation*}
For the Laplace-Beltrami operator~$-\Delta_{\Gamma_s}$ we recall a rigidity result combined with G\r{a}rding-type inequalities.
\begin{proposition} \label{prop-rigidity}
Assume that~$Z\in \mathbf{H}^s(\Gamma_s)$ with $s\geq 1$ satisfies~$\Delta_{\Gamma_s} Z =0$. Then $Z$ is a constant of~$\R^2$. Moreover, the following estimates hold:
\begin{eqnarray}
\|Z\|_{\mathbf{H}^1(\Gamma_s)/\R^2}  \leq C
\|\nabla_{\Gamma_s} Z\|_{\mathds{L}^2(\Gamma_s)}, & &
\label{eq-rigidity} \\
\|Z\|_{\mathbf{H}^{5/2}(\Gamma_s)/\R^2}  \leq C
\|\nabla_{\Gamma_s} Z\|_{\mathds{H}^{3/2}(\Gamma_s)},& & 
\|Z\|_{\mathbf{H}^{3/2}(\Gamma_s)/\R^2}  \leq C
\|\nabla_{\Gamma_s} Z\|_{\mathds{H}^{1/2}(\Gamma_s)}.
\label{eq-rigidity2}
\end{eqnarray}
Furthermore, if~$Z\in \mathbf{H}^2(\Gamma_s)$, the following estimate holds:
\begin{equation}
\|Z\|_{\mathbf{H}^2(\Gamma_s)/\R^2}  \leq C
\|\Delta_{\Gamma_s}Z \|_{\LL^2(\Gamma_s)}.
\label{eq-garding}
\end{equation}
\end{proposition}

\begin{proof}
See for example~\cite[section~2.8]{Hebey}. Estimate~\eqref{eq-rigidity} is deduced from the Poincar\'e inequality combined with the Petree-tartar lemma. The same inequality applies to high-order derivatives of~$Z$, and consequently estimates~\eqref{eq-rigidity2} are deduced by interpolation. Finally, estimate~\eqref{eq-garding} is classically obtained by using Fourier series on the circle.
\end{proof}

We deduce a unique continuation result that is used for proving Proposition~\ref{prop-approx-cont}.
\begin{lemma} \label{lemma-Temam2}
The mappings $X \in \HH^1(\Gamma_s)$ satisfying the following system
\begin{equation}
\left\{ \begin{array} {rcl}
-\divg \sigma (u^\pm, p^\pm) = 0 \quad \text{ and } \quad 
\divg u^\pm = 0 & & \text{in } \Omega_s^\pm, \\
u^+ = 0 & & \text{on } \p \Omega, \\
u^\pm = 0  \quad \text{ and } \quad  -\left[ \sigma(u,p) \right]n = \mu \Delta_{\Gamma_s} (X-\Id)
& & \text{on } \Gamma_s,
\end{array} \right. \label{sys-2nd}
\end{equation}
are translations of $\R^2$.
\end{lemma}

\begin{proof}
By using Lemma~\ref{lemma-Temam} with $\Omega^\pm = \Omega_s^\pm$, $\Gamma= \Gamma_s$ and $g = \mu \Delta_{\Gamma_s} (X-\Id)$, we deduce $\Delta_{\Gamma_s} (X-\Id) = cn_s$, where $c$ is a constant equal to the difference of constant pressures. Recall that the static pressures introduced in Lemma~\ref{lemma-Temam} satisfy $\mu \Delta_{\Gamma_s} \Id = [p_s]n_s$. Therefore we can assume that the constant pressures mentioned above are such that $c = 0$, and so $\Delta_{\Gamma_s} (X-\Id) = 0$. Since $\Gamma_s$ is a compact manifold, it follows that $X-\Id$ is a constant, and $X$ is a translation.
\end{proof}

\begin{remark}
In the proof of Lemma~\ref{lemma-Temam2}, another argument for deducing $c=0$ would have consisted in restricting the deformation~$X$ to volume-preserving deformations. Since $n_s = -r \Delta_{\Gamma_s} \Id$, we would have deduced $\Delta_{\Gamma_s} (X-\Id+cr\Id) = 0$, implying that $X+ (cr-1)\Id$ is a constant. Assuming that $X$ is volume-preserving then leads to $c=0$. This also amounts saying that the difference of static pressures remains the same after deformation by~$X$.
\end{remark}


The third lemma is an energy estimate that holds for an unsteady linear system:
\begin{lemma} \label{lemma-Energy}
Let be $T>0$ and $\lambda >0$. Assume that $Z$ satisfies
\begin{equation}
\left\{ \begin{array} {rcl}
-\divg \sigma (u^\pm, p^\pm) = 0 \quad \text{ and } \quad 
\divg u^\pm = 0 & & \text{in } \Omega_s^\pm \times (0,T), \\
u^+ = 0 & & \text{on } \p \Omega \times (0,T), \\
u^\pm = \displaystyle \frac{\p Z}{\p t} -\lambda Z  \quad \text{ and } \quad  -\left[ \sigma(u,p) \right]n = \mu \Delta_{\Gamma_s} Z + g
& & \text{on } \Gamma_s \times (0,T), \\
Z(\cdot,0) = Z_0 & & \text{on } \Gamma_s.
\end{array} \right. \label{sys-energy}
\end{equation}
Then the following identity holds almost everywhere in $(0,T)$:
\begin{equation}
\frac{\mu}{2}\frac{\p }{\p t} 
\| \nabla_{\Gamma_s} Z \|_{\mathds{L}^2(\Gamma_s)}^2
+ 2\nu \left( 
\|\varepsilon(u^+) \|_{\LLL^2(\Omega_s^+)}^2 + \|\varepsilon(u^-) \|_{\LLL^2(\Omega_s^-)}^2\right) =
\lambda \|\nabla_{\Gamma_s} Z \|_{\mathds{L}^2(\Gamma_s)}^2
+ \left\langle g ; u^\pm \right\rangle_{\LL^2(\Gamma_s)}.
\label{energy-estimate}
\end{equation}
In particular, if $g=0$ and $Z_0 = 0$ in $\HH^1(\Gamma_s)$, then $Z\equiv 0$ up to a constant of~$\R^2$.
\end{lemma}

\begin{proof}
Taking the scalar product of the first equation of~\eqref{sys-energy} by $u^\pm$, and integrating by parts yields
\begin{equation*}
2\nu \left( 
\|\varepsilon(u^+) \|_{\LLL^2(\Omega_s^+)}^2 + 
\|\varepsilon(u^-) \|_{\LLL^2(\Omega_s^-)}^2\right)
= \mu \left\langle \frac{\p Z}{\p t}-\lambda Z , \Delta_{\Gamma_s} Z \right\rangle_{\LL^2(\Gamma_s)} 
+\left\langle g , u^\pm \right\rangle_{\LL^2(\Gamma_s)}.
\end{equation*}
The Stokes formula~\eqref{Stokes-mani2} shows that $\divg_{\Gamma_s}^{\ast} =-\nabla_{\Gamma_s}$, so we deduce
\begin{equation*}
\begin{array} {rcl}
\displaystyle 
\left\langle \frac{\p Z}{\p t}-\lambda Z ,\mu \Delta_{\Gamma_s} Z \right\rangle_{\LL^2(\Gamma_s)} & = & - \displaystyle 
\left\langle \frac{\p \nabla_{\Gamma_s} Z}{\p t}-\lambda \nabla_{\Gamma_s}Z , 
\mu\nabla_{\Gamma_s} Z \right\rangle_{\mathds{L}^2(\Gamma_s)} \\
& = & \displaystyle 
-\frac{\mu}{2}\frac{\p }{\p t}
 \|  \nabla_{\Gamma_s}Z \|^2_{\mathds{L}^2(\Gamma_s)} 
+ \lambda \|  \nabla_{\Gamma_s}Z \|^2_{\mathds{L}^2(\Gamma_s)},
\end{array}
\end{equation*}
and thus~\eqref{energy-estimate} follows. By integrating in time this identity, with $g=0$, we obtain $ \| \nabla_{\Gamma_s} Z \|_{\mathds{L}^2(\Gamma_s)}^2 \leq  \| \nabla_{\Gamma_s} Z_0 \|_{\mathds{L}^2(\Gamma_s)}^2 = 0$, which yields $\nabla_{\Gamma_s}Z=0$, and Proposition~\ref{prop-rigidity} enables us to complete the proof. 
\end{proof}
Note that the energy estimate~\eqref{energy-estimate} is also valid for any geometric configuration satisfying the non-contact condition, but we will use it only for $\Gamma_s$.

\section{Extension of diffeomorphisms and change of variables}

In order to rewrite~\eqref{mainsys} in time-independent domains, we need to define in the whole domain $\Omega$ a change of variable that coincides with $X$ on $\Gamma_s$. Therefore the question of the extension of~$X$ in~$\Omega_s^\pm$ arises.

\subsection{Extension of diffeomorphisms} \label{sec-extension1}

Let us define an extension of $X$ which inherits its regularity properties. We state the following result:

\begin{proposition} \label{prop-extension}
Let be $X-\Id\in \mathcal{Z}_{\infty}(\Gamma_s)$ such that for all $t\geq 0$ the mapping $X(\cdot,t)$ is a diffeomorphism from~$\Gamma_s$ onto~$\Gamma(t)$. There exists a mapping~$\tilde{X}$ defined in $\Omega \times (0,\infty)$ such that for all $t \geq 0$ we have $\tilde{X}(\cdot,t)_{| \Omega_s^\pm} \in \HH^{5/2}(\Omega_s^\pm)$, $\tilde{X}_{|\p \Omega} = \Id$, $\tilde{X}_{|\Gamma_s} = X$, $\tilde{X}(\cdot,t)_{| \Omega_s^-}$ is a diffeomorphism from $\Omega_s^-$ onto $\Omega^-(t)$, and that satisfies
\begin{equation}
\|\tilde{X} - \Id \|_{\mathcal{X}_{\infty}(\Omega_s^\pm)} \leq C
\|X - \Id \|_{\mathcal{Z}_{\infty}(\Gamma_s)},
\label{est-extension1}
\end{equation}
where the constant $C>0$ is independent of~$X$. Furthermore, given $X_1-\Id, \, X_2-\Id \in \mathcal{Z}_{\infty}(\Gamma_s)$, the respective extensions $\tilde{X}_1$ and $\tilde{X}_2$ so obtained satisfy
\begin{equation}
\|\tilde{X}_1 - \tilde{X}_2 \|_{\mathcal{X}_{\infty}(\Omega_s^\pm)} \leq C
\|X_1 - X_2 \|_{\mathcal{Z}_{\infty}(\Gamma_s)}.
\label{est-extension2}
\end{equation}
\end{proposition}

The $\HH^{5/2}(\Omega_s^\pm)$ regularity implies in particular that for every $t\geq0$ the mapping $\tilde{X}( \cdot,t)_{|\Omega_s^\pm}$ is of class $\mathcal{C}^1$ on $\overline{\Omega_s^\pm}$. The proof of Proposition~\ref{prop-extension} combines different results that are not related to the main result of the paper. Therefore it is given in Appendix~\ref{appendix1}. 
Note that the domains $\Omega^\pm(t)$ defined as the two connected components of $\Omega \setminus \Gamma(t)$ are also described as $\Omega^\pm(t) = \tilde{X}(\Omega_s^\pm,t)$, due to conexity. By choosing $X-\Id$ small enough, we can define a local inverse for~$\tilde{X}_{|\Omega_s^+}(\cdot,t)$. We deduce regularity for the inverse of the Jacobian matrix of~$\tilde{X}(\cdot,t)_{| \Omega_s^\pm}$.

\begin{corollary} \label{coro-extension}
Given the assumptions of Proposition~\ref{prop-extension}, the inverse~$\tilde{Y}(\cdot,t)$ of mapping~$\tilde{X}(\cdot,t)_{|\Omega_s^\pm}$, so that
\begin{equation*}
\tilde{X}(\tilde{Y}(x,t),t) = x, \text{ for } x\in \Omega^\pm(t), \ t\in (0,\infty), \qquad
\tilde{Y}(\tilde{X}(y,t),t) = y \text{ for } (y,t) \in \Omega_s^\pm \times (0,\infty),
\end{equation*}
satisfies the following estimate, provided that $\|X - \Id \|_{\mathcal{Z}_{\infty}(\Gamma_s)}$ is small enough:
\begin{equation}
\|\nabla \tilde{Y}(\tilde{X}) - \I \|_{\L^\infty(0,\infty; \mathbb{H}^{3/2}(\Omega_s^\pm))} \leq C
\|X - \Id \|_{\mathcal{Z}_{\infty}(\Gamma_s)}.
 \label{est-extension-inv1}
\end{equation}
Denoting by $\tilde{Y}_1$ and $\tilde{Y}_2$ the respective inverses of $\tilde{X}_1$ and $\tilde{X}_2$, extensions of $X_1$ and $X_2$ respectively, we have
\begin{equation}
\|\nabla \tilde{Y}_1(\tilde{X}_1) - \nabla \tilde{Y}_2(\tilde{X}_2) \|_{\L^\infty(0,\infty; \mathbb{H}^{3/2}(\Omega_s^\pm))} \leq C
\|X_1 - X_2 \|_{\mathcal{Z}_{\infty}(\Gamma_s)},
 \label{est-extension-inv2}
\end{equation}
provided that $\|X_1 - \Id \|_{\mathcal{Z}_{\infty}(\Gamma_s)}$ and $\|X_2 - \Id \|_{\mathcal{Z}_{\infty}(\Gamma_s)}$ are small enough.
\end{corollary}

\begin{proof}
Even if $\tilde{X}_{|\Omega_s^+}(\cdot,t)$ is not globally invertible, we still use the notation $\nabla \tilde{Y}(\tilde{X})$ for the inverse of $\nabla \tilde{X}$, for the sake of simplicity. Recall that $\mathbb{H}^{3/2}(\Omega_s^\pm)$ is an algebra. The identity $\nabla \tilde{Y}(\tilde{X})-\I = (\I - \nabla \tilde{X})(\nabla \tilde{Y}(\tilde{X}) - \I) + (\I - \nabla \tilde{X})$ yields
\begin{equation*}
\| \nabla \tilde{Y}(\tilde{X})-\I \|_{\L^\infty(0,\infty; \mathbb{H}^{3/2}(\Omega_s^\pm))} \leq 
\frac{\| \nabla \tilde{X}-\I \|_{\L^\infty(0,\infty; \mathbb{H}^{3/2}(\Omega_s^\pm))}}{1- C \|\nabla \tilde{X}-\I\|_{\L^\infty(0,\infty; \mathbb{H}^{3/2}(\Omega_s^\pm))}}
\end{equation*}
which, combined with~\eqref{est-extension1}, implies~\eqref{est-extension-inv1}. Further, the identity 
\begin{equation*}
\nabla \tilde{Y}_1(\tilde{X}_1)- \nabla \tilde{Y}_2(\tilde{X}_2) =
\big(\nabla \tilde{Y}_1(\tilde{X}_1) - \nabla \tilde{Y}_2(\tilde{X}_2)\big) (\I - \nabla \tilde{X}_1)
- (\nabla \tilde{Y}_2(\tilde{X}_2) - \I)(\nabla \tilde{X}_1 - \nabla \tilde{X}_2)
-(\nabla \tilde{X}_1 - \nabla \tilde{X}_2)
\end{equation*}
enables us to derive~\eqref{est-extension-inv2} similarly, by using~\eqref{est-extension1} for controlling $(\I-\nabla \tilde{X}_1)$,~\eqref{est-extension-inv1} for controlling~$(\nabla \tilde{Y}_2(\tilde{X}_2) - \I)$, and~\eqref{est-extension2} for controlling $(\nabla \tilde{X}_1 - \nabla \tilde{X}_2)$, completing the proof.
\end{proof}

\begin{remark}
In case we would consider stabilizing $X$ around some $X_c \in \Diff$, instead of $\Id$, we would need to define an extension~$\tilde{X}_c \in \HH^{5/2}(\Omega_s^\pm)$ of $X_c \in \HH^2(\Gamma_s)$. Such an extension is provided by Proposition~\ref{prop-RKC} in Appendix~\ref{appendix1}, and the Lipschitz estimates~\eqref{est-extension1} and~\eqref{est-extension-inv1} would then be
\begin{equation*}
\begin{array} {rcl}
\|\tilde{X} - \tilde{X}_c \|_{\mathcal{X}_{\infty}(\Omega_s^\pm)} \leq C
\|X - X_c \|_{\mathcal{Z}_{\infty}(\Gamma_s)}, & & 
\|\nabla \tilde{Y}(\tilde{X}) - (\nabla \tilde{X}_c)^{-1} \|_{\L^\infty(0,\infty; \mathbb{H}^{3/2}(\Omega_s^\pm))} \leq C
\|X - X_c \|_{\mathcal{Z}_{\infty}(\Gamma_s)},
\end{array}
\end{equation*}
respectively.
\end{remark}

We now have the tools for rewriting system~\eqref{mainsys} in cylindrical domains. The Lipschitz estimates of Proposition~\ref{prop-extension} and Corollary~\ref{coro-extension} are used in section~\ref{sec-nonlinear}.

\subsection{Change of variables and system transformation} \label{sec-extension2}

We introduce the change of variables
\begin{equation*}
\begin{array} {rrl}
\tilde{u}^\pm(y,t) := u^{\pm}(\tilde{X}(y,t),t), &
\tilde{p}^\pm(y,t) := p^{\pm}(\tilde{X}(y,t),t), &
(y,t) \in \Omega_s^\pm \times (0,\infty), \\
u^\pm(x,t) = \tilde{u}^\pm(\tilde{Y}(x,t),t), & 
p^\pm(x,t) = \tilde{p}^\pm(\tilde{Y}(x,t),t), & 
(x,t) \in \displaystyle \bigcup_{t\in (0,\infty)} \Omega^\pm(t) \times \{t\}, \\
& \tilde{f}^\pm(y,t) := (\det \nabla \tilde{X}(y,t) )f^{\pm}(\tilde{X}(y,t),t), 
&  (y,t) \in \Omega_s^\pm \times  (0,\infty), \\
& \tilde{g}(y,t) : = |\cof(\nabla \tilde{X}(y,t))n_s |\,
g(X(y,t),t), &
(y,t) \in  \Gamma_s \times (0,\infty).
\end{array}
\end{equation*}
Implicitly, on $\Gamma_s$, $\nabla \tilde{X}$ refers to~$\nabla \tilde{X}^-$. Composing system~\eqref{mainsys} by $\tilde{X}$ and using the Piola's identity, we obtain
\begin{equation}
\begin{array} {rcl}
-\divg \left(\left(\sigma(u^\pm,p^\pm) \circ \tilde{X}\right) \cof(\nabla \tilde{X}) \right) = \tilde{f}^\pm \quad \text{ and } \quad 
\divg\left( \cof(\nabla \tilde{X})^T \tilde{u} \right) = 0 & & 
\text{in } \Omega_s^\pm \times (0,\infty), \\
\tilde{u}^+ = 0 & &  \text{on } \p \Omega \times (0,\infty), \\
\tilde{u}^\pm = \displaystyle \frac{\p X}{\p t} \quad \text{ and } \quad
-\left[ \left(\sigma(u,p)\circ X\right) \right]\cof \nabla \tilde X n_s = 
\mu |\cof \nabla \tilde X n_s| (\Delta_{\Gamma(t)} \Id) \circ X + \tilde{g} & & \text{on } \Gamma_s \times (0, \infty), \\
X(\cdot,0) = X_0 & & \text{on } \Gamma_s,
\end{array} \label{sys-cyl1}
\end{equation}
where 
\begin{equation*}
\sigma(u,p) \circ \tilde{X} =
\nu\, \mathrm{Sym}(\nabla \tilde{u} \nabla \tilde{Y}(\tilde{X}))  - \tilde{p}\, \I
=: \tilde{\sigma}(\tilde{u},\tilde{p}),
\end{equation*}
and where we used $n\circ X = \cof \nabla \tilde X n_s/|\cof \nabla \tilde X n_s|$. Let us develop $(\Delta_{\Gamma(t)} \Id) \circ X$. The parameterization $X_s : \ [0,2\pi r)  \ni \xi \mapsto X_s(\xi) \in \Gamma_s$ of~$\Gamma_s$ introduced in section~\ref{sec-notation2} enables us to define $X\circ X_s$ as a parameterization of~$\Gamma(t)$. We denote by $\mathfrak{g}_s$,~$\mathfrak{g}(t)$ the metric tensors of $\Gamma_s$ and $\Gamma(t)$, respectively. We have, in local coordinates, and using the Einstein notation\footnote{Even if in our case $\mathfrak{g}(t)$ is scalar-valued, we still keep the matrix formulation for generalizability to higher dimension.}:
\begin{equation}
\begin{array} {rcl}
(\Delta_{\Gamma(t)} \Id)_i \circ  X \circ X_s & = &
(\det \mathfrak{g}(t))^{-1/2} \displaystyle 
\frac{\p}{\p \xi_j} \left( 
(\det \mathfrak{g}(t))^{1/2}
\frac{\p (X\circ X_s)_i}{\p \xi_k} (\mathfrak{g}(t)^{-1})_{kj} \right).
\end{array} \label{Laplace-trans}
\end{equation}
Volume forms considerations yield 
\begin{equation*}
(\det \mathfrak{g}(t))^{1/2} = |\cof \nabla \tilde{X} n_s| (\det \mathfrak{g}_s)^{1/2} 
\end{equation*}
(see for instance~\cite[Lemma~6.23, p.~135]{Allaire}), and when $X-\Id$ is small, we write
\begin{equation*}
\begin{array} {rcl}
\mathfrak{g}(t) & = & \nabla_{\xi} (X\circ X_s)^T \nabla_{\xi}(X\circ X_s) = 
\mathfrak{g}_s + 
 \nabla_{\xi} (X\circ X_s)^T \nabla_{\xi}(X\circ X_s) 
- \nabla_{\xi}X_s^T \nabla_{\xi} X_s \\
& = & \mathfrak{g}_s + 2\, 
\mathrm{Sym}\left(
\nabla_{\xi}X_s^T \nabla_{\xi} ( (X-\Id)\circ X_s)
\right)  + 
\mathcal{O}(\|\nabla \tilde{X}-\I \|^2_{\mathbb{H}^1(\Gamma_s)} ), \\
\mathfrak{g}(t)^{-1} & = & \mathfrak{g}_s^{-1} 
- 2\, \mathfrak{g}_s^{-1}  
\mathrm{Sym}\left(
\nabla_{\xi}X_s^T \nabla_{\xi} ( (X-\Id)\circ X_s)
\right)\mathfrak{g}_s^{-1}   + 
\mathcal{O}(\|\nabla \tilde{X}-\I \|^2_{\mathbb{H}^1(\Gamma_s)} ) \\
& = & \mathfrak{g}_s^{-1} 
- 2\, \mathfrak{g}_s^{-1}  
\left(\nabla_{\xi}X_s^T \nabla_{\xi} ( (X-\Id)\circ X_s)
\right)\mathfrak{g}_s^{-1}   + 
\mathcal{O}(\|\nabla \tilde{X}-\I \|^2_{\mathbb{H}^1(\Gamma_s)} )  \\
\det \mathfrak{g}(t) & = & \det \mathfrak{g}_s +
2\, \cof \mathfrak{g}_s : \mathrm{Sym}\left(
\nabla_{\xi}X_s^T \nabla_{\xi} ( (X-\Id)\circ X_s)
\right) + 
\mathcal{O}(\|\nabla \tilde{X}-\I \|^2_{\mathbb{H}^1(\Gamma_s)} )\\ 
& = & (\det \mathfrak{g}_s)\left( 1+
2\, \mathfrak{g}_s^{-1} : \left(
\nabla_{\xi}X_s^T \nabla_{\xi} ( (X-\Id)\circ X_s)
\right)
\right) + 
\mathcal{O}(\|\nabla \tilde{X}-\I \|^2_{\mathbb{H}^1(\Gamma_s)} )
, \\
(\det \mathfrak{g}(t))^{1/2} & = & (\det \mathfrak{g}_s)^{1/2}\left(
1 + \mathfrak{g}_s^{-1} : \left(
\nabla_{\xi}X_s^T \nabla_{\xi} ( (X-\Id)\circ X_s)
\right)
\right)
+\mathcal{O}(\|\nabla \tilde{X}-\I \|^2_{\mathbb{H}^1(\Gamma_s)} ),\\
\displaystyle \frac{\p (X\circ X_s)_i}{\p \xi_k} & = & (\nabla_{\xi}X_s)_{ik} + 
\displaystyle \frac{\p ((X-\Id)\circ X_s)_i}{\p \xi_k},
\end{array}
\end{equation*}
where we have used the symmetry of~$\mathfrak{g}_s$ and~$\mathfrak{g}(t)$, and where the Landau's notation $\mathcal{O}$ applies when $\nabla \tilde{X} - \I$ is small in the algebra~$\mathbb{H}^{1}(\Gamma_s)$. Further, let us make some simplifications provided by dimension~2, using the parameterizaion by arc length given in section~\ref{sec-notation2}: The metric tensor $\mathfrak{g}_s$ is scalar valued, equal to $| \nabla_{\xi}X_s|^2 = 1$, and the tangent space of~$\Gamma_s$ is made of the tangent vectors $\tau_s$ such that $\tau_s \circ X_s = -\nabla_{\xi}X_s/|\nabla_{\xi} X_s| = - \nabla_{\xi} X_s$. We then have
\begin{equation*}
\begin{array} {rcl}
\mathfrak{g}(t)^{-1} & = & \mathfrak{g}_s^{-1} - 2 \displaystyle
\frac{\langle \nabla_{\xi} X_s, \nabla_{\xi}((X-\Id)\circ X_s) \rangle_{\R^2}}
{|\nabla_{\xi} X_s |^2}
+\mathcal{O}(\|\nabla \tilde{X}-\I \|^2_{\mathbb{H}^1(\Gamma_s)} )\\[10pt] 
& = & \mathfrak{g}_s^{-1} + 2 \displaystyle
\langle \tau_s\circ X_s, \nabla_{\xi}((X-\Id)\circ X_s) \rangle_{\R^2}
+\mathcal{O}(\|\nabla \tilde{X}-\I \|^2_{\mathbb{H}^1(\Gamma_s)} ),\\[10pt]
(\det \mathfrak{g}(t))^{1/2} & = & (\det \mathfrak{g}_s)^{1/2}\left(
1 + \displaystyle
\frac{\langle \nabla_{\xi} X_s, \nabla_{\xi}((X-\Id)\circ X_s) \rangle_{\R^2}}
{|\nabla_{\xi} X_s |^2}
\right)
+\mathcal{O}(\|\nabla \tilde{X}-\I \|^2_{\mathbb{H}^1(\Gamma_s)} )\\[10pt]
& = &
(\det \mathfrak{g}_s)^{1/2}\left(
1 - \displaystyle
\langle \tau_s\circ X_s, \nabla_{\xi}((X-\Id)\circ X_s) \rangle_{\R^2}
\right)
+\mathcal{O}(\|\nabla \tilde{X}-\I \|^2_{\mathbb{H}^1(\Gamma_s)} ),\\
\displaystyle \frac{\p (X\circ X_s)_i}{\p \xi_k} & = & -(\tau_s\circ X_s)_i + 
\displaystyle  \nabla_{\xi}((X-\Id)\circ X_s)_i .
\end{array}
\end{equation*}
Linearizing~\eqref{Laplace-trans}, we deduce
\begin{equation*}
\begin{array}{rcl}
(\Delta_{\Gamma(t)} \Id) \circ (X\circ X_s) & = & 
|\cof \nabla \tilde{X} n_s|^{-1}\big(\Delta_{\Gamma_s} \Id + 
\divg_{\Gamma_s}(\nabla_{\Gamma_s} (X-\Id)) \big) \circ X_s \\
& & \left.
-|\cof \nabla \tilde{X} n_s|^{-1} (\det \mathfrak{g}_s)^{-1/2}\displaystyle
\frac{\p }{\p \xi}\left(
(\det \mathfrak{g}_s)^{1/2} \displaystyle
\big((\tau_s \otimes
\tau_s)\circ X_s\big) \nabla_{\xi}((X-\Id)\circ X_s)
\right)
\right) \\
& & +\mathcal{O}(\|\nabla \tilde{X}-\I \|^2_{\mathbb{H}^1(\Gamma_s)} )\\
& = & |\cof \nabla \tilde{X} n_s|^{-1} \Big(
\Delta_{\Gamma_s} \Id + 
\divg_{\Gamma_s}(\nabla_{\Gamma_s} (X-\Id))
- \divg_{\Gamma_s}((\tau_s \otimes \tau_s)\nabla_{\Gamma_s} (X-\Id))
\Big)\circ X_s \\
& & +\mathcal{O}(\|\nabla \tilde{X}-\I \|^2_{\mathbb{H}^1(\Gamma_s)})
,\\
(\Delta_{\Gamma(t)} \Id) \circ X & = &
|\cof \nabla \tilde{X} n_s|^{-1}\big(\Delta_{\Gamma_s} \Id + 
\divg_{\Gamma_s}\left((n_s \otimes n_s)\nabla_{\Gamma_s} (X-\Id)\right) \big)
+\mathcal{O}(\|\nabla \tilde{X}-\I \|^2_{\mathbb{H}^1(\Gamma_s)}),
\end{array}
\end{equation*}
as $ n_s \otimes n_s= \I - \tau_s \otimes \tau_s$ is the projection operator on $n_s$. Thus, recalling the notation $\nabla^{n_s}_{\Gamma_s} = (n_s \otimes n_s) \nabla_{\Gamma_s}$, the fourth equation of~\eqref{sys-cyl1} writes 
\begin{equation*}
\begin{array} {rcl}
-\left[ \tilde{\sigma}(\tilde{u},\tilde{p}) \right]\cof \nabla \tilde X n_s = 
\mu  \Delta_{\Gamma_s} \Id  + \divg_{\Gamma_s} \nabla^{n_s}_{\Gamma_s}(X-\Id)
+ \tilde{g} 
+\mathcal{O}(\|\nabla \tilde{X}-\I \|^2_{\mathbb{H}^1(\Gamma_s)} )
& & \text{on } \Gamma_s \times (0, \infty).
\end{array}
\end{equation*}
The interest of~\eqref{sys-cyl1} lies in the fact that the space domains domains~$\Omega_s^\pm$ are time-independent. The price to pay is the nonlinear operators with respect to~$X$ that appear above. Recall that $\Delta_{\Gamma_s} \Id = [p_s]n_s$, where $p_s^\pm$ are the constant static pressures introduced in Lemma~\ref{lemma-Temam}. Further, we introduce the following unknowns:
\begin{equation}
\begin{array}{rcl}
\hat{u}^\pm = e^{\lambda t} \tilde{u}^\pm, \quad
\hat{p}^\pm = e^{\lambda t} (\tilde{p}^\pm - p_s^\pm), \quad
\hat{f}^\pm = e^{\lambda t} \tilde{f} & & \text{in } \Omega_s^\pm \times (0,\infty), \\
\hat{Z}(\cdot,t) = e^{\lambda t}(X(\cdot,t)-\Id), \quad \hat{g} = e^{\lambda t} \tilde{g} & & 
\text{on } \Gamma_s\times (0,\infty).
\end{array}\label{change-unknowns-bis}
\end{equation}
The idea here is to find a control~$\hat{g}$ that will make the variables~$(\hat{u}^\pm,\hat{p}^\pm,\hat{Z})$ bounded, so that the original unknwon will decrease exponentially with~$\lambda$ as decay rate. The system satisfied by $(\hat{u}^\pm,\hat{p}^\pm,\hat{X})$ is the following:
\begin{equation}
\begin{array} {rcl}
 - \divg (\sigma(\hat{u}^\pm, \hat{p}^\pm)
 =  \hat{f}^\pm +  F(\hat{u}^\pm,\hat{p}^\pm, \hat{Z})
\quad \text{ and } \quad
\divg \hat{u}  =  \divg H(\hat{u}^\pm, \hat{Z}) &  & 
\text{in } \Omega_s^\pm \times (0,\infty), \\
\hat{u}^+ = 0 & & \text{on } \p \Omega \times (0,\infty), \\
\hat{u}^\pm = \displaystyle \frac{\p \hat{Z}}{\p t} - \lambda \hat{Z}, 
\ \text{ and } \ 
-\left[\sigma(\hat{u},\hat{p})\right] n_s = 
\mu\divg_{\Gamma_s}\nabla^{n_s}_{\Gamma_s}\hat{Z} + 
\hat{g} + G(\hat{u}^+, \hat{p}^+, \hat{u}^-, \hat{p}^-, \hat{Z})
& & \text{on } \Gamma_s \times (0,\infty), \\
\hat{Z}(\cdot,0) = X_0-\Id & & \text{on } \Gamma_s.
\end{array} \label{sys-cyl2}
\end{equation}
where we have introduced
\begin{subequations} \label{rhs-nonlinear}
\begin{eqnarray}
 F(\hat{u}^\pm,\hat{p}^\pm, \hat{Z}) & = &  
\divg \left( \tilde{\sigma}(\hat{u}^\pm,\hat{p}^\pm) (\cof \nabla \tilde{X}-\I) \right)  + 2\nu\divg\left(
\mathrm{Sym}\left( \nabla \hat{u} (\tilde{Y}(\tilde{X})-\I)\right) 
\right), \label{rhs-F} \\
G(\hat{u}^+, \hat{p}^+, \hat{u}^-, \hat{p}^-, \hat{Z}) & = & 
\left[\tilde{\sigma}(\hat{u},\hat{p})\right] (\cof \nabla \tilde{X}-\I) n_s  + 2\nu\left[\mathrm{Sym}\left(\nabla \hat{u} (\tilde{Y}(\tilde{X})-\I) \right)
\right]n_s \nonumber \\
& & + \|\nabla \tilde{X}\|_{\mathbb{H}^1(\Gamma_s)} \mathcal{O}(\|\nabla \tilde{X}-I\|_{\mathbb{H}^1(\Gamma_s)}), \label{rhs-nonlinear-1} \label{rhs-g}\\
H(\hat{u}^\pm, \hat{Z}) & = &  (\I - \cof \nabla \tilde{X})^T\hat{u}^\pm . \label{rhs-h}
\end{eqnarray}
\end{subequations}
In section~\ref{sec-nonlinear} we deduce regularity for the functions $F$, $G$ and $H$. Note that the condition $\hat{u}^+_{|\p \Omega} = 0$ on the outer boundary implies that $H(\hat{u}^+,Z)_{|\p \Omega} = 0$, and so we will consider in particular $H(\hat{u}^+,Z) \in \mathcal{U}_{\infty}(\Omega_s^+)$. In system~\eqref{sys-cyl2}, the linear part is on the left-hand-side, and the nonlinear part is represented by the right-hand-sides $F$, $G$ and $H$. Considering small data, the displacements $X-\Id$ remain small, and the nonlinearities $F$, $G$ and $H$ too (see section~\ref{sec-nonlinear}). Therefore we first study the stabilizability of the linearized system.

\section{On the linearized system} \label{sec-linear}

From~\eqref{sys-cyl2} we deduce the linearized system with $(u^\pm, p^\pm)$ and $Z$ as unknowns:
\begin{eqnarray}
-\divg \sigma(u^\pm, p^\pm) = 0 \quad \text{ and } \quad 
\divg u^\pm = 0 & & \text{in } \Omega_s^\pm \times (0,\infty), \nonumber \\
u^+ = 0 & &\text{on } \p \Omega \times (0,\infty), \nonumber \\
u^+ = u^- = \frac{\p Z}{\p t} - \lambda Z 
\quad \text{ and } \quad 
-\left[\sigma(u,p) \right]n_s = \mu \divg_{\Gamma_s} \nabla^{n_s}_{\Gamma_s} Z + G & &
\text{on } \Gamma_s \times (0,\infty), \label{slava0}\\
Z(\cdot,0 ) = Z_0 & & \text{on } \Gamma_s. \nonumber
\end{eqnarray}
As explained in section~\ref{sec-kernel}, the operator $\nabla^{n_s}_{\Gamma_s}$ is non-coercive, and we define a first feedback operator $Z\mapsto \mu\divg_{\Gamma_s} \big( (\tau_s \otimes \tau_s)\nabla_{\Gamma_s} Z\big) = \mu \divg_{\Gamma_s} \nabla^{\tau_s}_{\Gamma_s} Z = G$, such that the resulting elliptic operator in~\eqref{slava0} becomes
\begin{equation*}
Z\mapsto \mu \divg_{\Gamma_s} \nabla^{n_s}_{\Gamma_s} Z 
+ \mu \divg_{\Gamma_s} \nabla^{\tau_s}_{\Gamma_s} Z = 
\divg_{\Gamma_s} \nabla_{\Gamma_s} Z = \Delta_{\Gamma_s}  Z.
\end{equation*}
Given $T\in(0,\infty)$, this section is devoted to wellposedness and operator formulation for the following linear system:
\begin{subequations} \label{mainsyslin}
\begin{eqnarray}
-\divg \sigma(u^\pm, p^\pm) = 0 \quad \text{ and } \quad 
\divg u^\pm = 0 & & \text{in } \Omega_s^\pm \times (0,T), \\
u^+ = 0 & &\text{on } \p \Omega \times (0,T), \\
u^+ = u^- = \frac{\p Z}{\p t} - \lambda Z 
\quad \text{ and } \quad 
-\left[\sigma(u,p) \right]n_s = \mu \Delta_{\Gamma_s} Z + G & &
\text{on } \Gamma_s \times (0,T), \label{mainsyslin-4}\\
Z(\cdot,0 ) = Z_0 & & \text{on } \Gamma_s.
\end{eqnarray}
\end{subequations}
The data are assumed to satisfy $G \in \mathcal{G}_T(\Gamma_s)=\L^2(0,T;\mathbf{H}^{1/2}(\Gamma_s))$, $Z_0 \in \mathcal{Z}_0(\Gamma_s)= \HH^{2}(\Gamma_s)/\R^2$. Our approach consists in writing~\eqref{mainsyslin} as an abstract evolution equation with $Z$ as the only unknown. The other unknowns $(u^\pm,p^\pm)$ can be then deduced as solutions of standard Stokes problems with the Dirichlet boundary condition of~\eqref{mainsyslin-4}.

\subsection{A Poincar\'e-Steklov operator}\label{sec-Steklov}

For $G\in \mathbf{H}^{1/2}(\Gamma_s)$ given, in this subsection we are interested in the following linear transmission problem:
\begin{subequations} \label{syslin}
\begin{eqnarray}
-\divg \sigma(u^{\pm}, p^{\pm}) = 0 \quad \text{ and } \quad 
\divg u^{\pm} = 0 & & \text{in } \Omega_s^{\pm},\\
u^+ = 0 & & \text{on } \p \Omega, \\
\left[ u\right] = 0 
\quad \text{ and } \quad
-\left[\sigma(u,p)\right]n_s = G & & \text{on } \Gamma_s.\label{syslin-4}
\label{syslin-5}
\end{eqnarray}
\end{subequations}
Following the approach of~\cite{Court-JCAM}, the equality of velocities in~\eqref{syslin-4} leads us to introduce a boundary velocity $\phi \in \HH^{1/2}(\Gamma_s)$, and we obtain a weak solution of~\eqref{syslin} as a critical point of the following Lagrangian functional:
\begin{equation*}
\begin{array} {rcl}
\mathcal{L}(u^+,p^+,u^-,p^-, \lambda^+,\lambda^-,\phi) & = & 
2\nu\|\varepsilon(u^+)\|^2_{\mathbb{L}^2(\Omega_s^+)} + 2\nu
\|\varepsilon(u^-)\|^2_{\mathbb{L}^2(\Omega_s^-)} \\
& & - \langle p^+ ,\divg u^+\rangle_{\L^2(\Omega_s^+)} 
- \langle p^- , \divg u^-\rangle_{\L^2( \Omega_s^-)} \\
& & -\langle\lambda^+,u^+-\phi\rangle_{\HH^{-1/2}(\Gamma_s):\HH^{1/2}(\Gamma_s)} -
 \langle \lambda^- , u^- -\phi\rangle_{\HH^{-1/2}(\Gamma_s):\HH^{1/2}(\Gamma_s)} \\
& &   - \langle G , \phi \rangle_{\HH^{-1/2}(\Gamma_s):\HH^{1/2}(\Gamma_s)} .
\end{array}
\end{equation*}
Note that the variable $\phi$ also plays the role of a multiplier for the transmission condition~\eqref{syslin-5}. Let us introduce
\begin{equation*}
\begin{array} {l}
\mathbf{V}^+ = \displaystyle \left\{ v \in \HH^1(\Omega_s^+) \mid v_{|\p \Omega} = 0 \right\}, \quad  
\mathbf{V}^- = \HH^1(\Omega_s^-)/\R^2, \quad Q^\pm = \L^2(\Omega_s^\pm) /\R, \\
\mathbf{W} = \displaystyle \left\{ v\in \HH^{1/2}(\Gamma_s) \mid 
\langle v, n \rangle_{\LL^2(\Gamma_s)} = 0 \right\}.
\end{array}
\end{equation*}
Relying on the Korn's inequality and the Petree-Tartar lemma, we equip $\mathbf{V}^\pm$ with the norms $\| v\|_{\mathbf{V}^\pm} := \| \varepsilon(v)\|_{\mathbb{L}^2(\Omega_s^\pm)}$. For the sake of brevity we denote
\begin{eqnarray*}
& & \mathfrak{u}  =  (u^+,p^+,u^-,p^-,\lambda^+,\lambda^-,\phi),\quad \mathfrak{v}  =  (v^+,q^+,v^-,q^-,\mu^+,\mu^-,\varphi),\\
& & \mathfrak{V} = \mathbf{V}^+\times Q^+\times \mathbf{V}^-\times Q^-\times \mathbf{W}' \times \mathbf{W}' \times \mathbf{W}.
\end{eqnarray*}
A weak solution of~\eqref{syslin} satisfies the variational formulation given by the first order optimality condition for functional~$\mathcal{L}$:
\begin{eqnarray} 
\begin{array}{l}
 \text{Find $\mathfrak{u} \in \mathfrak{V}$, such that for all $\mathfrak{v} \in \mathfrak{V} $:}\\
 \label{FV0}
\left\{\begin{array} {rcl} 
\displaystyle \langle\sigma(u^\pm,p^\pm),\varepsilon(v^\pm)\rangle_{\LLL^2(\Omega_s^\pm)}
-\langle \lambda^\pm, v^\pm\rangle_{\mathbf{W}';\mathbf{W}} = 0,  & & 
\displaystyle -\langle q^\pm ,\divg u^\pm \rangle_{\L^2(\Omega_s^\pm)} = 0,\\
\displaystyle -\langle\mu^\pm , u^\pm-\Phi \rangle_{\mathbf{W}';\mathbf{W}} = 0, & & 
\displaystyle  \langle \lambda^+ + \lambda^- - G, \varphi \rangle_{\mathbf{W}';\mathbf{W}} = 0.
\end{array} \right.
\end{array}
\end{eqnarray}
By integration by parts, we easily see that at the optimality we have $\lambda^\pm = \sigma(u^\pm,p^\pm)n^\pm$. We rewrite the variational problem~\eqref{FV0} more compactly, as follows:
\begin{equation*}
 \text{Find $\mathfrak{u} \in \mathfrak{V}$ such that} \quad
\mathcal{M}(\mathfrak{u};\mathfrak{v}) = \mathcal{G}(\mathfrak{v})\quad \forall \mathfrak{v} \in \mathfrak{V},
\end{equation*}
\begin{equation*}
\begin{array}{rcl}
\text{where}\quad \mathcal{M}(\mathfrak{u};\mathfrak{v}) &  :=  & 
2\nu \langle\varepsilon(u^+),\varepsilon(v^+)\rangle_{\mathbb{L}^2(\Omega_s^+)} + 2\nu \langle\varepsilon(u^-):\varepsilon(v^-)\rangle_{\mathbb{L}^2(\Omega_s^-)} \\
& & - \langle p^+,\divg v^+\rangle_{\L^2(\Omega_s^+)} 
-\langle q^+,\divg u^+ \rangle_{\L^2(\Omega_s^+)} 
- \langle p^-,\divg v^-\rangle_{\L^2(\Omega_s^-)}
- \langle q^-,\divg u^- \rangle_{\L^2(\Omega_s^-)} \\
& & - \langle \lambda^+ ,v^+-\varphi \rangle_{\mathbf{W}';\mathbf{W}} 
- \langle \mu^+ , u^+ - \Phi\rangle_{\mathbf{W}';\mathbf{W}} 
- \langle\lambda^- ,v^--\varphi\rangle_{\mathbf{W}';\mathbf{W}} 
- \langle \mu^-, u^- - \Phi\rangle_{\mathbf{W}';\mathbf{W}}, \\
 \mathcal{G}(\mathfrak{v}) & := & \left\langle G, \varphi \right\rangle_{\mathbf{W}';\mathbf{W}}.
\end{array}
\end{equation*}
The existence and uniqueness of a solution for~\eqref{FV0} is equivalent to the Ladyzhenskaya-Babu\v{s}ka-Brezzi inf-sup condition. In that sense we state the following result: 
\begin{proposition} \label{propinfsup}
There exists a constant $C >0$ such that
\begin{eqnarray*}
\inf_{\mathfrak{u} \in \mathfrak{V}\setminus \{0\}} \sup_{\mathfrak{v} \in \mathfrak{V}\setminus \{0\}} \frac{\mathcal{M}(\mathfrak{u};\mathfrak{v})}{\| \mathfrak{u} \|_{\mathfrak{V}} \| \mathfrak{v} \|_{\mathfrak{V}}} & \geq & C.
\end{eqnarray*}
\end{proposition} 
The proof of this proposition is given in Appendix~\ref{appendix2}. The consequence of this result is the existence and uniqueness of a weak solution for system~\eqref{syslin}.

\begin{corollary} \label{coroinfsup}
Assume that $G \in \mathbf{W}$. System~\eqref{syslin} admits a unique solution $(u^+,p^+,u^-,p^-)$ in $\mathbf{V}^+, Q^+, \mathbf{V}^- \times Q^-$. Moreover, there exists a constant $C >0$, depending only on $\Omega_s^+$ and $\Omega_s^-$, such that
\begin{eqnarray*}
\| u^+\|_{\mathbf{H}^1(\Omega_s^+)} + \| p^+\|_{\L^2(\Omega_s^+)} + \| u^-\|_{\mathbf{H}^1(\Omega_s^-)} + \| p^-\|_{\L^2(\Omega_s^-)} & \leq & 
C \| G \|_{\mathbf{H}^{-1/2}(\Gamma_s)}.
\end{eqnarray*}
\end{corollary}

The proof of Corollary~\ref{coroinfsup} is also given in Appendix~\ref{appendix2}. Thu, considering the trace on $\Gamma_s$ of the solution $u^\pm$ of system~\eqref{syslin}, we have defined the operator
\begin{equation} \label{def-Steklov}
\begin{array} {rccl}
\mathcal{P}_{\Gamma_s}: & \mathbf{W}' & \rightarrow & \mathbf{W} \\
& G & \mapsto & u^\pm_{|\Gamma_s}
\end{array}
\end{equation}
mapping the jump condition in the right side of~\eqref{syslin-5} to the velocity trace on $\Gamma_s$. From Proposition~\ref{propinfsup} and Corollary~\ref{coroinfsup}, we can deduce more regularity for system~\eqref{syslin}, and consequently for operator~$\mathcal{P}_{\Gamma_s}$ restricted to $\HH^{1/2}(\Gamma_s)$.

\begin{proposition} \label{prop-op-reg}
Assume that $\p \Omega$ is of class $C^1$, and that $G \in \mathbf{H}^{1/2}(\Gamma_s)$. Then system~\eqref{syslin} admits a unique solution $(u^+, p^+, u^-, p^-)$ in $\mathbf{H}^2(\Omega_s^+) \times \H^{1}(\Omega_s^+)\times \mathbf{H}^2(\Omega_s^-) \times \H^{1}(\Omega_s^-)$, and it satisfies the estimate
\begin{eqnarray*}
\| u^+\|_{\mathbf{H}^2(\Omega_s^+)} + \| p^+\|_{\H^1(\Omega_s^+)/\R}
+ \| u^-\|_{\mathbf{H}^2(\Omega^-)} + \| p^-\|_{\H^1(\Omega_s^-)/\R} & \leq &
C\| G \|_{\mathbf{H}^{1/2}(\Gamma_s)},
\end{eqnarray*}
where the constant $C >0$ depends only $\Omega_s^+$ and $\Omega_s^-$.
\end{proposition}

\begin{proof}
The proof of the regularity theorem of~\cite[Theorem~9.19, p.~278]{Hackbusch}, can be repeated in our context, and the regularity of $u^\pm$ in $\HH^2(\Omega_s^\pm)$ follows, as well as those of $p^\pm$ in $\H^1(\Omega_s^\pm)/\R$.
\end{proof}

Therefore $\mathcal{P}_{\Gamma_s}$ maps $\HH^{1/2}(\Gamma_s)$ onto $\HH^{3/2}(\Gamma_s)$. Finally, we state the following properties for the Neumann-to-Dirichlet operator $\mathcal{P}_{\Gamma_s}$:

\begin{proposition} \label{prop-PS}
Operator~$\mathcal{P}_{\Gamma_s}$ is self-adjoint and non-negative, and $\mathrm{Ker}(\mathcal{P}_{\Gamma_s}) = \mathrm{span}(n_s)$.
\end{proposition}

\begin{proof}
Let be $G_1$, $G_2 \in \mathbf{H}^{-1/2}(\Gamma_s)$, and denote by $(u^\pm_1,p^\pm_1)$ and $(u^\pm_2,p^\pm_2)$ the solutions of system~\eqref{syslin} corresponding to $G_1$ and $G_2$ respectively. By integration by parts, we obtain
\begin{eqnarray*}
\langle G_2,\mathcal{P}_{\Gamma_s}G_1 \rangle_{\mathbf{H}^{-1/2}(\Gamma_s);\mathbf{H}^{1/2}(\Gamma_s)} & = &  
\langle -\left[ \sigma(u_2,p_2)\right]n_s , u^\pm_1\rangle_{\mathbf{H}^{-1/2}(\Gamma_s);\mathbf{H}^{1/2}(\Gamma_s)} \\
& = &\langle  \sigma(u^+_2,p^+_2)n^+_s + \sigma(u_2^-,p_2^-)n_s^- , u^\pm_1\rangle_{\mathbf{H}^{-1/2}(\Gamma_s);\mathbf{H}^{1/2}(\Gamma_s)}\\
& = & 2\nu\left( \langle \varepsilon(u^+_1) , \varepsilon(u^+_2) \rangle_{\LLL^2(\Omega_s^+)}
+ \langle \varepsilon(u^-_1) , \varepsilon(u^-_2) \rangle_{\LLL^2(\Omega_s^-)}
\right).
\end{eqnarray*}
This symmetric form shows that $\mathcal{P}_{\Gamma_s}$ is self-adjoint. Further, with $G = G_1 = G_2$, we have 
\begin{equation*}
\langle G,\mathcal{P}_{\Gamma_s}G \rangle_{\mathbf{H}^{-1/2}(\Gamma_s);\mathbf{H}^{1/2}(\Gamma_s)}  = 
2\nu \left(\|\varepsilon(u_1^+)\|^2_{\mathbf{L}^2(\Omega_s^+)} + \|\varepsilon(u_1^-)\|^2_{\mathbf{L}^2(\Omega_s^-)} \right)
\geq 0,
\end{equation*}
and $\langle G, \mathcal{P}_{\Gamma_s}G \rangle_{\mathbf{H}^{-1/2}(\Gamma_s);\mathbf{H}^{1/2}(\Gamma_s)} = 0$ if and only if~$u_1^\pm \equiv 0$ from Lemma~\ref{lemma-Korn}. 
Finally, Lemma~\ref{lemma-Temam} also describes the kernel of~$\mathcal{P}_{\Gamma_s}$, finishing the proof.
\end{proof}

\subsection{The semi-homogeneous system} \label{sec-linSH}
Using the operator~$\mathcal{P}_{\Gamma_s}$, we rewrite system~\eqref{mainsyslin} with $\lambda=0$ as the following abstract evolution equation
\begin{equation}
\frac{\p Z}{\p t}  - \mathcal{P}_{\Gamma_s}(\mu\divg_{\Gamma_s} \nabla_{\Gamma_s} Z) = \mathcal{P}_{\Gamma_s} G 
\quad\text{ in } (0,T), \qquad Z(0) = Z_0, \label{eq-mani}
\end{equation}
with $Z_0 \in \HH^{2}(\Gamma_s)/\R^2$. Composing~\eqref{eq-mani} by $\nabla_{\Gamma_s}$, we obtain an equation dealing with $\nabla_{\Gamma_s}Z$ as unknown:
\begin{equation}
\frac{\p \nabla_{\Gamma_s} Z}{\p t}  - \mu \nabla_{\Gamma_s}\mathcal{P}_{\Gamma_s}(\divg_{\Gamma_s}\nabla_{\Gamma_s} Z) =
\nabla_{\Gamma_s}\mathcal{P}_{\Gamma_s} G \quad \text{ in } (0,T), \qquad \nabla_{\Gamma_s}Z(0) = \nabla_{\Gamma_s}Z_0.
\label{op-formu-super}
\end{equation}
The interest of this formulation is that the following linear operator is self-adjoint:
\begin{equation}
\begin{array} {rccl}
\mathcal{A}: & \nabla_{\Gamma_s}\HH^{5/2}(\Gamma_s) =: D(\mathcal{A}) & \rightarrow & \nabla_{\Gamma_s}\HH^{3/2}(\Gamma_s) \\
& \nabla_{\Gamma_s} Z & \mapsto & \mu\nabla_{\Gamma_s}
\mathcal{P}_{\Gamma_s}(\divg_{\Gamma_s}\nabla_{\Gamma_s} Z) 
\end{array}\label{def-operator}
\end{equation}
For~$\nabla_{\Gamma_s} Z \in \nabla_{\Gamma_s}\HH^{\kappa}(\Gamma_s)$ we consider the norm $\| \nabla_{\Gamma_s} Z \|_{\mathds{H}^{\kappa-1}(\Gamma_s)}$, with $\kappa\geq 1$. Let us derive the fundamental properties of $\mathcal{A} = \mu \nabla_{\Gamma_s} \mathcal{P}_{\Gamma_s} \divg_{\Gamma_s}$.

\begin{proposition} \label{prop-analytic}
The operator $(\mathcal{A}, D(\mathcal{A}))$ is self-adjoint, dissipative, and thus infinitesimal generator of an analytic semigroup of contraction on $\nabla_{\Gamma_s}\HH^{5/2}(\Gamma_s)$.
\end{proposition}

\begin{proof}
The Green's formula~\eqref{Stokes-mani2} shows that $\nabla_{\Gamma_s}^\ast = -\divg_{\Gamma_s}$. 
For~$Z_1$,\, $Z_2 \in \HH^{5/2}(\Gamma_s)$ we have
\begin{equation*}
\begin{array}{rcl}
\langle \mathcal{A}\nabla_{\Gamma_s} Z_1 , \nabla_{\Gamma_s} Z_2 \rangle_{\mathds{L}^2(\Gamma_s)} & = &
\mu 
\langle \mathcal{P}_{\Gamma_s}(\divg_{\Gamma_s}  \nabla_{\Gamma_s} Z_1) , (\nabla_{\Gamma_s})^{\ast} \nabla_{\Gamma_s} Z_2 \rangle_{\LL^2(\Gamma_s)}\\
& = & -\mu 
\langle \mathcal{P}_{\Gamma_s}(\divg_{\Gamma_s}  \nabla_{\Gamma_s} Z_1) , \divg_{\Gamma_s}\nabla_{\Gamma_s} Z_2 \rangle_{\LL^2(\Gamma_s)}\\
& = & -\mu 
\langle \mathcal{P}_{\Gamma_s}(\Delta_{\Gamma_s} Z_1) , \Delta_{\Gamma_s} Z_2 \rangle_{\LL^2(\Gamma_s)}.
\end{array}
\end{equation*}
Using Proposition~\ref{prop-PS}, we see that $\mathcal{A}$ is self-adjoint and dissipative.  Consequently, from~\cite[Chapter~1, Proposition~2.11]{Bensoussan}, the operator $(\mathcal{A},D(\mathcal{A}))$ generates an analytic semigroup of contraction.
\end{proof}


\begin{proposition} \label{prop-rescompact}
The resolvent of~$\mathcal{A}$ is compact.
\end{proposition}

\begin{proof} 
Let us show that there exists $\lambda \in \R$ such that $\lambda \Id - \mathcal{A}$ is invertible. Let be $\nabla_{\Gamma_s}W\in \nabla_{\Gamma_s} \HH^{5/2}(\Gamma_s) \subset \mathds{H}^{3/2}(\Gamma_s)$ and consider the following system:
\begin{equation}
\lambda \nabla_{\Gamma_s}Z - \mathcal{A}\nabla_{\Gamma_s}Z = \nabla_{\Gamma_s} W. \label{eq-resolvent}
\end{equation}
Taking the scalar product of this equation by $\nabla_{\Gamma_s}Z$, and using the Green's formula~\eqref{Stokes-mani2}, we obtain
\begin{equation*}
\lambda \| \nabla_{\Gamma_s}Z \|_{\mathds{L}^2(\Gamma_s)}^2 
+\mu \langle \mathcal{P}_{\Gamma_s}(\divg_{\Gamma_s}\nabla_{\Gamma_s} Z) , \divg_{\Gamma_s}\nabla_{\Gamma_s}Z\rangle_{\LL^2(\Gamma_s)} =  
\langle \nabla_{\Gamma_s} W , \nabla_{\Gamma_s} Z \rangle_{\mathds{L}^2(\Gamma_s)}.
\end{equation*}
Further, from the definition of~$\mathcal{P}_{\Gamma_s}$, we introduce $u^\pm$ the solution of
\begin{equation}
\left\{ \begin{array} {rcl}
-\divg \sigma(u^\pm,p^\pm) = 0 \quad \text{ and } \quad
\divg u^\pm = 0 & & \text{in } \Omega_s^\pm, \\
u^+ = 0 & & \text{on } \p \Omega, \\
\left[ u\right] = 0 \quad \text{ and } \quad 
-\left[\sigma(u,p)\right]n_s = \mu \divg_{\Gamma_s} \nabla_{\Gamma_s} Z & & \text{on } \Gamma_s,
\end{array} \right.
\label{sys-Lax}
\end{equation}
and by integration by parts we deduce
\begin{equation}
\lambda \| \nabla_{\Gamma_s}Z \|_{\mathds{L}^2(\Gamma_s)}^2 +
2\nu\left( \|\varepsilon(u^+)\|^2_{\mathbb{L}^2(\Omega^+)}+
\|\varepsilon(u^-)\|^2_{\mathbb{L}^2(\Omega^-)}
\right) = \langle \nabla_{\Gamma_s} W , \nabla_{\Gamma_s} Z \rangle_{\mathds{L}^2(\Gamma_s)}. \label{est-Lax}
\end{equation}
Introduce the bilinear form
\begin{equation*}
\begin{array} {rccl}
a: & \nabla_{\Gamma_s} \HH^1(\Gamma_s) \times \nabla_{\Gamma_s} \HH^1(\Gamma_s) & \rightarrow & \R \\
 & (\nabla_{\Gamma_s} Z_1 , \nabla_{\Gamma_s} Z_2) & \mapsto & 
\lambda \langle \nabla_{\Gamma_s}Z_1, \nabla_{\Gamma_s}Z_2 \rangle_{\mathds{L}^2(\Gamma_s)}\\[5pt]
& & & +
2\nu\left( \langle\varepsilon(u_1^+),\varepsilon(u_2^+\rangle_{\mathbb{L}^2(\Omega^+)}+
\langle\varepsilon(u_1^-), \varepsilon(u_2^-)\rangle_{\mathbb{L}^2(\Omega^-)} \right),
\end{array}
\end{equation*}
where $u_1^\pm$ and $u_2^\pm$ are solutions of~\eqref{sys-Lax} corresponding to $Z=Z_1$ and $Z= Z_2$, respectively. It satisfies
\begin{equation*}
a(\nabla_{\Gamma_s}Z,\nabla_{\Gamma_s}Z) \geq \lambda \| \nabla_{\Gamma_s} Z \|_{\mathds{L}^2(\Gamma_s)}^2 ,
\end{equation*}
and thus, for $\lambda>0$, it is coercive. Introduce the linear form $b:\nabla_{\Gamma_s} \HH^1(\Gamma_s) \ni\nabla_{\Gamma_s} Z \mapsto \langle \nabla_{\Gamma_s} W , \nabla_{\Gamma_s} Z \rangle_{\mathds{L}^2(\Gamma_s)}$, which is clearly continuous. We consider the variational formulation of~\eqref{eq-resolvent} as follows:
\begin{equation}
\text{Find $\nabla_{\Gamma_s}Z \in \nabla_{\Gamma_s}\HH^1(\Gamma_s)$ such that $a(\nabla_{\Gamma_s}Z, \nabla_{\Gamma_s} \tilde{Z}) = b(\nabla_{\Gamma_s}\tilde{Z})$ for all $\nabla_{\Gamma_s}\tilde{Z} \in \nabla_{\Gamma_s}\HH^1(\Gamma_s)$.}
\label{form-var}
\end{equation}
From the Lax-Milgram theorem there exists a unique $\nabla_{\Gamma_s}Z \in \nabla_{\Gamma_s}\mathbf{H}^1(\Gamma_s)$ solution of~\eqref{form-var}, and so satisfying~\eqref{eq-resolvent}. Note that it is sufficient to assume $\nabla_{\Gamma_s} W \in \nabla_{\Gamma_s}\HH^1(\Gamma_s)$ for obtaining $\nabla_{\Gamma_s}Z \in \nabla_{\Gamma_s} \HH^1(\Gamma_s)$. Choosing $\nabla_{\Gamma_s}\tilde{Z}=  \nabla_{\Gamma_s} Z$ in~\eqref{form-var}, with the Cauchy-Schwarz inequality we get the following estimate
\begin{equation} \label{est-resolvent1}
\| \nabla_{\Gamma_s} Z\|_{\mathds{L}^2(\Gamma_s)} \leq \frac{C}{\lambda} 
\|\nabla_{\Gamma_s} W\|_{\mathds{L}^2(\Gamma_s)}.
\end{equation}
Next, if we assume $\nabla_{\Gamma_s} W \in \nabla_{\Gamma_s} \HH^{2}(\Gamma_s)$, let us prove that $\nabla_{\Gamma_s} Z \in \nabla_{\Gamma_s} \HH^{2}(\Gamma_s)$ too. Since~$\nabla_{\Gamma_s}Z \in \nabla_{\Gamma_s} \HH^1(\Gamma_s)$, the identity
\begin{equation*}
-\mathcal{A} \nabla_{\Gamma_s}Z = \nabla_{\Gamma_s} W - \lambda \nabla_{\Gamma_s}Z
\end{equation*}
yields
\begin{equation}
\| \nabla_{\Gamma_s} \mathcal{P}_{\Gamma_s} \divg_{\Gamma_s} \nabla_{\Gamma_s} Z\|_{\mathds{L}^2(\Gamma_s)} = 
\|\mathcal{A} \nabla_{\Gamma_s} Z\|_{\mathds{L}^2(\Gamma_s)}
\leq \|\nabla_{\Gamma_s} W\|_{\mathds{L}^2(\Gamma_s)} + 
\lambda \|\nabla_{\Gamma_s} Z \|_{\mathds{L}^2(\Gamma_s)}
\leq C\| \nabla_{\Gamma_s} W\|_{\mathds{L}^2(\Gamma_s)}, 
\label{slava}
\end{equation}
where we used~\eqref{est-resolvent1}. Therefore $\mathcal{P}_{\Gamma_s} \divg_{\Gamma_s} \nabla_{\Gamma_s} Z \in \HH^1(\Gamma_s)$, which means that $u^\pm_{|\Gamma_s} \in \HH^1(\Gamma_s)$ in system~\eqref{sys-Lax}, leading to $u^\pm \in \HH^{3/2}(\Omega_s^\pm)$, and consequently to $-\left[\sigma(u,p)\right]n_s = \mu \divg_{\Gamma_s} \nabla_{\Gamma_s} Z 
\in \LL^2(\Gamma_s)$. Furthermore, using classical elliptic estimates for Stokes problems with prescribed non-homogeneous Dirichlet boundary conditions (see~\cite[Lemma~6.1, Chapter~IV]{Galdi}, the estimates in fractional spaces can be obtained by linear interpolation), we estimate
\begin{equation*}
\begin{array} {rcl}
\| \Delta_{\Gamma_s}Z \|_{\LL^2(\Gamma_s)} = \| \divg_{\Gamma_s} \nabla_{\Gamma_s} Z\|_{\LL^2(\Gamma_s)} & = & 
\| \left[\sigma(u,p)\right]n_s \|_{\LL^2(\Gamma_s)} \\
& \leq  & C \left(\| u\|_{\mathbf{H}^{3/2}(\Omega_s^\pm)/\R^2} + \| \nabla p\|_{\H^{-1/2}(\Omega_s^\pm)/\R^2} \right) \\[5pt]
& \leq & C\| u^\pm_{|\Gamma_s} \|_{\mathbf{H}^1(\Gamma_s)/\R^2} = 
C\| \mathcal{P}_{\Gamma_s} \divg_{\Gamma_s} \nabla_{\Gamma_s}Z \|_{\mathbf{H}^1(\Gamma_s)/\R^2} \\
& \leq & C\|\nabla_{\Gamma_s} \mathcal{P}_{\Gamma_s} \divg_{\Gamma_s} \nabla_{\Gamma_s}Z \|_{\mathds{L}^2(\Gamma_s)} 
= C\|\mathcal{A} \nabla_{\Gamma_s}Z \|_{\mathds{L}^2(\Gamma_s)}
\\
& \leq & C\| \nabla_{\Gamma_s} W\|_{\mathds{L}^2(\Gamma_s)},
\end{array}
\end{equation*}
where we have used the Poincar\'e inequality~\eqref{eq-rigidity} and estimate~\eqref{slava} above. Next, using~\eqref{eq-garding}, we deduce
\begin{equation}
\| \nabla_{\Gamma_s} Z \|_{\mathds{H}^1(\Gamma_s)} \leq 
\|Z\|_{\HH^2(\Gamma_s)/\R^2} \leq 
C\|\nabla_{\Gamma_s} W \|_{\mathds{L}^2(\Gamma_s)}.
\label{slava1}
\end{equation}
We proceed similarly in order to estimate
\begin{equation*}
\begin{array} {rcl}
\| \nabla_{\Gamma_s} Z\|_{\mathds{H}^2(\Gamma_s)} \leq \|Z\|_{\HH^3(\Gamma_s)} \leq 
\| \Delta_{\Gamma_s}Z \|_{\HH^1(\Gamma_s)}  & = & 
\| \left[\sigma(u,p)\right]n_s \|_{\HH^1(\Gamma_s)} \\
& \leq  & C \left(\| u\|_{\mathbf{H}^{5/2}(\Omega_s^\pm)/\R^2} + \| \nabla p\|_{\H^{1/2}(\Omega_s^\pm)/\R^2} \right) \\[5pt]
& \leq & C\| u^\pm_{|\Gamma_s} \|_{\mathbf{H}^2(\Gamma_s)/\R^2} = 
C\| \mathcal{P}_{\Gamma_s} \divg_{\Gamma_s} \nabla_{\Gamma_s}Z \|_{\mathbf{H}^2(\Gamma_s)/\R^2} \\
& \leq & C\|\nabla_{\Gamma_s} \mathcal{P}_{\Gamma_s} \divg_{\Gamma_s} \nabla_{\Gamma_s}Z \|_{\mathds{H}^1(\Gamma_s)} 
= C\|\mathcal{A} \nabla_{\Gamma_s}Z \|_{\mathds{H}^1(\Gamma_s)}.
\end{array}
\end{equation*}
Again, the identity $-\mathcal{A} \nabla_{\Gamma_s}Z = \nabla_{\Gamma_s} W - \lambda \nabla_{\Gamma_s}Z$ yields
\begin{equation*}
\|\mathcal{A} \nabla_{\Gamma_s} Z\|_{\mathds{H}^1(\Gamma_s)}
\leq \|\nabla_{\Gamma_s} W\|_{\mathds{H}^1(\Gamma_s)} + 
\lambda \|\nabla_{\Gamma_s} Z \|_{\mathds{H}^1(\Gamma_s)}
\leq C(1+\lambda)\| \nabla_{\Gamma_s} W\|_{\mathds{H}^1(\Gamma_s)},
\end{equation*}
where we have used~\eqref{slava1}, and thus we deduce
\begin{equation}
\| \nabla_{\Gamma_s} Z\|_{\mathds{H}^2(\Gamma_s)} \leq
C(1+\lambda)\| \nabla_{\Gamma_s} W\|_{\mathds{H}^1(\Gamma_s)}.
\label{slava2}
\end{equation}
Combining~\eqref{slava1} and~\eqref{slava2}, by interpolation we obtain
\begin{equation*}
\| \nabla_{\Gamma_s} Z\|_{\mathds{H}^{3/2}(\Gamma_s)} \leq
C(1+\lambda)^{1/2}\| \nabla_{\Gamma_s} W\|_{\mathds{H}^{1/2}(\Gamma_s)},
\end{equation*}
which completes the proof.
\end{proof}

We deduce wellposedness for system~\eqref{eq-mani}.
\begin{theorem} \label{th-syslin}
Let be $0<T<\infty$. For $G \in \mathcal{G}_T(\Gamma_s)$ and $Z_0 \in \mathcal{Z}_0(\Gamma_s)$, the following system
\begin{equation*}
\begin{array} {rcl}
-\divg(u^\pm, p^\pm) = 0, \quad \text{ and } \quad
\divg u^\pm = 0 & &  \text{in } \Omega_s^\pm \times (0,T), \\
u^+ = 0 & & \text{on } \p \Omega \times (0,T), \\
u^\pm = \displaystyle \frac{\p Z}{\p t} \quad \text{ and } \quad 
-\left[\sigma(u,p)\right]n_s = \mu \divg_{\Gamma_s} \nabla_{\Gamma_s} Z + G & & 
\text{on } \Gamma_s \times (0,T), \\
Z(\cdot,0) = Z_0 & & \text{on } \Gamma_s,
\end{array}
\end{equation*}
admits a unique solution $Z\in \mathcal{Z}_T(\Gamma_s)$. Moreover, there exists a constant $C>0$, non-decreasing with respect to~$T$, such that
\begin{equation*}
\| Z \|_{\mathcal{Z}_T(\Gamma_s)} + 
\| Z \|_{\L^{\infty}(0,T;\HH^2(\Gamma_s))}
\leq C \left(
\|Z_0\|_{\mathcal{Z}_0(\Gamma_s)} + \|G \|_{\mathcal{G}_T(\Gamma_s)}
\right).
\end{equation*}
\end{theorem}

\begin{proof}
See for example~\cite[Proposition~3.3]{Tucsnak04}, that provides from~\eqref{op-formu-super} existence and uniqueness of~$\nabla_{\Gamma_s} Z$ satisfying
\begin{equation*}
\|\nabla_{\Gamma_s} Z \|_{\L^{2}(0,T;\mathds{H}^{3/2}(\Gamma_s))\cap \H^{1}(0,T;\mathds{H}^{1/2}(\Gamma_s))} 
+ \|\nabla_{\Gamma_s} Z \|_{\L^{\infty}(0,T;\mathds{H}^1(\Gamma_s))}
\leq C \left(
\|Z_0\|_{\mathcal{Z}_0(\Gamma_s)} + \|G \|_{\mathcal{G}_T(\Gamma_s)}
\right).
\end{equation*}
From $\nabla_{\Gamma_s} Z$, we retrieve $Z$ up to a constant, by using estimates~\eqref{eq-rigidity}-\eqref{eq-rigidity2} of Proposition~\ref{prop-rigidity}, leading to the announced result.
\end{proof}

\subsection{The non-homogeneous system} \label{sec-lifting}
We now address system~\eqref{mainsyslin} in finite-time horizon, for general right-hand-sides, and any $\lambda \geq 0$.
\begin{equation}
\begin{array} {rcl}
-\divg \sigma(u^\pm, p^\pm) = F^\pm  \quad \text{ and } \quad 
\divg u^\pm = \divg H^\pm & & \text{in } \Omega_s^\pm \times (0,T), \\
u^+ = 0 & &\text{on } \p \Omega \times (0,T), \\
u^+ = u^-  = \displaystyle \frac{\p Z}{\p t} - \lambda Z 
\quad \text{ and } \quad
-\left[\sigma(u,p) \right]n_s = \mu \divg_{\Gamma_s} \nabla_{\Gamma_s} Z + G 
& &
\text{on } \Gamma_s \times (0,T), \\
Z(\cdot,0 ) = Z_0 & & \text{on } \Gamma_s.
\end{array} \label{syslin-a}
\end{equation}
With $0<T<\infty$, we assume that $F^\pm \in \mathcal{F}_T(\Omega_s^\pm)$, $
G \in \mathcal{G}_T(\Gamma_s)$, $H^\pm \in \mathcal{U}_T(\Omega_s^\pm)$, and $Z_0 \in \mathcal{Z}_0(\Gamma_s)$. We use a {\it lifting} method: Let us describe a solution of~\eqref{syslin-a} as
\begin{equation*}
u^\pm = v^\pm + w^\pm, \quad p^\pm = q^\pm + \pi^\pm,
\end{equation*}
where $(w^\pm, \pi^\pm)$ are solutions of the following Stokes problems
\begin{equation}
\begin{array} {rcl}
-\divg(w^\pm, \pi^\pm) = F^\pm  \quad \text{ and } \quad 
\divg w^\pm = \divg H^\pm & & \text{in } \Omega_s^\pm \times (0,T), \\
w^+ = 0 & &\text{on } \p \Omega \times (0,T), \\
w^+ = w^- = 0 & & \text{on } \Gamma_s \times (0,T),
\end{array}
\label{syslin-b}
\end{equation}
and $(v^\pm, q^\pm)$ satisfy
\begin{equation}
\begin{array} {rcl}
-\divg \sigma(v^\pm, q^\pm) = 0 \quad \text{ and } \quad 
\divg v^\pm = 0 & & \text{in } \Omega_s^\pm \times (0,T), \\
v^+ = 0 & &\text{on } \p \Omega \times (0,T), \\
v^+ = v^- = \displaystyle \frac{\p Z}{\p t} - \lambda Z
\quad \text{ and } \quad
-\left[\sigma(v,q) \right]n_s = \mu \divg_{\Gamma_s}\nabla_{\Gamma_s} Z + G 
 + \left[\sigma(w,\pi) \right]n_s & &
\text{on } \Gamma_s \times (0,T), \\
Z(\cdot,0 ) = Z_0 & & \text{on } \Gamma_s.
\end{array}\label{syslin-c}
\end{equation}
Note that the equations of~\eqref{syslin-b} are uncoupled, as both Stokes systems can be considered in $\Omega_s^+$ and $\Omega_s^-$ independently. By considering $\overline{w}^\pm := w^\pm - H^\pm$, we eliminate the non-homogeneous divergence condition, and we reduce~\eqref{syslin-b} to standard Stokes problems with non-homogeneous Dirichlet condition:
\begin{equation}
\begin{array} {rcl}
-\divg \sigma(\overline{w}^\pm, \pi^\pm) = F^\pm + 2\nu \divg\varepsilon(H^\pm) 
\quad \text{ and } \quad 
\divg \overline{w}^\pm = 0 & & \text{in } \Omega_s^\pm, \\
\overline{w}^+ = 0 & &\text{on } \p \Omega, \\
\overline{w}^\pm = - H^\pm & & \text{on } \Gamma_s.
\end{array}
\label{syslin-b-bis}
\end{equation}
It is well-known that for almost every $t\in (0,T)$ there exists a unique solution $(\overline{w}^\pm,\pi^\pm)$ satisfying
\begin{equation*}
\| \overline{w}^\pm \|^2_{\HH^{2}(\Omega_s^\pm)} 
+ \| \pi \|^2_{\HH^1(\Omega_s^\pm)/\R} \leq
C\left(
\|F^\pm \|^2_{\LL^2(\Omega_s^\pm)} 
+ \| \divg \varepsilon(H^\pm)\|^2_{\LL^2(\Omega_s^\pm)}
+ \| H^\pm \|^2_{\HH^{3/2}(\Gamma_s)}
\right).
\end{equation*}
See for example~\cite[Lemma~6.1, Chapter~IV]{Galdi}. We deduce the same estimate for $(w^\pm,\pi) = (\overline{w}^\pm + H^\pm, \pi)$. Further, integrating in time this estimate, it follows from the trace theorem the following estimate
\begin{equation}
\| \left[ \sigma(w,\pi) \right]n_s \|_{\mathcal{G}_T(\Gamma_s)} \leq 
C\left(
\|F^\pm \|_{\mathcal{F}_T(\Omega_s^\pm)} 
+ \| H^\pm\|_{\mathcal{U}_T(\Omega_s^\pm)}
\right). \label{est-wpi}
\end{equation}
On the other side, equation~\eqref{syslin-c} admits the following operator formulation
\begin{equation*}
\frac{\p Z}{\p t} - \lambda Z - \mathcal{P}_{\Gamma_s}(\mu \divg_{\Gamma_s} \nabla_{\Gamma_s} Z) = 
\mathcal{P}_{\Gamma_s}\big(G  + \left[\sigma(w,\pi) \right]n_s  \big) 
\quad \text{in (0,T),} \quad Z(0) = Z_0.
\end{equation*}
Following Proposition~\ref{prop-rescompact} and Theorem~\ref{th-syslin}, system~\eqref{syslin-c} admits a unique solution, satisfying
\begin{equation*}
\| Z \|_{\mathcal{Z}_T(\Gamma_s)}
\leq C\left(
\|Z_0\|_{\mathcal{Z}_0(\Gamma_s)} + 
\|G\|_{\mathcal{G}_T(\Gamma_s)} +
\|\left[\sigma(w,\pi)\right]n_s\|_{\mathcal{G}_T(\Gamma_s)}
\right).
\end{equation*}
Combined with~\eqref{est-wpi}, this estimate yields
\begin{equation}
\| Z \|_{\mathcal{Z}_T(\Gamma_s)}
 \leq C\left(
\|Z_0\|_{\mathcal{Z}_0(\Gamma_s)} + 
\|G\|_{\mathcal{G}_T(\Gamma_s)} +
\|F^\pm\|_{\mathcal{F}_T(\Omega_s^\pm)} + 
\|H^\pm\|_{\mathcal{U}_T(\Omega_s^\pm)}
\right). \label{super-est-T}
\end{equation}
Wellposedness of the linear system~\eqref{syslin-a} is stated as follows:
\begin{proposition}
For $0<T<\infty$, if $F^\pm \in \mathcal{F}_T(\Omega_s^\pm)$, $H^\pm \in \mathcal{U}_T(\Omega_s^\pm)$, $G \in \mathcal{G}_T(\Gamma_s)$ and $Z_0 \in \mathcal{Z}_0(\Gamma_s)$, then there exists a unique solution $Z\in \mathcal{Z}_T(\Gamma_s)$ to system~\eqref{syslin-a}. It satisfies~\eqref{super-est-T}.
\end{proposition}

\begin{proof}
Existence is provided by the lifting method described above. For proving uniqueness, we use the linearity of the system, and assume $F^\pm = H^\pm = G = Z_0 = 0$. Then from Lemma~\ref{lemma-Energy} we obtain~\eqref{energy-estimate} with $g=0$, namely the following identity
\begin{equation*}
\frac{\mu}{2}\frac{\d}{\d t} \| \nabla_{\Gamma_s} Z \|^2_{\mathds{L}^2(\Gamma_s)}  
+2\nu \left(\|\varepsilon(u^+) \|^2_{\LLL^2(\Omega_s^+)} 
+ \|\varepsilon(u^-) \|_{\LLL^2(\Omega_s^-)} 
\right) = \lambda \| \nabla_{\Gamma_s} Z \|^2_{\mathds{L}^2(\Gamma_s)}.
\end{equation*}
The Gr\"onwall's lemma combined with with $Z_0 =0$ yields $Z\equiv 0$ (up to a constant of~$\R^2$), and concludes the proof.
\end{proof}

\section{Feedback operator for the linear system} \label{sec-feedback}

This section is devoted to the design of a feedback operator that stabilizes system~\eqref{mainsyslin} in infinite time horizon. Let us first study its controllability properties in finite time horizon.

\subsection{Approximate controllability} \label{sec-approx-cont}

Let be $0<T<\infty$. We consider system~\eqref{mainsyslin} with control $G$, with null data and initial condition:
\begin{equation}
\begin{array} {rcl}
-\divg \sigma(u^\pm,p^\pm) = 0 \quad \text{ and } \quad
\divg u^\pm = 0 & & \text{in } \Omega_s^\pm \times (0,T), \\
u^+ = 0 & & \text{on } \p \Omega \times (0,T), \\
u^\pm = \displaystyle \frac{\p Z}{\p t} 
\quad \text{ and } \quad  
-\left[ \sigma(u,p)\right]n = \mu \Delta_{\Gamma_s} Z + G & & \text{on } \Gamma_s \times (0,T), \\
Z(\cdot,0)  = 0 & & \text{on } \Gamma_s.
\end{array} \label{syslin0}
\end{equation}
Recall the definition of {\it approximate controllability} and {\it exact controllability} for linear evolution equations of type~\eqref{syslin0}, that we state in our context as follows:

\begin{definition}
Define the reachable set as
\begin{equation*}
R(T)  :=  \left\{\text{$Z(\cdot,T)$ such that $Z$ is solution of~\eqref{syslin0}} \mid G \in \mathcal{G}_T(\Gamma_s) \right\}.
\end{equation*}
We say that system~\eqref{syslin0} is approximately controllable if $R(T)$ is dense in $\LL^2(\Gamma_s)/\R^2$, or equivalently if $R(T)^{\perp} = \R^2$. We say that~\eqref{syslin0} is exactly controllable if $R(T)=\LL^2(\Gamma_s)/\R^2$.
\end{definition}
We obtain the following key result:
\begin{proposition} \label{prop-approx-cont}
System~\eqref{syslin0} is approximately controllable.
\end{proposition}

\begin{proof}
Introduce $Z_T \in R(T)^\perp$, and the following adjoint system
\begin{equation}
\begin{array} {rcl}
\displaystyle  - \divg \sigma(\phi^\pm, \psi^\pm) = 0 
\quad \text{ and } \quad
\divg \phi^\pm = 0 & &  \text{in } \Omega_s^\pm \times (0,T),\\
\phi^+ = 0 & & \text{on } \p \Omega \times (0,T), \\
\phi^\pm =  \displaystyle -\frac{\p \zeta}{\p t} 
\quad \text{ and } \quad
-\left[ \sigma (\phi, \psi)\right]n_s =  \mu \Delta_{\Gamma_s} \zeta & & 
\text{on } \Gamma_s \times (0,T), \\
\zeta(T) = Z_T & & \text{on }\Gamma_s,
\end{array} \label{sysadj}
\end{equation}
with $(\phi^+,\psi^+, \phi^-, \psi^-, \zeta)$ as unknowns. Now consider $(u^+, p^+, u^-,p^-, Z)$ the solution of~\eqref{syslin0}. Taking the inner product in $\L^2(0,T;\LL^2(\Omega_s^\pm))$ of the left equation in the first line of~\eqref{sysadj} by $u^\pm$, by integration by parts we obtain for all $G \in \mathcal{G}_T(\Gamma_s)$
\begin{equation*}
\mu\left\langle  \nabla_{\Gamma_s} Z_T, \nabla_{\Gamma_s} Z(T)\right\rangle_{\mathds{L}^2(\Gamma_s)}  = 
\int_0^T \left\langle G, \frac{\p \zeta}{\p t}\right\rangle_{\LL^2(\Gamma_s)} \d t.
\end{equation*}
Since $Z_T \in R(T)^\perp$, this identity implies that $\displaystyle \frac{\p \zeta}{\p t} = 0$ in $\LL^2(\Gamma_s)$. System~\eqref{sysadj} then becomes system~\eqref{sys-2nd} of Lemma~\ref{lemma-Temam2}, which yields that $\zeta$ is a constant of~$\R^2$, and therefore $Z_T$ too, completing the proof.
\end{proof}

\subsection{Feedback operator}

\begin{theorem} \label{th-feedback}
For all $\lambda >0$ and $Z_0 \in \mathcal{Z}_0(\Gamma_s)$, there exists a finite-dimensional subspace $\mathds{H}_u^{(\lambda)}$ of $\nabla_{\Gamma_s}\HH^{5/2}(\Gamma_s)$, with orthogonal projection $P_{\lambda}: \nabla_{\Gamma_s}\HH^{5/2}(\Gamma_s) \rightarrow \mathds{H}_u^{(\lambda)}$, a finite-dimensional space $\Xi \subset \mathbf{H}^{1/2}(\Gamma_s)$ and a linear operator $\Pi_{\lambda} \in \mathcal{L}\left(\mathds{H}_u^{(\lambda)}, (\mathds{H}_u^{(\lambda)})^{\ast}\right)$ defining the feedback operator
\begin{equation*}
\mathcal{K}_{\lambda} :=  -\mathcal{P}_{\Gamma_s}
 (\nabla_{\Gamma_s})^{\ast}\Pi_{\lambda}P_{\lambda}
\in \mathcal{L} \left(\nabla_{\Gamma_s}\mathbf{H}^{5/2}(\Gamma_s) ,  \Xi \right)
\end{equation*}
such that the solution $Z$ of system~\eqref{mainsyslin} with $G = \mathcal{K}_{\lambda}\nabla_{\Gamma_s}Z$ satisfies $\left\|Z \right\|_{\mathcal{Z}_{\infty}(\Gamma_s)} \leq C_0$, where the constant $C_0$ depends only on $Z_0$. Further, the operator $\Pi_{\lambda}$ is the solution of the following finite-dimensional algebraic Riccati equation
\begin{equation}
\Pi_{\lambda} = \Pi_{\lambda}^{\ast} \succeq 0, \quad 
\Pi_{\lambda} \mathcal{A}_{\lambda} + \mathcal{A}_{\lambda}^{\ast} \Pi_{\lambda}
- \Pi_{\lambda} \mathcal{B}_{\lambda} \mathcal{B}_{\lambda}^{\ast} \Pi_{\lambda} + \I = 0, \label{Riccati}
\end{equation}
where we have introduced $
\mathcal{A}_{\lambda} = P_{\lambda}\mathcal{A} P_{\lambda} \in \mathcal{L}(\mathds{H}_u^{(\lambda)}, \mathds{H}_u^{(\lambda)})$ and $
\mathcal{B}_{\lambda} = P_{\lambda}\nabla_{\Gamma_s} \mathcal{P}_{\Gamma_s} \in \mathcal{L}(\mathbf{H}^{1/2}(\Gamma_s),\mathds{H}_u^{(\lambda)})$.
\end{theorem}

\begin{proof}
Let us consider the operator formulation of system~\eqref{mainsyslin}, namely 
\begin{equation} \label{eq-evol-lambda}
\frac{\p \nabla_{\Gamma_s}Z}{\p t} - (\mathcal{A}+\lambda \Id) \nabla_{\Gamma_s}Z = \nabla_{\Gamma_s}\mathcal{P}_{\Gamma_s}G \quad \text{ in } (0,T), 
\qquad 
\nabla_{\Gamma_s}Z(0) = \nabla_{\Gamma_s}Z_0,
\end{equation}
where $\mathcal{A}$ is defined in~\eqref{def-operator}. Recall that from~$\nabla_{\Gamma_s}Z$ we can retrieve~$Z$ (up to a constant) via Proposition~\ref{prop-rigidity}. We can choose $\lambda$ in the resolvent of~$\mathcal{A}$, without loss of generality. From Proposition~\ref{prop-analytic}, the spectrum of $\mathcal{A}$ is a discrete set of complex eigenvalues $(\lambda_i)_{i\mathbb{N}}$, contained in an angular domain $\left\{z\in \mathbb{C}\setminus \{0\} \mid \mathrm{arg}(\theta - z) \in (-\alpha, \alpha) \right\}$ where $\theta \in (0, \pi /2)$ and $\alpha \in \R$. We can order them such that
\begin{equation*}
\dots < \Re(\lambda_{N+1}) < - \lambda <  \Re(\lambda_{N}) <
\cdots < \Re(\lambda_{2}) < \Re(\lambda_{1}) < 0.
\end{equation*}
Furthermore, the generalized eigenspace associated with each eigenvalue is of finite dimension (see for instance~\cite[Chapter~III, Theorem~6.29 page~187]{Kato}). Denoting by $\Lambda(\lambda_i)$ the real generalized eigenspace of $\lambda_i$ or $(\lambda_i, \overline{\lambda_i})$ whether $\Im(\lambda_i) = 0$ or not, respectively, we introduce the Hilbert spaces
\begin{equation*}
\mathds{H}_u^{(\lambda)} = \bigoplus_{i=1}^N \Lambda(\lambda_i), \quad 
\mathds{H}_s^{(\lambda)} = \bigoplus_{i=N+1}^\infty \Lambda(\lambda_i).
\end{equation*}
Let us explain what we mean by {\it real generalized eigenspace}: If $(e_j(\lambda_i))_{1\leq j \leq m(\lambda_i)}$ is a basis of the complex generalized eigenspace of $\lambda_i$, where $m(\lambda_i)$ denotes its multiplicity, then $\Lambda(\lambda_i)$ is generated by the family $\left\{ \Re(e_j(\lambda_i)) , \Im(e_j(\lambda_i)) \mid 1\leq j \leq m(\lambda_i)\right\}$. Note that $\mathds{H}_u^{(\lambda)}$, the space of {\it unstable modes}, is of finite dimension. Both $\mathds{H}_u^{(\lambda)}$ and $\mathds{H}_s^{(\lambda)}$ are invariant under $\mathcal{A}$. Denote by $P_{\lambda}$ the orthogonal projection on~$\mathds{H}_u^{(\lambda)}$, parallel to~$\mathds{H}_s^{(\lambda)}$. Projecting equation~\eqref{eq-evol-lambda}, with $\lambda = 0$, on $\mathds{H}_u^{(\lambda)}$ yields
\begin{equation}
\frac{\p P_{\lambda}\nabla_{\Gamma_s}Z}{\p t} - P_{\lambda}\mathcal{A}P_{\lambda}\nabla_{\Gamma_s}Z =  P_{\lambda} \nabla_{\Gamma_s}\mathcal{P}_{\Gamma_s}G \quad \text{ in } (0,T),
\qquad 
P_{\lambda}\nabla_{\Gamma_s}Z(0) = P_{\lambda} \nabla_{\Gamma_s}Z_0.
\label{eq-evol-0}
\end{equation}
The approximate controllability of system~\eqref{syslin0} obtained in Proposition~\ref{prop-approx-cont} implies that~\eqref{eq-evol-0} too is approximately controllable. Its reachable set is dense in $\mathds{H}_u^{(\lambda)}$, and since this space is of finite dimension, it is actually equal to $\mathds{H}_u^{(\lambda)}$. This means that equation~\eqref{eq-evol-0} is exactly controllable. From there, we use for instance the result of~\cite[Chapter~I, Theorem~2.9, page~35]{Zabczyk} stating that there exists a linear operator $K_{\lambda}$ defined on $\mathds{H}_u^{(\lambda)}$ such that $P_{\lambda}\mathcal{A}P_{\lambda} + P_{\lambda} \nabla_{\Gamma_s}\mathcal{P}_{\Gamma_s} K_{\lambda}$ is exponentially stable with $\lambda$ as a decay rate. Since the same property holds for $\mathcal{A}(\Id - P_{\lambda})$, we merely set $
\mathcal{K}_{\lambda} = \Re K_{\lambda}P_{\lambda}$ and $\Xi = \Re K_{\lambda}(\mathds{H}_u^{(\lambda)})$. Further, following~\cite{Sontag} (more specifically Lemma~8.4.1 page~381 and Theorem~41 page~384), we consider the following infinite time horizon optimal control problem: 
\begin{equation}
\displaystyle \inf_{G\in \Xi} \left\{\mathcal{J}(Z,G)\mid \text{ $Z$ satisfies~\eqref{syslin0}} \right\},
\label{pbopt}
\end{equation}
with $
\mathcal{J}(Z,G) = \displaystyle \frac{1}{2}\int_0^\infty \|P_{\lambda}\nabla_{\Gamma_s}Z\|_{\mathds{H}_u^{(\lambda)}}^2 \d t + \frac{1}{2} \int_0^{\infty} \|G\|_{\Xi}^2 \d t$. The first-order optimality conditions for Problem~\eqref{pbopt} lead to $G = - \mathcal{P}_{\Gamma_s}^{\ast}(\nabla_{\Gamma_s})^{\ast} \Pi_{\lambda} P_{\lambda} \nabla_{\Gamma_s}Z =\mathcal{P}_{\Gamma_s} \divg_{\Gamma_s} \Pi_{\lambda} P_{\lambda} \nabla_{\Gamma_s}Z$, where $\Pi_{\lambda} = \Pi_{\lambda}^{\ast} \succeq 0$ satisfies the Riccati equation~\eqref{Riccati}, finishing the proof.
\end{proof}

We deduce an estimate for the stabilized linear system with non-homogeneous right-hand-sides.

\begin{corollary} \label{coro-final}
Assume $Z_0 \in \mathcal{Z}_0(\Gamma_s)$, $F^\pm \in \mathcal{F}_{\infty}(\Omega_s^\pm)$, $
H^\pm \in \mathcal{U}_{\infty}(\Omega_s^\pm)$, and $G \in \mathcal{G}_{\infty}(\Gamma_s)$. Using the feedback operator $\mathcal{K}_{\lambda}$ obtained in Theorem~\ref{th-feedback}, there exists a unique solution $ Z \in  \mathcal{Z}_{\infty}(\Gamma_s)$ to the following system
\begin{equation}
\begin{array} {rcl}
-\divg \sigma(u^\pm, p^\pm) = F^\pm \quad \text{ and } \quad 
\divg u^\pm = \divg H^\pm & & \text{in } \Omega_s^\pm \times (0,\infty), \\
u^+ = 0 & &\text{on } \p \Omega \times (0,\infty), \\
u^+ = u^- = \displaystyle \frac{\p Z}{\p t} - \lambda Z 
\quad \text{ and } \quad
-\left[\sigma(u,p) \right]n = 
\mu \divg_{\Gamma_s} \nabla_{\Gamma_s} Z + \mathcal{K}_{\lambda}\nabla_{\Gamma_s}Z + G & &
\text{on } \Gamma_s \times (0,\infty),  \\
Z(\cdot,0 ) = Z_0 & & \text{on } \Gamma_s,
\end{array}
\label{mainsyslin-NH}
\end{equation}
and it satisfies
\begin{equation}
\begin{array} {l}
\|u^+\|_{\mathcal{U}_{\infty}(\Omega_s^+)} +
\|p^+\|_{\mathcal{Q}_{\infty}(\Omega_s^+)} +
\|u^-\|_{\mathcal{U}_{\infty}(\Omega_s^-)} +
\|p^-\|_{\mathcal{Q}_{\infty}(\Omega_s^-)} +
\|Z\|_{\mathcal{Z}_{\infty}(\Gamma_s)} \\
\leq
C_s(1+\lambda) \left( 
\|Z_0\|_{\mathcal{Z}_0(\Gamma_s)} +
\|F^+\|_{\mathcal{F}_{\infty}(\Omega_s^+)} +
\|F^-\|_{\mathcal{F}_{\infty}(\Omega_s^-)} +
\|H^+\|_{\mathcal{U}_{\infty}(\Omega_s^+)} +
\|H^-\|_{\mathcal{U}_{\infty}(\Omega_s^-)} +
\|G\|_{\mathcal{G}_{\infty}(\Gamma_s)}
\right),
\end{array}
\label{estimate-fixedpoint}
\end{equation}
where the constant $C_s >0$ depends only on $\Gamma_s$.
\end{corollary}

\begin{proof}
The lifting method of section~\ref{sec-lifting} can be used here: Introduce $u^\pm = v^\pm + w^\pm$ and $p^\pm = q^\pm + \pi^\pm$, where $(w^\pm,\pi^\pm)$ satisfy the Stokes problems~\eqref{syslin-b} with $(0,T)$ replaced by $(0,\infty)$, and where $(v^\pm, q^\pm, Z)$ satisfies
\begin{equation}
\begin{array} {rcl}
-\divg \sigma(v^\pm, q^\pm) = 0 \quad \text{ and } \quad 
\divg v^\pm = 0 & & \text{in } \Omega_s^\pm \times (0,\infty), \\
v^+ = 0 & &\text{on } \p \Omega \times (0,\infty), \\
v^\pm = \displaystyle \frac{\p Z}{\p t} - \lambda Z
\ \text{ and } \ 
-\left[\sigma(v,q) \right]n_s = \mu \divg_{\Gamma_s} \nabla_{\Gamma_s} Z 
+ \mathcal{K}_{\lambda}\nabla_{\Gamma_s}Z + G 
 + \left[\sigma(w,\pi) \right]n_s & &
\text{on } \Gamma_s \times (0,\infty), \\
Z(\cdot,0 ) = Z_0 & & \text{on } \Gamma_s.
\end{array}\label{syslin-d}
\end{equation}
We formulate system~\eqref{syslin-d} as
\begin{equation*}
\begin{array} {rcl}
\displaystyle
\frac{\p \nabla_{\Gamma_s}Z}{\p t} -\lambda \nabla_{\Gamma_s}Z 
- \nabla_{\Gamma_s}\mathcal{P}_{\Gamma_s}(\mu \divg_{\Gamma_s} \nabla_{\Gamma_s} Z) - \mathcal{K}_{\lambda} \nabla_{\Gamma_s}Z 
& = &
\nabla_{\Gamma_s}\mathcal{P}_{\Gamma_s}\left(G+ \left[ \sigma(w,\pi)\right]n_s\right),\\[5pt]
\displaystyle
\frac{\p \nabla_{\Gamma_s}Z}{\p t} -(\lambda\Id + \mathcal{A} + \mathcal{K}_{\lambda}) \nabla_{\Gamma_s}Z 
& = &
\nabla_{\Gamma_s}\mathcal{P}_{\Gamma_s}\left(G+ \left[ \sigma(w,\pi)\right]n_s\right),
\end{array}
\end{equation*}
and since the operator $\lambda \Id + \mathcal{A} + \mathcal{K}_{\lambda}$ is the infinitesimal generator of an analytic semigroup of negative type, a consequence of~\cite[Theorem~3.1 page~143, Part~II]{Bensoussan} and Proposition~\ref{prop-rigidity} is the existence of $Z\in \mathcal{Z}_{\infty}(\Gamma_s)$, satisfying
\begin{equation*}
\| Z \|_{\mathcal{Z}_{\infty}(\Gamma_s)}
\leq C\left(
\|Z_0\|_{\mathcal{Z}_0(\Gamma_s)} + 
\|G\|_{\mathcal{G}_{\infty}(\Gamma_s)} +
\|\left[\sigma(w,\pi)\right]n_s\|_{\mathcal{G}_{\infty}(\Gamma_s)}
\right).
\end{equation*}
Next, the steps of section~\ref{sec-lifting} can be repeated to obtain the existence and uniqueness of $Z$, which satisfies
\begin{equation}
\|Z\|_{\mathcal{Z}_{\infty}(\Gamma_s)} 
\leq
C \left( 
\|Z_0\|_{\mathcal{Z}_0(\Gamma_s)} +
\|F^+\|_{\mathcal{F}_{\infty}(\Omega_s^+)} +
\|F^-\|_{\mathcal{F}_{\infty}(\Omega_s^-)} +
\|H^+\|_{\mathcal{U}_{\infty}(\Omega_s^+)} +
\|H^-\|_{\mathcal{U}_{\infty}(\Omega_s^-)} +
\|G\|_{\mathcal{G}_{\infty}(\Gamma_s)}
\right).
\label{houpala}
\end{equation}
Further, $(u^\pm,p^\pm)$ are also obtained uniquely as the solutions of the classical Stokes problems with Dirichlet boundary conditions and non-homogeneous divergence condition, namely
\begin{equation*}
\begin{array} {rcl}
-\divg \sigma(u^\pm, p^\pm) = F^\pm \quad \text{ and } \quad 
\divg u^\pm = \divg H^\pm & & \text{in } \Omega_s^\pm \times (0,\infty), \\
u^+ = 0 & &\text{on } \p \Omega \times (0,\infty), \\
u^\pm = \displaystyle \frac{\p Z}{\p t} - \lambda Z & &
\text{on } \Gamma_s \times (0,\infty).
\end{array}
\end{equation*}
Up to considering $u^\pm - H^\pm$, from~\cite[Lemma~6.1, Chapter~IV]{Galdi} they satisfy the estimate
\begin{equation*}
\begin{array}{rcl}
\|u^\pm\|_{\mathcal{U}_{\infty}(\Omega_s^\pm)} +
\|p^\pm\|_{\mathcal{Q}_{\infty}(\Omega_s^\pm)} & \leq & C \left(
\left\|\displaystyle \frac{\p Z}{\p t} - \lambda Z\right\|_{\L^2(0,\infty;\HH^{3/2}(\Gamma_s))}
 + \| H^\pm \|_{\mathcal{U}_{\infty}(\Omega_s^\pm)}
\right)\\
& \leq & \left( (1+\lambda)\|Z\|_{\mathcal{Z}_{\infty}(\Gamma_s)} +  \| H^\pm \|_{\mathcal{U}_{\infty}(\Omega_s^\pm)}  \right)
\end{array}
\end{equation*}
which, combined with~\eqref{houpala}, leads to~\eqref{estimate-fixedpoint}, and thus the announced result.
\end{proof}

\section{Feedback stabilization of the nonlinear system} \label{sec-nonlinear}
In this section we prove Theorem~\ref{th-main}. We first prove wellposedness of system~\eqref{sys-cyl2}-\eqref{rhs-nonlinear} when $\hat{g}$ is replaced by $\mu\divg_{\Gamma_s}\big( (\tau_s \otimes \tau_s) \nabla_{\Gamma_s}\hat{Z}\big) + \mathcal{K}_{\lambda} \nabla_{\Gamma_s}\hat{Z}$ in~\eqref{sys-cyl2}:
\begin{equation}
\begin{array} {rl}
 - \divg (\sigma(\hat{u}^\pm, \hat{p}^\pm))
 =  \hat{f}^\pm + F(\hat{u}^\pm,\hat{p}^\pm, \hat{Z})
\quad \text{ and } \quad
\divg \hat{u}  =  \divg H(\hat{u}^\pm, \hat{Z})  & 
\text{in } \Omega_s^\pm \times (0,\infty), \\
\hat{u}^+ = 0 & \text{on } \p \Omega \times (0,\infty), \\
\hat{u}^\pm = \displaystyle \frac{\p \hat{Z}}{\p t} - \lambda \hat{Z}, 
\ \text{ and } \ 
-\left[\sigma(\hat{u},\hat{p})\right] n_s = \mu\divg_{\Gamma_s} \nabla_{\Gamma_s}\hat{Z} + 
\mathcal{K}_{\lambda}\nabla_{\Gamma_s} \hat{Z} +
G(\hat{u}^+, \hat{p}^+, \hat{u}^-, \hat{p}^-, \hat{Z})
& \text{on } \Gamma_s \times (0,\infty), \\
\hat{Z}(\cdot,0) = X_0-\Id & \text{on } \Gamma_s.
\end{array} \label{sys-cyl3}
\end{equation}
Denote
\begin{equation*}
\mathcal{H}(\Omega_s^+,\Omega_s^-,\Gamma_s) :=
\mathcal{U}_{\infty}(\Omega_s^+) \times \mathcal{Q}_{\infty}(\Omega_s^+) \times \mathcal{U}_{\infty}(\Omega_s^-) \times \mathcal{Q}_{\infty}(\Omega_s^-) \times \mathcal{Z}_{\infty}(\Gamma_s),
\end{equation*}
that we equip with the norm that goes without saying. A solution for system~\eqref{sys-cyl3} is obtained as a fixed point of the mapping
\begin{equation*}
\begin{array} {rccc}
\mathcal{N}: & \mathcal{H}(\Omega_s^+,\Omega_s^-,\Gamma_s) & \rightarrow &
\mathcal{H}(\Omega_s^+,\Omega_s^-,\Gamma_s) \\
 & (\hat{u}_1^+,\hat{p}_1^+,\hat{u}_1^-,\hat{p}_1^-,\hat{Z}_1) & \mapsto & (\hat{u}_2^+,\hat{p}_2^+,\hat{u}_2^-,\hat{p}_2^-,\hat{Z}_2), 
\end{array}
\end{equation*}
where $(\hat{u}_2^+,\hat{p}_2^+,\hat{u}_2^-,\hat{p}_2^-,\hat{Z}_2)$ is the solution of~\eqref{mainsyslin-NH} with $F^\pm$,  $H^\pm$ and $G$ replaced by $F(\hat{u}_1^\pm,\hat{p}_1^\pm, \hat{Z}_1)$, $ H(\hat{u}_1^\pm, \hat{Z}_1)$ and $G(\hat{u}_1^+, \hat{p}_1^+, \hat{u}_1^-, \hat{p}_1^-, \hat{Z}_1)$, respectively: 
\begin{equation*}
\begin{array} {rl}
 - \divg (\sigma(\hat{u}_2^\pm, \hat{p}_2^\pm))
 =  \hat{f}^\pm + F(\hat{u}_1^\pm,\hat{p}_1^\pm, \hat{Z}_1)
\quad \text{ and } \quad
\divg \hat{u}_2  =  H(\hat{u}_1^\pm, \hat{Z}_1) &  
\text{in } \Omega_s^\pm \times (0,\infty), \\
\hat{u}_2^+ = 0 & \text{on } \p \Omega \times (0,\infty), \\
\hat{u}_2^\pm = \displaystyle \frac{\p \hat{Z}_2}{\p t} - \lambda \hat{Z}_2, 
\ \text{ and } \ 
-\left[\sigma(\hat{u}_2,\hat{p}_2)\right] n_s = \mu\divg_{\Gamma_s} \nabla_{\Gamma_s}\hat{Z}_2 + 
\mathcal{K}_{\lambda}\nabla_{\Gamma_s} \hat{Z}_2 +
G(\hat{u}_1^+, \hat{p}_1^+, \hat{u}_1^-, \hat{p}_1^-, \hat{Z}_1)
& \text{on } \Gamma_s \times (0,\infty), \\
\hat{Z}_2(\cdot,0) = X_0-\Id & \text{on } \Gamma_s.
\end{array} \label{sys-cyl4}
\end{equation*}
Estimate~\eqref{estimate-fixedpoint} of Corollary~\ref{coro-final} yields
\begin{equation}
\begin{array} {rcl}
\| (\hat{u}_2^+,\hat{p}_2^+,\hat{u}_2^-,\hat{p}_2^-,\hat{Z}_2)\|_{\mathcal{H}(\Omega_s^+,\Omega_s^-,\Gamma_s)} 
& \leq & C_s(1+\lambda)
 \Big( 
\|X_0-\Id\|_{\mathcal{Z}_0(\Gamma_s)} + \|\hat{f}^+\|_{\mathcal{F}_{\infty}(\Omega_s^+)} +
\|\hat{f}^-\|_{\mathcal{F}_{\infty}(\Omega_s^-)}\\
& &  +\|F(\hat{u}_1^+,\hat{p}_1^+, \hat{Z}_1)\|_{\mathcal{F}_{\infty}(\Omega_s^+)} +
\|F(\hat{u}_1^-,\hat{p}_1^-, \hat{Z}_1)\|_{\mathcal{F}_{\infty}(\Omega_s^-)} \\
& & 
+\|H(\hat{u}_1^+, \hat{Z}_1)\|_{\mathcal{U}_{\infty}(\Omega_s^+)} +
\|H(\hat{u}_1^-, \hat{Z}_1)\|_{\mathcal{U}_{\infty}(\Omega_s^-)} \\
& & +\|G(\hat{u}_1^+, \hat{p}_1^+, \hat{u}_1^-, \hat{p}_1^-, \hat{Z}_1)\|_{\mathcal{G}_{\infty}(\Gamma_s)}
\Big).
\end{array}
\label{estimate-fixedpoint12}
\end{equation}
Consider the following closed subset of~$\mathcal{Z}_{\infty}(\Gamma_s)$
\begin{equation*}
\begin{array} {rcl}
\mathcal{B}_\rho & := & \left\{
(\hat{u}^+,\hat{p}^+,\hat{u}^-,\hat{p}^-,\hat{Z}) \in \mathcal{H}_{\infty}(\Omega_s^+,\Omega_s^-,\Gamma_s) \mid \|(\hat{u}^+,\hat{p}^+,\hat{u}^-,\hat{p}^-,\hat{Z}\|_{\mathcal{H}(\Omega_s^+,\Omega_s^-,\Gamma_s)} \leq 
2C_s(1+\lambda)\rho
\right\},
\end{array}
\end{equation*}
where
\begin{equation*}
\rho := 
\|X_0 - \Id\|_{\mathcal{Z}_0(\Gamma_s)}
+ \| \hat{f}^+ \|_{\mathcal{F}_{\infty}(\Omega_s^+)}
+ \| \hat{f}^- \|_{\mathcal{F}_{\infty}(\Omega_s^-)},
\end{equation*}
and $C_s$ is the constant of estimate~\eqref{estimate-fixedpoint}. Let us prove that $\mathcal{N}$ is a contraction in $\mathcal{B}$, provided that $\|X_0-\Id\|_{\mathcal{Z}_0(\Gamma_s)}$ and $\|\hat{f}^\pm\|_{\mathcal{F}_{\infty}(\Omega_s^\pm)}$ are small enough. Since the different nonlinearities in the right-hand-side of~\eqref{estimate-fixedpoint12} are polynomial, from~\cite[Proposition~B.1, page~283]{Grubb} we can address them with estimates of type
\begin{equation*}
\| \nabla \tilde{Y}(\tilde{X}) \nabla \hat{u} \|_{\mathbb{H}^1(\Omega_s^\pm)} 
\leq C
\| \nabla \tilde{Y}(\tilde{X}) \|_{\mathbb{H}^{3/2}(\Omega_s^\pm)}
\|  \nabla \hat{u} \|_{\mathbb{H}^1(\Omega_s^\pm)}.
\end{equation*}
Combined with the Lipschitz estimates of Proposition~\ref{prop-extension} and Corollary~\ref{coro-extension}, we deduce
\begin{equation*}
\begin{array} {rcl}
\|F(\hat{u}^\pm,\hat{p}^\pm, \hat{Z})\|_{\mathbf{L}^2(\Omega_s^\pm)} 
& \leq &  
C\| \nabla \tilde{Y}(\tilde{X}) \|_{\mathbb{H}^{3/2}(\Omega_s^\pm)}
\left(\|\nabla  \hat{u} \|_{\mathbb{H}^1(\Omega_s^\pm)} +
\|  \hat{p} \|_{\H^1(\Omega_s^\pm)}\right)
\| \nabla \tilde{X}-\I \|_{\mathbb{H}^{3/2}(\Omega_s^\pm)}, \\
\|F(\hat{u}^\pm,\hat{p}^\pm, \hat{Z})\|_{\mathcal{F}_{\infty}(\Omega_s^\pm)} 
& \leq &  
C\left(1+\| X-\Id \|_{\mathcal{Z}_{\infty}(\Gamma_s)}\right)
\left(\|  \hat{u} \|_{\mathcal{U}_{\infty}(\Omega_s^\pm)}+
\|  \hat{p} \|_{\mathcal{Q}_{\infty}(\Omega_s^\pm)}
\right)
\| X-\Id \|_{\mathcal{Z}_{\infty}(\Gamma_s)} \\
& \leq & 
C\left(1+\| e^{-\lambda t} \hat{Z} \|_{\mathcal{Z}_{\infty}(\Gamma_s)}\right)
\left(\|  \hat{u} \|_{\mathcal{U}_{\infty}(\Omega_s^\pm)}+
\|  \hat{p} \|_{\mathcal{Q}_{\infty}(\Omega_s^\pm)}
\right)
\| e^{-\lambda t}\hat{Z} \|_{\mathcal{Z}_{\infty}(\Gamma_s)}
\\
\|H(\hat{u}^\pm, \hat{Z})\|_{\mathcal{U}_{\infty}(\Omega_s^\pm)} & \leq & 
C\|  \hat{u} \|_{\mathcal{U}_{\infty}(\Omega_s^\pm)}
\| X-\Id \|_{\mathcal{Z}_{\infty}(\Gamma_s)}
\leq C\|  \hat{u} \|_{\mathcal{U}_{\infty}(\Omega_s^\pm)}
\| e^{-\lambda t}\hat{Z} \|_{\mathcal{Z}_{\infty}(\Gamma_s)}
,
\\
\|G(\hat{u}^+, \hat{p}^+, \hat{u}^-, \hat{p}^-, \hat{Z})\|_{\mathcal{G}_{\infty}(\Gamma_s)} & \leq & 
C\left(1+\| e^{-\lambda t}\hat{Z} \|_{\mathcal{Z}_{\infty}(\Gamma_s)}\right)
\left(\|  \hat{u} \|_{\mathcal{U}_{\infty}(\Omega_s^\pm)}+
\|  \hat{p} \|_{\mathcal{Q}_{\infty}(\Omega_s^\pm)}
\right)
\| e^{-\lambda t}\hat{Z} \|_{\mathcal{Z}_{\infty}(\Gamma_s)}\\ 
& & 
+ \mathcal{O}(\|\nabla \tilde{X}-I\|^2_{\mathbb{H}^1(\Gamma_s)}) \\
& \leq & C\left(1+\| e^{-\lambda t}\hat{Z} \|_{\mathcal{Z}_{\infty}(\Gamma_s)}\right)
\left(\|  \hat{u} \|_{\mathcal{U}_{\infty}(\Omega_s^\pm)}+
\|  \hat{p} \|_{\mathcal{Q}_{\infty}(\Omega_s^\pm)}
\right)
\| e^{-\lambda t}\hat{Z} \|_{\mathcal{Z}_{\infty}(\Gamma_s)}\\ 
&  & 
+C\| e^{-\lambda t}\hat{Z} \|^2_{\mathcal{Z}_{\infty}(\Gamma_s)} .
\end{array}
\end{equation*}
We see easily that $\| e^{-\lambda t}\hat{Z} \|_{\mathcal{Z}_{\infty}(\Gamma_s)} \leq C(1+\lambda)\| \hat{Z} \|_{\mathcal{Z}_{\infty}(\Gamma_s)}$. Consequently, if $(\hat{u}_1^+,\hat{p}_1^+,\hat{u}_1^-,\hat{p}_1^-, \hat{Z}_1) \in \mathcal{B}_\rho$, from~\eqref{estimate-fixedpoint12} we obtain 
\begin{equation*}
\begin{array} {rcl}
\| (\hat{u}_2^+,\hat{p}_2^+,\hat{u}_2^-,\hat{p}_2^-,\hat{Z}_2)\|_{\mathcal{H}(\Omega_s^+,\Omega_s^-,\Gamma_s)} 
& \leq & 
 C_s(1+\lambda)\left(\rho + C\rho^2(1+\lambda)^3(1+(1+\lambda) + \rho(1+\lambda)^2)
\right).
\end{array}
\label{estimate-fixedpoint13}
\end{equation*}
Therefore $\mathcal{N}$ is well-defined, and if $\rho$ is small enough, that is 
\begin{equation*}
C\rho^2(1+\lambda)^3(1+(1+\lambda) + \rho(1+\lambda)^3) \leq \rho,
\end{equation*}
the ball $\mathcal{B}$ is left invariant under $\mathcal{N}$. Next, let be
\begin{equation*}
(\hat{u}_{i}^+,\hat{p}_{i}^+,\hat{u}_{i}^-,\hat{p}_{i}^-,\hat{Z}_{i}) \in \mathcal{B}_\rho 
\end{equation*}
for $i\in \{1,2\}$. The difference
\begin{equation*}
(\overline{u}^+,\overline{p}^+,\overline{u}^-,\overline{p}^-,\overline{Z}) :=
\mathcal{N}(\hat{u}_{1}^+,\hat{p}_{1}^+,\hat{u}_{1}^-,\hat{p}_{1}^-,\hat{Z}_{1})-
\mathcal{N}(\hat{u}_{2}^+,\hat{p}_{2}^+,\hat{u}_{2}^-,\hat{p}_{2}^-,\hat{Z}_{2})
\end{equation*}
satisfies
\begin{equation}
\begin{array} {rcl}
 - \divg (\sigma(\overline{u}^\pm, \overline{p}^\pm))
 =   \overline{F}^\pm
\quad \text{ and } \quad
\divg \overline{u}  =  \divg \overline{H}^\pm &  & 
\text{in } \Omega_s^\pm \times (0,\infty), \\
\overline{u}^+ = 0 & & \text{on } \p \Omega \times (0,\infty), \\
\overline{u} = \displaystyle \frac{\p \overline{Z}}{\p t} - \lambda \overline{Z}, 
\quad \text{ and } \quad
-\left[\sigma(\overline{u},\overline{p})\right] n_s = \mu\divg_{\Gamma_s} \nabla_{\Gamma_s}\overline{Z} + 
\mathcal{K}_{\lambda}\nabla_{\Gamma_s} \overline{Z} +
\overline{G}
& & \text{on } \Gamma_s \times (0,\infty), \\
\overline{Z}(\cdot,0) = 0 & & \text{on } \Gamma_s.
\end{array} \label{sys-cyl7}
\end{equation}
where we have introduced
\begin{equation*}
\begin{array}{rcl}
\overline{F}^\pm & := & F(\hat{u}_{1}^\pm,\hat{p}_{1}^\pm, \hat{Z}_1)
- F(\hat{u}_{2}^\pm,\hat{p}_{2}^\pm, \hat{Z}_2), \\
\overline{H}^\pm & := & H(\hat{u}^\pm_{1}, \hat{Z}_1)-H(\hat{u}^\pm_{2}, \hat{Z}_2), \\
\overline{G} & := & 
G(\hat{u}^+_1, \hat{p}^+_1, \hat{u}^-_1, \hat{p}^-_1, \hat{Z}_1)
- G(\hat{u}^+_2, \hat{p}^+_2, \hat{u}^-_2, \hat{p}^-_2, \hat{Z}_2).
\end{array}
\end{equation*}
Using the Lipschitz estimates of Proposition~\ref{prop-extension} and Corollary~\ref{coro-extension}, they satisfy
\begin{equation*}
\begin{array} {rcl}
\|\overline{F}^\pm\|_{\mathcal{F}_{\infty}(\Omega_s^\pm)} 
+ \|\overline{G}\|_{\mathcal{G}_{\infty}(\Gamma_s^\pm)}
& \leq & 
C\left( \|e^{-\lambda t}( \hat{Z}_1-\hat{Z}_2)\|_{\mathcal{Z}_{\infty}(\Gamma_s)}
+ \|\hat{u}_1 - \hat{u}_2\|_{\mathcal{U}_{\infty}(\Omega_s^\pm)}
+ \|\hat{p}_1 - \hat{p}_2\|_{\mathcal{Q}_{\infty}(\Omega_s^\pm)}
\right) \\
& & \times \left( \displaystyle\sum_{i=1}^2
\|\hat{u}_i \|_{\mathcal{U}_{\infty}(\Omega_s^\pm)} 
+ \|\hat{p}_i \|_{\mathcal{Q}_{\infty}(\Omega_s^\pm)}
+\|e^{-\lambda t}\hat{Z}_i\|_{\mathcal{Z}_{\infty}(\Gamma_s)}
\right) \\
& & \times\left(1+ \|e^{-\lambda t} \hat{Z}_1\|_{\mathcal{Z}_{\infty}(\Gamma_s)} 
+ \|e^{-\lambda t} \hat{Z}_2\|_{\mathcal{Z}_{\infty}(\Gamma_s)}\right),\\
\|\overline{H}^\pm\|_{\mathcal{U}_{\infty}(\Omega_s^\pm)} & \leq & 
C\left( \|e^{-\lambda t}( \hat{Z}_1-\hat{Z}_2)\|_{\mathcal{Z}_{\infty}(\Gamma_s)}
+ \|\hat{u}_1 - \hat{u}_2\|_{\mathcal{U}_{\infty}(\Omega_s^\pm)}
\right)\\
& & \times \left( 
 \|\hat{u}_2\|_{\mathcal{U}_{\infty}(\Omega^\pm_s)}
+\|e^{-\lambda t} \hat{Z}_1\|_{\mathcal{Z}_{\infty}(\Gamma_s)}
\right).
\end{array}
\end{equation*}
Combined with estimate~\eqref{estimate-fixedpoint} of Corollary~\ref{coro-final}, we then obtain
\begin{equation*}
\begin{array} {rcl}
\| (\overline{u}^+,\overline{u}^+,\overline{u}^-,\overline{p}^-,\overline{Z})\|_{\mathcal{H}(\Omega_s^+,\Omega_s^-,\Gamma_s)} 
& \leq & C_s(1+\lambda)
 \Big( 
\|\overline{F}^+\|_{\mathcal{F}_{\infty}(\Omega_s^+)} +
\|\overline{F}^-\|_{\mathcal{F}_{\infty}(\Omega_s^-)} \\
& & 
+\|\overline{H}^+\|_{\mathcal{U}_{\infty}(\Omega_s^+)} +
\|\overline{H}^-\|_{\mathcal{U}_{\infty}(\Omega_s^-)}  +\|\overline{G}\|_{\mathcal{G}_{\infty}(\Gamma_s)}
\Big) \\
& \leq & 
C(1+\lambda)^2\rho(1+(1+\lambda)\rho) \\
& & \times
\|(\hat{u}_{1}^+ -\hat{u}_2^+,\hat{p}_{1}^+-\hat{p}_2^+,\hat{u}_{1}^- -\hat{u}_2^-,\hat{p}_{1}^--\hat{p}_2^-,\hat{Z}_{1}-\hat{Z}_2)\|_{\mathcal{H}(\Omega_s^+,\Omega_s^-,\Gamma_s)}.
\end{array}
\label{estimate-fixedpoint19}
\end{equation*}
Choosing once again $\rho$ small enough, that is $C(1+\lambda)^2\rho(1+(1+\lambda)\rho) < 1$, we obtain that $\mathcal{N}$ is a contraction in $\mathcal{B}_\rho$. Therefore wellposedness for~\eqref{sys-cyl3} is a consequence of the Banach fixed-point theorem. Furthermore, $\|\hat{Z}\|_{\mathcal{Z}_{\infty}(\Gamma_s)}$ is bounded. Recall that in~\eqref{change-unknowns-bis} we introduced $\hat{Z} = e^{\lambda t} (X-\Id)$, where $\Id$ can be replaced by any deformation $X_c \in \Diff$. In~\eqref{sys-cyl3} we have chosen
\begin{equation*}
\begin{array} {rcl}
\hat{g} & = & \mu\divg_{\Gamma_s}\big( (\tau_s \otimes \tau_s) \nabla_{\Gamma_s}\hat{Z}\big) +\mathcal{K}_{\lambda} \nabla_{\Gamma_s}\hat{Z}\\ 
& = & 
\mu e^{\lambda t} \divg_{\Gamma_s}\big( (\tau_s \otimes \tau_s)\nabla_{\Gamma_s}(X-X_c)\big) +
e^{\lambda t} \mathcal{K}_{\lambda}\nabla_{\Gamma_s}(X-X_c).
\end{array}
\end{equation*}
Still following section~\ref{sec-extension2}, we note that system~\eqref{sys-cyl3} is equivalent to~\eqref{mainsys} by choosing
\begin{equation*}
\begin{array} {rcl}
g & = & \big(|\cof \nabla \tilde{X} n_s |^{-1} \tilde{g} \big) \circ X^{-1}
= \big(|\cof \nabla \tilde{X} n_s |^{-1} e^{-\lambda t}\hat{g} \big) \circ X^{-1}
\\[5pt]
& = & \left(|\cof \nabla \tilde{X} n_s |^{-1}\left(
\divg_{\Gamma_s} \big( (\tau_s \otimes \tau_s) \nabla_{\Gamma_s} (X-X_c) \big)
+ \mathcal{K}_{\lambda}\nabla_{\Gamma_s}(X-X_c)
\right)\right) \circ X^{-1}.
\end{array} 
\end{equation*}
Since $(\det \mathfrak{g}(t)^{1/2} = |\cof \nabla \tilde{X} n_s | (\det \mathfrak{g}_s)^{1/2}$, and $(\det \mathfrak{g}_s)^{1/2} = r$ is constant, we consider
\begin{equation*}
g = \left(r(\det \mathfrak{g})^{-1/2}\left(
\divg_{\Gamma_s} \big( (\tau_s \otimes \tau_s) \nabla_{\Gamma_s} (X-X_c) \big)
+\mathcal{K}_{\lambda}\nabla_{\Gamma_s}(X-X_c)\right)\right) 
\circ X^{-1}.
\end{equation*}
Thus the result announced in Theorem~\ref{th-main} follows.




\section{Comments on a possible extension to dimension 3} \label{sec-dim3}

Some results obtained in the present paper could certainly and straightforwardly be extended to the three-dimensional case, like the study of the Poincar\'e-Steklov operator for example, or the design of the feedback operator. Higher-order Sobolev spaces may be considered for guaranteeing the $C^1$~regularity and stability of Sobolev spaces by product. However, some geometric aspects would deserve a careful investigation. Let us make comments on the difficulties that appear in dimension~3:
\begin{itemize}
\item About the stationary state obtained in Lemma~\ref{lemma-Temam}: In dimension~2, the interface $\Gamma(t)$ is a curve, and its mean curvature is simply called the {\it curvature}. From the fundamental theorem of curves, this curvature determines entirely $\Gamma(t)$, up to proper rigid deformations. In the case of dimension~3, the interface $\Gamma(t)$ is then a surface, and this is the Gaussian curvature which characterizes the metric of the surface. We say that this is an {\it intrinsic} property of the surface $\Gamma(t)$ (cf. the Gauss's {\it Theorem Egregium}). More precisely, two surfaces with the same Gauss curvature differ only up to proper rigid deformations, we say that they are {\it congruent}. The mean curvature which appears in the surface-tension model is only extrinsic in dimension~3, which means that two surfaces with the same mean curvature could not be congruent. However, when restricting the framework to closed surfaces, the Alexandrov's theorem~\cite{Aleksandrov} (see~\cite{Aleksandrov-en} for an English translation) provides a positive result: Two closed surfaces with the same mean curvature are identical, {\it up to essential transformations}. Essential transformations refer to proper rigid deformations and dilatation. In the incompressible case, the volume contained inside the surface is constant, and thus this notion reduces to proper rigid deformations, like in dimension~2.

\item About the linearized system in dimension~3: Simplifications specific to dimension~2 have been made in section~\ref{sec-extension2} when linearizing the mean curvature of $\Gamma(t)$ for small displacements. The expression so obtained involves the operator~$\nabla^{n_s}_{\Gamma_s}$. A priori the linear operator which appears in dimension~3 is more complex, and discussions of section~\ref{sec-kernel} about the kernel of~$\nabla^{n_s}_{\Gamma_s}$ would no longer be relevant.

\item About the extension of diffeomorphisms on the sphere into the ball: This question is less simple in the case of a 2-sphere. In~\cite[system~(8.3), section~8]{Ye1994}, the author gave comments on conditions under which we could extend a diffeomorphism defined on a boundary of a given domain. A sufficient condition is that the set of diffeomorphisms of this boundary preserving the orientation is connected. In $\R^2$, this sufficient condition is always fulfilled and thus the answer is positive. In $\R^3$, things are more delicate, and counter-examples to this sufficient condition exist. However, in the case of the sphere, Smale provided a positive answer in~\cite{Smale}. The result requires a $C^\infty$~regularity, and we do not know whether it could be used for obtaining an extension with the same properties as in section~\ref{sec-extension1}. Further comments on these geometric questions would go beyond the scope of the present article.

\end{itemize}

\appendix

\section{Appendix} \label{sec-appendix}

\subsection{Proof of Proposition~\ref{prop-extension}} \label{appendix1}

\paragraph{Harmonic extension of $X$.}
Let us first discuss of how to extend $X$ from the circle $\Gamma_s$ into the unit ball $\Omega_s^-$. In dimension~2, one way of extending diffeomorphisms of the circle is to consider the Douady-Earle extension~\cite{Douady86}, which is harmonic, and therefore inherits of the elliptic regularity from its Dirichlet boundary condition. It relies on the Rad\'o-Kneser-Choquet theorem, and more specifically the Poisson integral formula. Many generalizations of such a result have been obtained afterwards, in particular requiring the strong Choquet condition, namely that $X(\Gamma_s)$ shall be convex. But more recently it has been extended in~\cite{Alessandrini2009, Alessandrini2017} to the case of homeomorphisms from the unit circle onto a simple closed curve of $\R^2$. We state~\cite[Theorem~1.3]{Alessandrini2009} in our context as follows:

\begin{lemma} \label{lemma-RKC}
Let $X:\Gamma_s \rightarrow X(\Gamma_s)$ be an orientation preserving difeomorphism of class~$C^1$ onto a simple closed curve $X(\Gamma_s)$. Let $D$ be the bounded domain such that $\p D = X(\Gamma_s)$. Denote by $X_{\mathrm{RKC}}$ the solution of the Dirichlet problem
\begin{equation}
\left\{ \begin{array} {rcl}
\Delta X_{\mathrm{RKC}} = 0 & & \text{in } \Omega_s^-, \\
X_{\mathrm{RKC}} = X & & \text{on } \Gamma_s.
\end{array} \right.
\label{Diri-pb}
\end{equation}
When $X_{\mathrm{RKC}} \in C^1(\overline{\Omega_s^-})$, it is a difeomorphism of $\overline{\Omega_s^-}$ onto $\overline{D}$ if and only if $\det \nabla X_{\mathrm{RKC}} >0$ everywhere on $\Gamma_s$.
\end{lemma}
The condition $X_{\mathrm{RKC}} \in C^1(\overline{\Omega_s^-})$ is satisfied when the elliptic regularity of~\eqref{Diri-pb} provides a solution in a Sobolev space that is embedded in $C^1(\overline{\Omega_s^-})$, in our case~$\HH^{5/2}(\Omega_s^\pm)$. The condition $\det \nabla X_{\mathrm{RKC}} >0$ can be guaranteed by assuming the data $X$ close enough to the identity. The estimates provided in the next step shows this.

\paragraph{Elliptic regularity, Lipschitz estimates and invertibility.}
The interest of extending $X$ via the Dirichlet problem~\eqref{Diri-pb} lies in the linearity and simplicity of the latter. Therefore we derive straightforwardly the following result, leading to Proposition~\ref{prop-extension}:
\begin{proposition} \label{prop-RKC}
If $X \in \mathcal{Z}_{\infty}(\Gamma_s)$, the solutions~$\tilde{X}^\pm$ of
\begin{equation}
\left\{ \begin{array} {rcl}
\Delta \tilde{X}^\pm = 0 & & \text{in } \Omega_s^\pm \times (0,\infty), \\
\tilde{X}^\pm = X & & \text{on } \Gamma_s \times (0,\infty), \\
\tilde{X}^+ = \Id & & \text{on } \p \Omega \times (0, \infty),
\end{array} \right.
\label{Diri-pb2}
\end{equation}
satisfy
\begin{equation}
\|\tilde{X}^\pm - \Id \|_{\mathcal{X}_{\infty}(\Omega_s^\pm)} \leq
C\|X- \Id \|_{\mathcal{Z}_{\infty}(\Gamma_s)}.
\label{est-Diri1}
\end{equation}
If $\|X-\Id \|_{\mathcal{Z}_{\infty}(\Gamma_s)}$ is small enough, the extension~$\tilde{X}^+$ is locally invertible, and $\tilde{X}^-$ is globally invertible. Furthermore, given two mappings $X_1-\Id, \, X_2-\Id \in \mathcal{Z}_{\infty}(\Gamma_s)$, the respective solutions $\tilde{X}^\pm_1, \, \tilde{X}^\pm_2$ of~\eqref{Diri-pb2} satisfy
\begin{equation}
\|\tilde{X}^\pm_1 - \tilde{X}^\pm_2 \|_{\mathcal{X}_{\infty}(\Omega_s^\pm)} \leq
C\|X_1 - X_2 \|_{\mathcal{Z}_{\infty}(\Gamma_s)}.
\label{est-Diri2}
\end{equation}
\end{proposition}
Without ambiguity we omit the notation~$\tilde{X}^\pm$ for keeping only~$\tilde{X}$.

\begin{proof}
Estimates~\eqref{est-Diri1} and~\eqref{est-Diri2} are straightforwardly deduced from the elliptic regularity of system~\eqref{Diri-pb2} in $\HH^{5/2}(\Omega_s^\pm)$. Next, recalling that the differential of $A \mapsto \det A$ is $H\mapsto \cof(A):H$, and that in dimension~2 the mapping $A\mapsto \cof(A)$ is linear, for all $t\geq 0$ we deduce from the mean value theorem
\begin{equation*}
\begin{array}{rcl}
|\det \nabla \tilde{X}(y,t) - 1 |_{\R} & \leq & \displaystyle 
\sup_{\alpha\in [0,1]} |\cof (\alpha y + (1-\alpha)\tilde{X}(y,t)) |_{\R^{2\times 2}}
| \tilde{X}(y,t) - y |_{\R^{2\times 2}}, \\
\| \det \nabla \tilde{X}(\cdot,t) -1 \|_{\mathcal{C}(\overline{\Omega_s^\pm})} & \leq &
C\left(1 + \|\tilde{X}(\cdot,t) \|_{\LL^{\infty}(\Omega_s^\pm)} \right)\| X-\Id \|_{\mathcal{Z}_{\infty}(\Gamma_s)}.
\end{array}
\end{equation*}
Therefore, assuming $\|X-\Id \|_{\mathcal{Z}_{\infty}(\Gamma_s)}$ small enough shows that $\det \nabla \tilde{X}(y,t) > C>0$ for all $(y,t) \in \Omega_s \times (0,\infty)$. Then the local invertibility of~$\tilde{X}_{|\Omega_s^+}$ is due to the inverse function theorem, and the global invertibility of~$\tilde{X}_{|\Omega_s^-}$ follows from Lemma~\ref{lemma-RKC}.
\end{proof}

\subsection{Proof of Proposition~\ref{propinfsup} and Corollary~\ref{coroinfsup}} \label{appendix2}
We start by recalling the two following lemmas that are needed for what follows. The first one can be deduced from~\cite[Exercise~3.4, Chapter~III]{Galdi}.
\begin{lemma} \label{lemma-app1}
For $p^{\pm} \in \L^2(\Omega_s^{\pm})$, there exists $v_p^{\pm} \in \HH^1_0(\Omega_s^{\pm})$ satisfying
\begin{equation}
-\divg v_p^{\pm} = p^{\pm} \quad \text{in } \Omega_s^{\pm}, \qquad
v_p^{\pm} = 0 \quad \text{on } \p \Omega_s^{\pm},
\label{sysdivGaldi}
\end{equation}
and
\begin{equation}
\| v_p^{\pm} \|_{\HH^1(\Omega_s^{\pm})}  \leq  C \|p^{\pm}\|_{\L^2(\Omega_s^{\pm})}.
\label{estdivGaldi}
\end{equation}
\end{lemma}
\noindent The second lemma is given in~\cite[Theorem~1.1, Chapter~IV]{Galdi}.
\begin{lemma} \label{lemma-app2}
For $h^{\pm} \in \mathbf{W}$, there exists a unique solution $(v_h^\pm,q_h^\pm)$ in $\mathbf{V}^\pm \times Q^\pm$ to the Stokes problem
\begin{eqnarray}
\left\{ \begin{array} {rcl}
-\divg \sigma(v_h^\pm,q_h^\pm) = 0 \quad \text{ and } \quad 
\divg v_h^\pm = 0 & & \text{in } \Omega_s^{\pm}, \\
v_h^+  = 0 &  & \text{on } \p \Omega, \\
v_h^{\pm} = h^{\pm} & & \text{on } \Gamma_s.
\end{array} \right.
\label{sysstokesNHDBC}
\end{eqnarray}
Moreover, it satisfies
\begin{eqnarray}
\| v_h^{\pm} \|_{\mathbf{H}^1(\Omega_s^{\pm})} + \| q_h^{\pm} \|_{\L^2(\Omega_s^{\pm})} & \leq & C \|h^{\pm}\|_{\mathbf{H}^{1/2}(\Gamma_s)}. \label{eststokesNHDBC}
\end{eqnarray}
\end{lemma}
\noindent We are now in position to prove Proposition~\ref{propinfsup} and Corollary~\ref{coroinfsup}.

\subsubsection*{Proof of Proposition~\ref{propinfsup}}
We adopt the method used for proving~\cite[Lemma~6, p.~144]{Stenberg}. Recall the definition of the bilinear form
\begin{equation*}
\begin{array} {rcl}
\mathcal{M}(\mathfrak{u},\mathfrak{v}) &  =  & 
2\nu \langle \varepsilon(u^+),\varepsilon(v^+)\rangle_{\mathbb{L}^2(\Omega_s^+)} 
+ 2\nu \langle \varepsilon(u^-),\varepsilon(v^-)\rangle_{\mathbb{L}^2(\Omega_s^-)} \\
& & - \langle p^+,\divg v^+ \rangle_{\L^2(\Omega_s^+)} 
- \langle q^+\divg u^+ \rangle_{\L^2(\Omega_s^+)} 
- \langle  p^-,\divg v^- \rangle_{\L^2(\Omega_s^-)} 
-\langle q^-,\divg u^-\rangle_{\L^2(\Omega_s^-)} \\
& & - \langle \lambda^+ , v^+-\varphi\rangle_{\mathbf{W}';\mathbf{W}} - 
\langle \mu^+, u^+ - \phi\rangle_{\mathbf{W}';\mathbf{W}}  - \langle \lambda^- , v^--\varphi\rangle_{\mathbf{W}';\mathbf{W}} - \langle \mu^-, u^- - \phi \rangle_{\mathbf{W}';\mathbf{W}},
\end{array}
\end{equation*}
with the notation $ \mathfrak{u} = (u^+,p^+,u^-,p^-,\lambda^+,\lambda^-,\phi)$ and $\mathfrak{v} = (v^+,q^+,v^-,q^-,\mu^+,\mu^-,\varphi) \in \mathfrak{V}$.\\

\noindent {\bf Step 1.} 
Choose $\mathfrak{v}_1 = (u^+,-p^+,u^-,-p^-,-\lambda^+,-\lambda^-,\phi)$. Then
\begin{eqnarray}
\mathcal{M}(\mathfrak{u},\mathfrak{v}_1) & = & 2\nu \left(\| \varepsilon(u^+) \|_{\mathbb{L}^2(\Omega_s^+)}^2 +   \| \varepsilon(u^-) \|_{\mathbb{L}^2(\Omega_s^-)}^2\right). \label{eststep1}
\end{eqnarray}

\noindent {\bf Step 2.} Choose $\mathfrak{v}_2 = (v^+_p,0,v^-_p,0,0,0,0)$, where $v_p^{\pm} \in \mathbf{H}^1_0(\Omega_s^{\pm})$ is the solution of system~\eqref{sysdivGaldi} corresponding to~$p^\pm$, satisfying~\eqref{estdivGaldi}. Then, using successively the Cauchy-Schwarz and the Young's inequalities for any $\alpha >0$, we obtain
\begin{eqnarray*}
\mathcal{M}(\mathfrak{u},\mathfrak{v}_2) & = & 
2\nu \langle\varepsilon(u^+),\varepsilon(v_p^+)\rangle_{\mathbb{L}^2(\Omega_s^+)} 
+ 2\nu \langle\varepsilon(u^-),\varepsilon(v_p^-)\rangle_{\mathbb{L}^2(\Omega_s^-)}
+ \| p^+\|^2_{\L^2(\Omega_s^+)} +  \| p^-\|^2_{\L^2(\Omega_s^-)} \\
& \geq & \| p^+\|^2_{\L^2(\Omega_s^+)} +  \| p^-\|^2_{\L^2(\Omega_s^-)} 
- \alpha \nu \| \varepsilon(u^+) \|_{\mathbb{L}^2(\Omega_s^+)}^2 
-  \alpha \nu \| \varepsilon(u^-) \|_{\mathbb{L}^2(\Omega_s^-)}^2 \\
& &  - \frac{\nu}{\alpha}\| \varepsilon(v_p^+) \|_{\mathbb{L}^2(\Omega_s^+)}^2 - \frac{\nu}{\alpha}\| \varepsilon(v_p^-) \|_{\mathbb{L}^2(\Omega_s^-)}^2.
\end{eqnarray*}
Furthermore, combining the Korn's inequality for $v_p^{\pm}$ and the estimate~\eqref{estdivGaldi}, we deduce
\begin{eqnarray}
\mathcal{M}(\mathfrak{u},\mathfrak{v}_2) & \geq & \left(1 - \frac{C\nu}{\alpha} \right)\left( \| p^+\|^2_{\L^2(\Omega_s^+)} +  \| p^-\|^2_{\L^2(\Omega_s^-)} \right) 
- \alpha \nu \left(\| \varepsilon(u^+) \|_{\mathbb{L}^2(\Omega_s^+)}^2 
+ \| \varepsilon(u^-) \|_{\mathbb{L}^2(\Omega_s^-)}^2 \right), \label{eststep2}
\end{eqnarray}
where the constant $C>0$ is independent of $\alpha>0$.\\

\noindent {\bf Step 3.} 
Choose $\mathfrak{v}_3 = (v_{h_n}^+,q_{h_n}^+,v_{h_n}^-,q_{h_n}^-,0,0,0)$, where $(v_{h_n}^{\pm},q_{h_n}^{\pm})$ is the solution of the Stokes system~\eqref{sysstokesNHDBC} with ${v^{\pm}_{h_n}}_{|\Gamma} = \|\lambda^\pm\|_{\mathbf{W}'} h_n^\pm$ as data, where the sequences $(h_n^\pm)_n$ are such that $\|h_n^\pm\|_{\mathbf{W}} = 1$ and $-\langle \lambda^\pm , h_n^\pm \rangle_{\mathbf{W}';\mathbf{W}} \rightarrow \|\lambda^\pm\|_{\mathbf{W}'}$. Then for some $\beta >0$, the same combination of the Cauchy-Schwarz and Young's inequalities yields
\begin{eqnarray*}
\mathcal{M}(\mathfrak{u},\mathfrak{v}_3) & = & 
2\nu \langle\varepsilon(u^+),\varepsilon(v_{h_n}^+)\rangle_{\mathbb{L}^2(\Omega_s^+)}
+ 2\nu \langle\varepsilon(u^-),\varepsilon(v_{h_n}^-)\rangle_{\mathbb{L}^2(\Omega_s^-)} \\
& & -\langle q_{h_n}^+,\divg u^+ \rangle_{\L^2(\Omega_s^-)} 
-\langle q_{h_n}^-,\divg u^- \rangle_{\L^2(\Omega_s^-)}  
- \langle \lambda^+,h_n^+ \rangle_{\mathbf{W}',\mathbf{W}}
- \langle \lambda^-,h_n^- \rangle_{\mathbf{W}',\mathbf{W}} \\
& \geq & 
- \langle \lambda^+,h_n^+ \rangle_{\mathbf{W}',\mathbf{W}}
- \langle \lambda^-,h_n^- \rangle_{\mathbf{W}',\mathbf{W}} \\
& & -\frac{1}{\beta}\left( \nu\| \varepsilon(v_{h_n}^+) \|_{\mathbb{L}^2(\Omega_s^+)}^2 + \|q^+_{h_n}\|^2_{\L^2(\Omega_s^+)} + \nu\| \varepsilon(v_{h_n}^-) \|_{\mathbb{L}^2(\Omega_s^-)}^2 +  \|q^-_{h_n}\|^2_{\L^2(\Omega_s^-)} \right) \\
& & - \beta \left(\nu\| \varepsilon(u^+) \|_{\mathbb{L}^2(\Omega_s^+)}^2 + \|\divg u^+\|^2_{\L^2(\Omega_s^+)} + \nu\| \varepsilon(u^-) \|_{\mathbb{L}^2(\Omega_s^-)}^2 +  \|\divg u^-\|^2_{\L^2(\Omega_s^-)} \right).
\end{eqnarray*}
Estimate~\eqref{eststokesNHDBC} yields $\|v_{h_n} \|_{\HH^1(\Omega_s^\pm)} + \|q^\pm_{{h_n}}\|_{\L^2(\Omega_s^\pm)} \leq C\|\lambda^\pm\|_{\mathbf{W}'}$, and by passing to the limit we deduce
\begin{eqnarray}
\mathcal{M}(\mathfrak{u},\mathfrak{v}_3) & \geq & \left(1-\frac{C}{\beta} \right)\left(\| \lambda^+ \|_{\mathbf{W}'}^2 + \| \lambda^- \|_{\mathbf{W}'}^2 \right) - C\beta\nu \left(\| \varepsilon(u^+) \|_{\mathbb{L}^2(\Omega^+)}^2 + \| \varepsilon(u^-) \|_{\mathbb{L}^2(\Omega^-)}^2 \right), \label{eststep3}
\end{eqnarray}
where here again the generic constant $C>0$ is independent of $\beta>0$.\\

\noindent {\bf Step 4.} Choose $\mathfrak{v}_4 = (0,0,0,0,\phi,\phi,0)$. With the Young's inequalities we estimate
\begin{eqnarray}
\mathcal{M}(\mathfrak{u},\mathfrak{v}_4) & = & 2\| \phi \|^2_{\mathbf{W}} - \langle \phi , u^+ \rangle_{\mathbf{W}';\mathbf{W}} - \langle \phi , u^- \rangle_{\mathbf{W}';\mathbf{W}} \geq
2\| \phi \|^2_{\mathbf{W}} -
\| \phi\|_{\mathbf{W}'} \|u^+\|_{\mathbf{W}}
- \| \phi\|_{\mathbf{W}'} \|u^-\|_{\mathbf{W}}
, \nonumber \\
\mathcal{M}(\mathfrak{u},\mathfrak{v}_4) & \geq & \| \phi \|^2_{\mathbf{W}}- \frac{1}{2} \left(\|u^+\|^2_{\mathbf{W}} + \|u^-\|^2_{\mathbf{W}} \right), \nonumber \\
\mathcal{M}(\mathfrak{u},\mathfrak{v}_4) & \geq & \| \phi \|^2_{\mathbf{W}}- C \left(\|u^+\|^2_{\mathbf{V}} + \|u^-\|^2_{\mathbf{V}} \right), \label{eststep4}
\end{eqnarray}
where $C>0$ is deduced from the constant of the trace operators and those of the Korn's inequality.

\hfill \\
{\bf Step 5.} Choose $\mathfrak{v} = \mathfrak{v}_1 + \gamma_2 \mathfrak{v}_2 + \gamma_3\mathfrak{v}_3 + \gamma_4 \mathfrak{v}_4$, for some positive constants $\gamma_2$, $\gamma_3$ and $\gamma_4$. Then, the estimates~\eqref{eststep1}--\eqref{eststep4} yields
\begin{eqnarray*}
\mathcal{M}(\mathfrak{u},\mathfrak{v}) & \geq & \left(2\nu - \alpha \nu \gamma_2 - \beta \nu \gamma_3 - C\gamma_4 \right)\left( \| \varepsilon(u^+) \|_{\mathbb{L}^2(\Omega^+)}^2 +   \| \varepsilon(u^-) \|_{\mathbb{L}^2(\Omega^-)}^2 \right) \\
& & +\gamma_2\left(1-\frac{C\nu}{\alpha}\right)\left( \| p^+\|^2_{\L^2(\Omega^+)} +  \| p^-\|^2_{\L^2(\Omega^-)} \right) + \gamma_3\left(1-\frac{C}{\beta} \right)\left(\| \lambda^+ \|_{\mathbf{W}'}^2 + \| \lambda^- \|_{\mathbf{W}'}^2 \right)
+ \gamma_4\| \phi \|^2_{\mathbf{W}}.
\end{eqnarray*}
By choosing $\alpha$ and $\beta$ large enough ($\alpha > C\nu$ and $\beta > C$), and next $\gamma_2$, $\gamma_3$ and $\gamma_4$ small enough, we get
\begin{equation}  \label{final-est}
\mathcal{M}(\mathfrak{u},\mathfrak{v})  \geq  \delta\| \mathfrak{u} \|^2,
\end{equation}
for some constant $\delta>0$, after having used the Korn's inequality of Lemma~\ref{lemma-Korn}. It remains us to verify that the norm of $\mathfrak{v}$ so chosen is controlled by the norm of $\mathfrak{u}$, namely the estimate $
\| \mathfrak{v}\|  \leq  C \| \mathfrak{u}\|$,
which holds from~\eqref{estdivGaldi} and~\eqref{eststokesNHDBC}. Thus we obtain $
\displaystyle \frac{\mathcal{M}(\mathfrak{u},\mathfrak{v})}{\| \mathfrak{v} \|}  \geq  C\| \mathfrak{u} \|$, which enables us to complete the proof.

\subsubsection*{Proof of Corollary~\ref{coroinfsup}}
The result is a direct consequence of the Banach-Ne\v{c}as-Babu\v{s}ka theorem (see~\cite[Theorem~2.6, page~85]{Ern}). The continuity of the linear mapping $\mathcal{G}$ is obvious. Given the inf-sup condition of Proposition~\ref{propinfsup}, it remains us to verify the injectivity property for the bilinear form. Let $\mathfrak{u} \in \mathfrak{V}$ be such that $\mathcal{M}(\mathfrak{u},\mathfrak{v}) = 0$ for all $\mathfrak{v} \in \mathfrak{V}$. It is sufficient to choose~$\mathfrak{v}$ like in the last step of the proof of Proposition~\ref{propinfsup}, so that~\eqref{final-est} holds, and implies $\mathfrak{u} = 0$, which concludes the proof.

\printbibliography

\end{document}